\documentclass[english]{article}
\usepackage[T1]{fontenc}
\usepackage[latin9]{inputenc}
\usepackage{amsmath}
\usepackage{amsthm}
\usepackage{amssymb}
\usepackage{xargs}[2008/03/08]
\usepackage[all]{xy}

\makeatletter
\theoremstyle{plain}
\newtheorem{thm}{\protect\theoremname}[section]
\theoremstyle{remark}
\newtheorem{rem}[thm]{\protect\remarkname}
\theoremstyle{plain}
\newtheorem{prop}[thm]{\protect\propositionname}
\theoremstyle{definition}
\newtheorem{defn}[thm]{\protect\definitionname}
\theoremstyle{definition}
\newtheorem{example}[thm]{\protect\examplename}
\theoremstyle{plain}
\newtheorem{lem}[thm]{\protect\lemmaname}
\theoremstyle{plain}
\newtheorem{cor}[thm]{\protect\corollaryname}

\@ifundefined{date}{}{\date{}}

\usepackage{pdflscape}

\usepackage{tikz}

\usepackage{pgfplots}

\usetikzlibrary{arrows}

\usetikzlibrary{automata}

\usetikzlibrary{chains}

\usetikzlibrary{positioning}

\usetikzlibrary{backgrounds}

\usetikzlibrary{calc,decorations.pathreplacing}

\pgfplotsset{compat=1.9}

\g@addto@macro\@floatboxreset{\centering}

\usepackage{enumitem}
\setenumerate{ label= (\alph*)}
\setenumerate[2]{ label= (\alph{enumi}\arabic*)} 

\newcommand{\X}{\mathcal{X}}
\newcommand{\Y}{\mathcal{Y}}
\newcommand{\Z}{\mathcal{Z}}
\newcommand{\T}{\mathcal{T}}
\newcommand{\F}{\mathcal{F}}
\newcommand{\U}{\mathcal{U}}

\newcommand{\W}{\mathcal{W}}
\newcommand{\D}{\mathcal{D}}
\newcommand{\Hcal}{\mathcal{H}}
\newcommand{\Le}{\mathsf{L}} 
\newcommand{\R}{\mathsf{R}} 
\newcommand{\Et}{\mathsf{E}} 
\newcommand{\Pres}{\mathrm{Pres}}
\newcommand{\Gen}{\mathrm{Gen}}
\newcommand{\Csf}{\mathsf{C}} 
\newcommand{\Hsf}{\mathsf{H}}

\usepackage{babel}
\providecommand{\definitionname}{Definition}
\providecommand{\examplename}{Example}
\providecommand{\lemmaname}{Lemma}

\providecommand{\propositionname}{Proposition}
\providecommand{\remarkname}{Remark}
\providecommand{\corollaryname}{Corollary}
\providecommand{\theoremname}{Theorem}

\makeatother

\usepackage{babel}
\providecommand{\corollaryname}{Corollary}
\providecommand{\definitionname}{Definition}
\providecommand{\examplename}{Example}
\providecommand{\lemmaname}{Lemma}
\providecommand{\propositionname}{Proposition}
\providecommand{\remarkname}{Remark}
\providecommand{\theoremname}{Theorem}

\begin{document}
\title{Tilting objects in the extended heart of a $t$-structure}
\author{Alejandro Argud{\'i}n-Monroy, Octavio Mendoza, Carlos E. Parra \thanks{The first named author was supported by a postdoctoral fellowship
EPM(1) 2024 from SECIHTI. The first and second named authors were
supported by the Project PAPIIT-IN100124 Universidad Nacional Aut{\'o}noma
de M{\'e}xico. The third named author was supported by ANID+FONDECYT/REGULAR+1240253.}}

\maketitle
\newcommandx\Mod[1][usedefault, addprefix=\global, 1=R]{\operatorname{Mod}\left(#1\right)}%
\newcommandx\modd[1][usedefault, addprefix=\global, 1=R]{\operatorname{mod}\left(#1\right)}%
\newcommandx\Ker[1][usedefault, addprefix=\global, 1=M]{\operatorname{Ker}#1}%
\newcommandx\Ann[1][usedefault, addprefix=\global, 1=M]{\operatorname{Ann}#1}%
\newcommandx\im[1][usedefault, addprefix=\global, 1=M]{\operatorname{Im}\left(#1\right)}%
\newcommandx\coim[1][usedefault, addprefix=\global, 1=M]{\operatorname{CoIm}\left(#1\right)}%
\newcommandx\Cok[1][usedefault, addprefix=\global, 1=M]{\operatorname{Coker}\left(#1\right)}%
\newcommandx\pdr[2][usedefault, addprefix=\global, 1=M, 2=\mathcalA]{\mathrm{pd}{}_{#2}\left(#1\right)}%
\newcommandx\idr[2][usedefault, addprefix=\global, 1=M, 2=\mathcalB]{\mathrm{id}{}_{#2}\left(#1\right)}%
\newcommandx\smd[1][usedefault, addprefix=\global, 1=\mathcalM]{\operatorname{smd}\left(#1\right)}%
\newcommandx\suc[5][usedefault, addprefix=\global, 1=N, 2=M, 3=K, 4=, 5=]{#1\overset{#4}{\rightarrow}#2\overset{#5}{\rightarrow}#3}%
\newcommandx\Ext[4][usedefault, addprefix=\global, 1=M, 2=, 3=, 4=X]{\mathbb{E}{}_{#3}^{#2}\left(#1,#4\right)}%
\newcommandx\p[2][usedefault, addprefix=\global, 1=\mathcal{A}, 2=\mathcal{B}]{\left(#1,#2\right)}%
\newcommandx\Hom[3][usedefault, addprefix=\global, 1=\mathcal{A}, 2=M, 3=N]{\mathrm{Hom}{}_{#1}(#2,#3)}%
\renewcommandx\Hom[3][usedefault, addprefix=\global, 1=\mathcal{D}, 2=M, 3=N]{\mathrm{Hom}{}_{#1}(#2,#3)}%
\newcommandx\End[2][usedefault, addprefix=\global, 1=R, 2=M]{\mathrm{End}{}_{#1}(#2)}%
\newcommandx\proj[1][usedefault, addprefix=\global, 1=R]{\operatorname{proj}\left(#1\right)}%
\newcommandx\Inj[1][usedefault, addprefix=\global, 1=R]{\operatorname{Inj}\left(#1\right)}%
\newcommandx\inj[1][usedefault, addprefix=\global, 1=R]{\operatorname{inj}\left(#1\right)}%
\newcommandx\Proj[1][usedefault, addprefix=\global, 1=R]{\operatorname{Proj}\left(#1\right)}%
\newcommandx\Homk[4][usedefault, addprefix=\global, 1=\mathcalK(R), 2=\sigma, 3=\omega, 4=1]{\mathrm{Hom}{}_{#1}(#2,#3[#4])}%
\global\long\def\Add{\operatorname{Add}}%
\global\long\def\Gen{\operatorname{Gen}}%
\global\long\def\Pres{\operatorname{Pres}}%
\newcommandx\add[1][usedefault, addprefix=\global, 1=\mathcal{M}]{\operatorname{add}\left(#1\right)}%
\newcommandx\colim[1][usedefault, addprefix=\global, 1=M]{\mathsf{colim}_{\Sigma}\left(#1\right)}%
\global\long\def\Fun{\operatorname{Fun}}%
\global\long\def\A{\mathcal{A}}%
\global\long\def\B{\mathcal{B}}%
\global\long\def\col{\mathsf{colim}}%
\global\long\def\limite{\mathsf{lim}}%
\global\long\def\stors{\mathsf{stors}}%
\global\long\def\Extu{\mathsf{Ext.u.}}%
\global\long\def\C{\mathcal{C}}%
\global\long\def\s{\mathfrak{s}}%
\global\long\def\E{\mathbb{E}}%
\global\long\def\t{\mathbf{t}}%
\global\long\def\u{\mathbf{u}}%
\global\long\def\v{\mathbf{v}}%
\global\long\def\pd{\mathsf{pd}}%
\global\long\def\id{\mathsf{id}}%

\begin{abstract}

Building on the recent work of Adachi, Enomoto and Tsukamoto on a
generalization of the Happel-Reiten-Smal{\o} tilting process, we
study extended tilting objects in extriangulated categories with negative
first extension. These objects coincide with the 1-tilting objects
in abelian categories as in the work of Parra, Saor{\'i}n and Virili.
We will be particularly interested in the case where the extriangulated
category in question is the heart $\mathcal{H}_{[\t_{1},\t_{2}]}$
of an interval of $t$-structures $[\t_{1},\t_{2}]$. Our main results
consist of a characterization of the extended tilting objects of a
heart $\mathcal{H}_{[\t_{1},\t_{2}]}$ for the case when $\text{\ensuremath{\t}}_{2}\leq\Sigma^{-1}\t_{1}$,
and another one for the case when  
$\Sigma^{-2}\t_{1}\leq\t_{2}$. In
the first one, we give conditions for these objects to coincide with
the quasi-tilting objects of the abelian category $\mathcal{H}_{[\t_{1},\Sigma^{-1}\t_{1}]}$.
 In the second one,  it is proved that the heart $\mathcal{H}_{[\t_{1},\t_{2}]}$ admits an extended tilting object
only if it is an extended heart (i.e. $\t_2 =\Sigma ^{-2} \t _1$). 
Furthermore, we will characterize the extended tilting objects of an extended heart 
$\mathcal{H}_{[\t_{1},\t_{2}]}$  as projective generators in the abelian category
  $\mathcal{H}_{[\t_{1},\Sigma\t_{2}]}$ whenever  $\mathcal{H}_{[\t_{1},\Sigma\t_{2}]}$ is a cogenerating class of the extended heart.
  

\end{abstract}

\section{Introduction}

One of the fundamental tools to study a triangulated category is the
notion of $t$-structure \cite{BBD}. Among their properties, we can
mention that any abelian category can be realized as the heart $\mathcal{H}_{\mathbf{x}}:=\mathcal{X}\cap\mathcal{Y}$
of a $t$-structure $\mathbf{x}=(\mathcal{X},\mathcal{Y})$ in a triangulated
category. The Happel-Reiten-Smal{\o} tilting (\emph{right HRS-tilt})
process shows us how, starting from a $t$-structure $\mathbf{x}=(\mathcal{X},\mathcal{Y})$,
the torsion pairs of the heart $\mathcal{H}_{\mathbf{x}}$ parameterize
a certain family of $t$-structures. Namely, they parameterize the
$t$-structures $\mathbf{x}'=(\mathcal{X}',\mathcal{Y}')$ such that
$\Sigma\mathcal{X}\subseteq\mathcal{X}'\subseteq\mathcal{X}$ \cite{HRS}.
Moreover, within the parameterized $t$-structures, we can highlight
those whose heart is derived equivalent to $\mathcal{H}_{\mathbf{x}}$.
For example, the ones corresponding to the so-called \emph{tilting
torsion pairs} satisfy this property. It is of particular interest
when such torsion pair can be chosen so that its heart has a projective
generator. In which case, one has that the torsion class is of the
form $\Gen(V)$, where $V$ is what is known as a \emph{quasi-tilting
object} \cite[Proposition 3.8]{PS3}.

Recently, Adachi, Enomoto and Tsukamoto presented in \cite{AET} a
generalization of the HRS-tilt process in the context of extriangulated
categories with negative first extension, which can be referred as
the \emph{AET-tilt process}. To state the AET-tilt process in the
context of a triangulated category $\D,$ the following notions are
presented. Given two $t$-structures $\t_{1}=(\mathcal{X}_{1},\mathcal{Y}_{1})$
and $\t_{2}=(\mathcal{X}_{2},\mathcal{Y}_{2})$ in $\D$ such that
$\t_{1}\leq\t_{2}$ (that is $\mathcal{X}_{1}\subseteq\mathcal{X}_{2}$),
define the interval $[\t_{1},\t_{2}]$ which is the class of all the
$t$-structures $\t'=(\mathcal{X}',\mathcal{Y}')$ such that $\mathcal{X}_{1}\subseteq\mathcal{X}'\subseteq\mathcal{X}_{2}.$
The class $\mathcal{H}_{[\t_{1},\t_{2}]}:=\Sigma^{-1}\mathcal{Y}_{1}\cap\mathcal{X}_{2}$
is known as the heart of the interval $[\t_{1},\t_{2}].$ The AET-tilt
process tells us that there exists a bijection between the $t$-structures
in $[\t_{1},\t_{2}]$ and the $s$-torsion pairs in the extriangulated
category $\mathcal{H}_{[\t_{1},\t_{2}]}$.

The aim of this paper is to study the structure of the heart $\mathcal{H}_{[\t_{1},\t_{2}]}$
as well as its tilting objects. To this end, 
we will be considering three different kinds of hearts $\mathcal{H}_{[\t_{1},\t_{2}]}$. 
The first kind (\emph{standard hearts})
consists of hearts $\mathcal{H}_{[\t_{1},\t_{2}]}$ where $\t_{1}=\Sigma\t_{2}$
(that is $\mathcal{H}_{[\t_{1},\t_{2}]}=\Hcal_{\t_{2}}$ which is
an abelian category). The second kind (\emph{restricted hearts}) consists
of hearts $\mathcal{H}_{[\t_{1},\t_{2}]}$ where $\Sigma\t_{2}<\t_{1}<\t_{2}$.
Finally, the third kind (\emph{large hearts}) consists of hearts $\mathcal{H}_{[\t_{1},\t_{2}]}$
where $\t_{1}<\Sigma\t_{2}<\t_{2}$. 
 A particular case of this last kind is when $\t_2 = \Sigma ^{-2}\t_1$. In this situation, we say that $\mathcal{H}_{[\t_{1},\t_{2}]}$ is an \emph{extended heart}.

Let us look with more precision at the content of the paper. In Section
2, we will look at some basic notions from the theory of extriangulated
and triangulated categories. Specifically, it is recalled the notion
of $s$-torsion pair in an extriangulated category (with negative
first extension) and its relationship with $t$-structures (in case
the extriangulated category in question is a triangulated one). Finally,
it is stated the theorem of Adachi, Enomoto and Tsukamoto on the AET-tilt
process.

The third section contains a deeper study on the structure of the
heart $\mathcal{H}_{[\t_{1},\t_{2}]}$ for an interval $[\t_{1},\t_{2}]$
of $s$-torsion pairs in an extriangulated category. On the one hand,
we will seek to study the $s$-torsion pairs of $\mathcal{H}_{[\t_{1},\t_{2}]}$
as well as its closure properties. Specifically, we will see in Propositions
\ref{lem:cerraduras stors} and \ref{lemaA} conditions for a $s$-torsion
pair to be closed under extensions, (co)cones, direct summands, $s$-subobjects
and $s$-quotients. On the other hand, we introduce the notion of
\emph{normal interval} (see Definition \ref{nst-int}). We will see
that, under certain conditions (see Theorem \ref{semi-ab-triang}),
if the interval $[\t_{1},\t_{2}]$ is normal, then $\mathcal{H}_{[\t_{1},\t_{2}]}$
is an exact and a semi-abelian category. If in addition, see Theorem
\ref{triang-n-smab}, the ambient extriangulated category is triangulated,
we have that the heart $\mathcal{H}_{[\t_{1},\t_{2}]}$ is a quasi-abelian
category. It is worth mentioning that similar results to the above
ones have been proved (with different tools) by Rump, Schneiders,
Tattar and Fiorot \cite{R1,S,T,F}. By using the above, see Theorem
\ref{param-stp-qac}, we will make use of the AET-tilt process to
parameterize the torsion pairs in a quasi-abelian category.

In Section 4, we will apply the results obtained in Section 3 to study
the $t$-structures of a triangulated category together with the  extended hearts induced by 
 $t$-structures. In particular, we will review the right HRS-tilt
and the left HRS-tilt processes (see Theorems \ref{ext-tilt-HRS-p}
and \ref{ext-tilt-HRS-n}).

Lastly, in Section 5, we introduce the notion of \emph{extended tilting
}object in an extriangulated category with negative first extension.
This notion coincides with the definition of 1-tilting object in abelian
categories introduced recently in \cite{PSV2} (see Lemma \ref{lem:tiltexact}).
We will see that the definition of extended tilting object in this
more general context shares some properties with the abelian case.
Specifically, we will show that an extended tilting object is $\mathbb{E}$-universal
and that it also has projective dimension at most one under certain
conditions (which include the exact case and others), see Propositions
\ref{pdT1} and \ref{et-univ}. Having done so, we will seek to characterize
extended tilting objects in restricted hearts and large hearts. Namely,
we will show that an object is extended tilting in a restricted heart
$\mathcal{H}_{[\t_{1},\t_{2}]}$ if and only if it is quasi-tilting
in the standard heart $\mathcal{H}_{[\t_{1},\Sigma^{-1}\t_{1}]}$
(see Theorem \ref{thm:restrict tilt}). In contrast, given an interval
$[\t_{1},\t_{2}]$ with $\Sigma^{-2}\t_{1}\leq\t_{2}$, an object
$V$ will be extended tilting in $\mathcal{H}_{[\t_{1},\t_{2}]}$
if and only if $\mathcal{H}_{[\t_{1},\Sigma\t_{2}]}$ is cogenerating
in $\mathcal{H}_{[\t_{1},\t_{2}]}$ and $V$ is a projective generator
in $\mathcal{H}_{[\t_{1},\Sigma\t_{2}]}$. Moreover, we will see that
in this case $\t_{2}=\Sigma^{-2}\t_{1}$ (see Theorem \ref{teo:aplication-main-1}).
Using the above, Corollary \ref{cor:main} gives a recipe of how to
obtain extended tilting objects in the non-abelian and non-triangulated
context. Furthermore, Examples \ref{ejemplo} and \ref{exa:ult} tell
us that the family of tilting objects obtained in this way is not
empty.

\section{Preliminaries}

Exact categories and triangulated categories are valuable tools of
contemporary mathematics. Their main quality lies in the fact that
they are suitable contexts for homological algebra. Recently, H. Nakaoka
and Y. Palu presented in \cite{NP} the notion of extriangulated category:
a concept that encompasses triangulated categories and exact categories.
We will omit the precise definition of extriangulated category, see
\cite[Def. 2.12]{NP} for details, but we include below some essential
notions for the convenience of the reader. To get an intuitive picture,
the inexperienced reader may understand extriangulated categories
as subcategories of triangulated categories which are closed under
extensions. \
 An \textbf{extriangulated category} consists of a triple $(\mathcal{D},\mathbb{E},\mathfrak{s})$,
where $\mathcal{D}$ is an additive category, $\mathbb{E}:\mathcal{D}^{op}\times\mathcal{\mathcal{D}}\rightarrow\operatorname{Ab}$
is an additive bifunctor and $\mathfrak{s}$ is a correspondence that
associates an equivalence class $[A\overset{x}{\rightarrow}X\overset{y}{\rightarrow}B]$
to each element $\delta$ of $\mathbb{E}(B,A)$ satisfying a series
of axioms as in \cite[Def. 2.12]{NP}. To fix ideas, the reader can
think of $\mathcal{D}$ as an abelian category and $\mathbb{E}=\mathbf{Ext}_{\mathcal{\mathcal{D}}}^{1}(-,-)$,
or of $\mathcal{D}$ as a triangulated category with $\mathbb{E}=\mathbf{Hom}_{\mathcal{D}}(-,-[1])$.
Moreover, whenever we refer to an abelian category or a triangulated
category as an extriangulated category, it will be by means of such
a functor $\mathbb{E}$. \
 Let $(\mathcal{D},\mathbb{E},\mathfrak{s})$ be an extriangulated
category. By following \cite{NP}, we have the following notions.
For $A,B\in\mathcal{D},$ an element $\delta\in\mathbb{E}(B,A)$ is
called an \textbf{$\mathbb{E}$-extension}, and the zero element $0\in\mathbb{E}(B,A)$
is called the \textbf{split $\mathbb{E}$-extension}. If $\delta\in\mathbb{E}(B,A)$
and $\s(\delta)=[\suc[A][C][B][x][y]]$, we say that $\suc[A][C][B][x][y]$
\textbf{realizes} $\delta$ or that it is an \textbf{$\mathfrak{s}$-conflation},
also known as an $\mathbb{E}$-triangle. In such case, we will use
the notation $\delta:\:\suc[A][C][B][x][y]$ or $A\xrightarrow{x}C\xrightarrow{y}B\xrightarrow{\delta}$.
For every $\delta\in\mathbb{E}(B,A),$ $a\in\Hom[\mathcal{D}][A][A']$
and $b\in\Hom[\mathcal{D}][B'][B],$ define $a\cdot\delta:=\mathbb{E}(B,a)(\delta)$
and $\delta\cdot b:=\mathbb{E}(b,A)(\delta)$. Observe that, each
$\delta\in\mathbb{E}(B,A)$ induces the morphisms $\delta\cdot-:\Hom[\mathcal{D}][W][B]\rightarrow\mathbb{E}(W,A)$
and $-\cdot\delta:\Hom[\mathcal{\mathcal{D}}][A][W]\rightarrow\mathbb{E}(B,W),$
for every $W\in\mathcal{D}.$ Let $\delta\in\mathbb{E}(B,A)$ and
$\delta'\in\mathbb{E}(B',A')$. A \textbf{morphism of extensions}
$\delta\rightarrow\delta'$ is a pair $(a,b)\in\Hom[\mathcal{D}][A][A']\times\Hom[\mathcal{D}][B][B']$
such that $a\cdot\delta=\delta'\cdot b.$ According to the definition
of extriangulated category, we have that, for any morphism of $\mathbb{E}$-extensions
$(a,b):\delta\rightarrow\delta'$, with $\s(\delta)=[\suc[A][C][B][x][y]]$
and $\s(\delta')=[\suc[A'][C'][B][x'][y']']$, there exists a morphism
$c:C\rightarrow C'$ such that $c\circ x=x'\circ a$ and $y'\circ c=b\circ y$.
In this case we say that $(a,c,b)$ \textbf{realizes} the morphism
$(a,b):\delta\rightarrow\delta'$. \
 Let $(\mathcal{D},\mathbb{E},\mathfrak{s})$ be an extriangulated
category. Given two classes of objects $\mathcal{X},\mathcal{Y}\subseteq\mathcal{D}$,
we define the class $\mathcal{X}\star\mathcal{Y}$ as the subcategory
of objects $C\in\mathcal{D}$ that admit an $\mathfrak{s}$-conflation
$\suc[X][C][Y]$ with $X\in\mathcal{X}$ and $Y\in\mathcal{Y}$. We
say that a class of objects $\mathcal{X}\subseteq\mathcal{D}$ is
\textbf{closed under extensions} if $\mathcal{X}\star\mathcal{X}\subseteq\mathcal{X}.$
A class $\Z\subseteq\D$ is \textbf{closed under cones} if for any
$\mathfrak{s}$-conflation $\suc[A][B][C]$ with $A,B\in\Z,$ we have
that $C\in\Z.$ Dually, $\Z$ is \textbf{closed under cocones} if
for any $\mathfrak{s}$-conflation $\suc[A][B][C]$ with $B,C\in\Z,$
we have that $A\in\Z.$ 
\begin{rem}
Let $(\mathcal{C},\mathbb{E},\mathfrak{s})$ be an extriangulated
category and $\mathcal{X},\mathcal{Y},\mathcal{Z}\subseteq\mathcal{\mathcal{D}}$
be classes of objects. 
\begin{enumerate}
\item It is a known fact that $\mathcal{X}\star(\mathcal{Y}\star\mathcal{Z})=(\mathcal{X}\star\mathcal{Y})\star\mathcal{Z}$
(see for example \cite[p.454]{AET}). 
\item If $\mathcal{X}$ is closed under extensions, then $\mathcal{X}$
becomes an extriangulated category by restricting $\mathbb{E}$ and
$\mathfrak{s}$ on $\X,$ see \cite[Rk. 2.18]{NP}. 
\item For the sake of simplicity, we write $\C$ instead of $(\mathcal{C},\mathbb{E},\mathfrak{s})$
for referring to the extriangulated category $(\mathcal{C},\mathbb{E},\mathfrak{s}).$ 
\end{enumerate}
\end{rem}

The following long exact sequences that come from an $\mathfrak{s}$-conflation
will be very useful throughout the paper. 
\begin{prop}
\label{suc-1-ex} \cite[Cor. 3.12]{NP} For any $\mathfrak{s}$-conflation
$\delta:\:\suc[A][B][C][u][v]$ in an extriangulated category $\D$
and any $W\in\mathcal{D}$, the sequences 
\begin{alignat*}{1}
(W,A)\overset{(W,u)}{\rightarrow}(W,B)\overset{(W,v)}{\rightarrow}(W,C)\overset{\delta\cdot-}{\rightarrow}\mathbb{E}(W,A)\overset{\mathbb{E}(W,u)}{\rightarrow}\mathbb{E}(W,B)\overset{\mathbb{E}(W,v)}{\rightarrow}\mathbb{E}(W,C),\\
{}\\
(C,W)\overset{(v,W)}{\rightarrow}(B,W)\overset{(u,W)}{\rightarrow}(A,W)\overset{-\cdot\delta}{\rightarrow}\mathbb{E}(C,W)\overset{\mathbb{E}(v,W)}{\rightarrow}\mathbb{E}(B,W)\overset{\mathbb{E}(u,W)}{\rightarrow}\mathbb{E}(A,W)
\end{alignat*}
are exact, where $(W,-):=\mathcal{\mathcal{D}}(W,-)$ and $(-,W):=\mathcal{\mathcal{D}}(-,W)$. 
\end{prop}

Let $\C$ be an additive category and $\suc[A][B][C][u][v]$ be morphisms
in $\C.$ We recall that $u$ is a \textbf{weak kernel} of $v$ if
$v\circ u=0$ and any morphism $a:W\rightarrow B$ in $\C$ such that
$v\circ a=0$ factors through $u.$ Dually, $v$ is a \textbf{weak
cokernel} of $u$ if $v\circ u=0$ and any morphism $b:B\rightarrow W$
in $\C$ such that $b\circ u=0$ factors through $v.$ 
\begin{rem}
Let $\D$ be an extriangulated category and $\delta:\:\suc[A][B][C][u][v]$
be any $\mathfrak{s}$-conflation. Then, as a consequence of Proposition
\ref{suc-1-ex}, we have that $u$ is a weak kernel of $v$ and $v$
is a weak cokernel of $u.$ 
\end{rem}

\subsection{\label{subsec:Extriangulated-categories-with-neg}Extriangulated
categories with negative first extensions and $s$-torsion pairs}

Let $\D=(\mathcal{D},\mathbb{E},\mathfrak{s})$ be an extriangulated
category and $\delta:\:\suc[A][B][C][u][v]$ be any $\mathfrak{s}$-conflation.
By Proposition \ref{suc-1-ex}, we get that any $W\in\mathcal{D}$
gives the exact sequences 
\begin{alignat*}{1}
(W,A)\overset{(W,u)}{\rightarrow}(W,B)\overset{(W,v)}{\rightarrow}(W,C)\overset{\delta\cdot-}{\rightarrow}\mathbb{E}(W,A)\overset{\mathbb{E}(W,u)}{\rightarrow}\mathbb{E}(W,B)\overset{\mathbb{E}(W,v)}{\rightarrow}\mathbb{E}(W,C),\\
(C,W)\overset{(v,W)}{\rightarrow}(B,W)\overset{(u,W)}{\rightarrow}(A,W)\overset{-\cdot\delta}{\rightarrow}\mathbb{E}(C,W)\overset{\mathbb{E}(v,W)}{\rightarrow}\mathbb{E}(B,W)\overset{\mathbb{E}(u,W)}{\rightarrow}\mathbb{E}(A,W).
\end{alignat*}
It is well-known that the above exact sequences can be extended to
the left for many extriangulated categories. For example, in the case
of abelian categories, the sequence is extended with zeros; and for
triangulated categories, it is extended with the functor $\mathcal{D}(W,-[-1])$.
When an extriangulated category satisfies this, it is said to have
negative first extension. In the following we recall such notion from
\cite{AET}, the reader is referred to \cite{G} for additional information. 
\begin{defn}
\cite[Def. 2.3]{AET} Let $\D=(\mathcal{D},\mathbb{E},\mathfrak{s})$
be an extriangulated category. We say that $\D$ has \textbf{negative
first extension} if there is an additive bifunctor $\mathbb{E}^{-1}:\mathcal{\mathcal{D}}^{op}\times\mathcal{D}\rightarrow\operatorname{Ab}$,
and for each $\delta\in\mathbb{E}(C,A)$ there exist two natural transformations
$\delta^{-1}:\mathbb{E}^{-1}(-,C)\rightarrow\mathcal{D}(-,A)$ and
$\delta_{-1}:\mathbb{E}^{-1}(A,-)\rightarrow\mathcal{D}(C,-)$ such
that, for any $\mathfrak{s}$-conflation $\delta:\:\suc[A][B][C][u][v]$
and each $W\in\mathcal{D}$, the sequences 
\begin{alignat*}{1}
\mathbb{E}^{-1}(W,A)\overset{\mathbb{E}^{-1}(W,u)}{\rightarrow}\mathbb{E}^{-1}(W,B)\overset{\mathbb{E}^{-1}(W,v)}{\rightarrow}\mathbb{E}^{-1}(W,C)\overset{\delta^{-1}}{\rightarrow}(W,A)\overset{(W,u)}{\rightarrow}(W,B),\\
\mathbb{E}^{-1}(C,W)\overset{\mathbb{E}^{-1}(v,W)}{\rightarrow}\mathbb{E}^{-1}(B,W)\overset{\mathbb{E}^{-1}(u,W)}{\rightarrow}\mathbb{E}^{-1}(A,W)\overset{\delta_{-1}}{\rightarrow}(C,W)\overset{(v,W)}{\rightarrow}(B,W)
\end{alignat*}
are exact. In this case, we say that $\D=(\D,\mathbb{E},\mathfrak{s},\mathbb{E}^{-1})$
is an extriangulated category with a negative first extension. 
\end{defn}

It can be seen that, by restricting the structure of an extriangulated
category with negative first extension to a subcategory closed under
extensions, we obtain a new extriangulated category with negative
first extension, see \cite[Ex. 2.4(iii)]{AET}. A particular case,
in which we will be interested, is the one of subcategories of triangulated
categories which are closed under extensions.

The structure of extriangulated categories allows us to develop concepts
parallel to that of torsion pair in abelian categories. In this section
we will present the notion of $s$-torsion for extriangulated categories
with negative first extension and recall some basic properties.

We start with the following definition of torsion pair in extriangulated
categories that comes naturally from the notion of torsion pair in
abelian categories. 
\begin{defn}
\cite[Def. 3.1]{HHZ} Let $\mathcal{D}$ be an extriangulated category.
A \textbf{torsion pair} in $\mathcal{D}$ is a pair $\mathbf{x}=(\mathcal{X},\mathcal{Y})$
of full subcategories of $\D$ which are closed under isomorphisms
in $\D$ and such that $\mathcal{D}(\mathcal{X},\mathcal{Y})=0$ and
$\mathcal{D}=\mathcal{X}\star\mathcal{Y}.$ 
\end{defn}

In case the extriangulated category has negative first extensions,
we have the notion of $s$-torsion pair. By following \cite{AET},
we recall that the term ``$s$'' in $s$-torsion stands for ``shift-closed''
by means of \cite[Lem. 3.3]{AET}. 
\begin{defn}
\cite[Def. 3.1]{AET} \label{def-neg-ext} Let $\mathcal{D}$ be an
extriangulated category with negative first extension. An \textbf{$s$-torsion
pair} in $\mathcal{\mathcal{D}}$ is a torsion pair $\mathbf{x}=(\mathcal{X},\mathcal{Y})$
such that $\mathbb{E}^{-1}(\mathcal{X},\mathcal{Y})=0$. 
\end{defn}

Let $\D$ be an extriangulated category with negative first extension.
For a class $\X\subseteq\D$ and $i\in\{-1,0,1\},$ we define the
right $i$-th orthogonal complement of $\X$ by $\X^{\perp_{i}}:=\{D\in\D\;:\;\mathbb{E}^{i}(\X,D)=0\},$
where $\mathbb{E}^{0}(-,-):=\Hom[\D][-][-].$ Dually, we have $^{\perp_{i}}\X$
which is the left $i$-th orthogonal complement of $\X.$ In case
the class $\X$ has only one element, i.e. $\X=\{X\},$ we set $X^{\perp_{i}}:=\X^{\perp_{i}}$
and $^{\perp_{i}}X:={}^{\perp_{i}}\X.$ As a consequence of the above
definitions, we get the following remark. 
\begin{rem}
\label{ceys} For an extriangulated category $\D$ with negative first
extension, a class $\X\subseteq\D$ and $i\in\{-1,0,1\},$ the classes
$\X^{\perp_{i}}$ and $^{\perp_{i}}\X$ are closed under extensions
and direct summands. 
\end{rem}

\begin{prop}
\cite[Prop. 3.2]{AET}\label{prop:clases ort en stors} Let $\p[\mathcal{T}][\mathcal{F}]$
be an $s$-torsion pair in $\D.$ Then, $\mathcal{T}^{\perp_{0}}=\mathcal{F}$
and $^{\perp_{0}}\mathcal{F}=\mathcal{T}$. 
\end{prop}

\begin{prop}
\label{lem:cerraduras stors} For an $s$-torsion pair $(\mathcal{T},\mathcal{F})$
in $\D,$ the following statements hold true. 
\begin{enumerate}
\item $\T$ is closed under cones, extensions and direct summands. 
\item $\F$ is closed under cocones, extensions and direct summands. 
\item Let $\suc[X][Y][Z]$ be an $\mathfrak{s}$-conflation. Then: 
\begin{enumerate}
\item $X\in{}^{\bot_{-1}}\mathcal{F}$ if $Y,Z\in\mathcal{T}$; 
\item $Z\in\mathcal{T}^{\bot_{-1}}$ if $X,Y\in\mathcal{F}$. 
\end{enumerate}
\item For any $s$-torsion pair $(\X,\Y)$ in $\D,$ $\X\subseteq\T$ $\Leftrightarrow$
$\F\subseteq\Y.$ 
\end{enumerate}
\end{prop}

\begin{proof}
The items (a), (b) and (c) follow from the long exact sequences appearing
in Definition \ref{def-neg-ext}, Remark \ref{ceys} and Proposition
\ref{prop:clases ort en stors}. Finally, (d) can be obtained from
Proposition \ref{prop:clases ort en stors}. 
\end{proof}
\begin{rem}
\label{adj-stor} Let $\mathbf{t}=(\mathcal{T},\mathcal{F})$ be an
$s$-torsion pair in $\D.$ It follows that, for all $C\in\mathcal{\mathcal{D}},$
there is an $\mathfrak{s}$-conflation $\delta_{C}:\,\suc[T_{C}][C][F_{C}][u][v]$
with $T_{C}\in\mathcal{T}$ and $F_{C}\in\mathcal{F}$. An important
property of $s$-torsion pairs is that the $\mathfrak{s}$-conflation
$\delta_{C}$ is unique up to isomorphisms of $\mathfrak{s}$-conflations.
Thus, we have a pair of functors: the torsion functor $\mathbf{t}:\mathcal{\mathcal{D}}\rightarrow\mathcal{T}$,
$C\mapsto T_{C}$, and the torsion-free functor $\left(\mathbf{1:t}\right):\mathcal{\mathcal{D}}\rightarrow\mathcal{F}$,
$C\mapsto F_{C}$, which are, respectively, the right and left adjoint
functors of the inclusion functors $\mathcal{T}\rightarrow\mathcal{\mathcal{D}}$
and $\mathcal{F}\rightarrow\mathcal{\mathcal{D}}.$ The reader is
referred to \cite[Prop. 3.7]{AET} for more details. 
\end{rem}

\subsection{$t$-structures in triangulated categories}

A particular case of an $s$-torsion pair in an extriangulated category
with negative first extension is the notion of $t$-structure in a
triangulated category. Since this will be an object of interest in
the paper, we include below the definition and some remarks.

Recall that a triangulated category consists of a triple $(\mathcal{\mathcal{D}},\Sigma,\triangle),$
where $\mathcal{\mathcal{D}}$ is an additive category, $\Sigma:\mathcal{\mathcal{D}}\rightarrow\mathcal{\mathcal{D}}$
is an automorphism, and $\triangle$ is a class of sequences of morphisms
(called triangles) of the form $\suc\rightarrow\Sigma N$ satisfying
a series of axioms. The reader is referred to \cite{N} for a precise
definition of triangulated category. The notion of $t$-structure
was first introduced in \cite{BBD}. Here we present an equivalent
definition, see \cite[Rk. 4.2]{PSV3} for details. 
\begin{defn}
Let $(\mathcal{D},\Sigma,\triangle)$ be a triangulated category.
A pair $(\mathcal{U},\mathcal{W})$ of full subcategories of $\mathcal{D}$
is called a \textbf{$t$-structure} in $\mathcal{D}$ if: $\U$ and
$\W$ are closed under isomorphisms, $\Sigma\U\subseteq\U$ and $(\U,\Sigma^{-1}\W)$
is a torsion pair in $\D.$ 
\end{defn}

\begin{rem}
\label{st-ts-tricat}\label{rem:s-torsion-implies-t-struc} Let $(\mathcal{D},\Sigma,\triangle)$
be a triangulated category and $(\mathcal{U},\mathcal{W})$ a pair
of full subcategories of $\mathcal{D}.$ 
\begin{enumerate}
\item \cite[Lem. 3.3]{AET} The pair $(\mathcal{U},\mathcal{W})$ is a $t$-structure
in $\D$ if and only if $(\mathcal{U},\Sigma^{-1}\mathcal{W})$ is
an $s$-torsion pair in $\mathcal{D}.$ 
\item \cite{BBD} If $\mathbf{u}:=(\mathcal{U},\mathcal{W})$ is a $t$-structure
in $\mathcal{D}$, then $\Hcal=\mathcal{H}_{\mathbf{u}}:=\mathcal{U}\cap\mathcal{W}$
is an abelian category and it is known as the \textbf{heart} of $\mathbf{u}$. 
\item Let $t:=(\T,\F)$ be an $s$-torsion pair in $\D.$ Then $\Sigma\,t:=(\Sigma\T,\Sigma\F)$
and $\Sigma^{-1}\,t:=(\Sigma^{-1}\T,\Sigma^{-1}\F)$ are $s$-torsion
pairs in $\D$ and $\Sigma\T\subseteq\T\subseteq\Sigma^{-1}\T.$ 
\end{enumerate}
\end{rem}

\subsection{Intervals of $s$-torsion pairs }

Let $\D$ be an extriangulated category with negative first extension.
Consider two $s$-torsion pairs $\mathbf{t}=(\mathcal{T},\mathcal{F})$
and $\mathbf{t}'=(\mathcal{T}',\mathcal{F}')$ in $\mathcal{\mathcal{D}}.$
By following \cite{AET}, $\mathbf{t}\leq\mathbf{t}'$ if $\mathcal{T}\subseteq\mathcal{T}'$.
By using this relation, we can consider the poset of $s$-torsion
pairs in $\mathcal{\mathcal{D}}$ that is denoted by $\stors\,\D.$
Moreover, given $\mathbf{r}:=(\mathcal{R},\mathcal{M})$ and $\mathbf{s}:=(\mathcal{S},\mathcal{N})$
in $\stors\,\D$ such that $\mathbf{r}\leq\mathbf{s}$, we have the
interval $[\mathbf{r},\mathbf{s}]:=\{\mathbf{t}\in\stors\,\D\;:\;\mathbf{r}\leq\mathbf{t}\leq\mathbf{s}\}$
and define the \textbf{heart} of $[\mathbf{r},\mathbf{s}]$ as the
class $\mathcal{H}_{[\mathbf{r},\mathbf{s}]}:=\mathcal{M}\cap\mathcal{S}.$
Observe that $\mathcal{H}_{[\mathbf{r},\mathbf{s}]}$ is closed under
extensions, and thus $\mathcal{H}_{[\mathbf{r},\mathbf{s}]}$ inherits
an extriangulated structure with negative first extension from the
extriangulated category $\D.$ 
\begin{example}
Let $\mathcal{\mathcal{D}}$ be a triangulated category and $\mathbf{t}:=(\T,\F)$
be an $s$-torsion pair in $\D.$ Then, by Remark \ref{st-ts-tricat}(c),
we have that $\Sigma\,\mathbf{t}:=(\Sigma\T,\Sigma\F)$ and $\Sigma^{-1}\,\mathbf{t}:=(\Sigma^{-1}\T,\Sigma^{-1}\F)$
are $s$-torsion pairs in $\D$ and $\Sigma\,\mathbf{t}\leq\mathbf{t}\leq\Sigma^{-1}\,\mathbf{t}.$
Now, consider the $t$-structures which correspond, respectively,
to $\Sigma\,\mathbf{t}$ and $\Sigma^{-1}\,\mathbf{t};$ namely: $\mathbf{t}_{*}:=(\Sigma\T,\Sigma^{2}\F)$
and $\mathbf{t}^{*}:=(\Sigma^{-1}\T,\F).$ Therefore, it follows that
$\mathcal{H}_{[\Sigma\,\mathbf{t},\mathbf{t}]}=\Sigma\,\F\cap\T=\Sigma^{-1}\,\mathcal{H}_{\mathbf{t}_{*}}$
and $\mathcal{H}_{[\mathbf{t},\Sigma^{-1}\,\mathbf{t}]}=\F\cap\Sigma^{-1}\,\T=\mathcal{H}_{\mathbf{t}^{*}}$
are abelian categories. 
\end{example}

\begin{example}
\label{exa:quasi-abelian} \cite[Sect. 1.2]{S}, \cite[Thm. 2.15]{F},
\cite[Sect. 5]{T}. Let $\mathcal{E}$ be a \textbf{quasi-abelian}
category. That is, $\mathcal{E}$ is an additive category such that:
every morphism admits a kernel and a cokernel; and every push-out
(pull-back) of a kernel (cokernel) is a kernel (cokernel). We can
associate to $\mathcal{E}$ a triangulated category $\mathcal{D}$.
Moreover, $\mathcal{D}$ is equipped with a pair of $s$-torsion pairs
$\t_{r},\t_{\ell}\in\stors\,\mathcal{D}$ such that $\Sigma\t_{r}\leq\t_{\ell}\leq\t_{r}$
and $\mathcal{H}_{[\Sigma\t_{r},\t_{\ell}]}\cong\mathcal{E}.$ The
category $\mathcal{R}_{\mathcal{E}}:=\mathcal{H}_{[\Sigma\t_{r},\t_{r}]}$
is known as the \textbf{right associated abelian category of} $\mathcal{E}$,
and the category $\mathcal{L}_{\mathcal{E}}:=\mathcal{H}_{[\Sigma\t_{\ell},\t_{\ell}]}$
is known as the \textbf{left associated abelian category of} $\mathcal{E}$. 
\end{example}

\begin{rem}
\cite[Lem. 3.11]{AET}\label{rem:cor} Let $\mathbf{r}:=(\mathcal{R},\mathcal{M}),\mathbf{s}:=(\mathcal{S},\mathcal{N})\in\stors\,\mathcal{D}$
be such that $\mathbf{r}\leq\mathbf{s}$. Then $\mathcal{S}=\mathcal{R}\star\mathcal{H}_{[\mathbf{r},\mathbf{s}]}$
and $\mathcal{M}=\mathcal{H}_{[\mathbf{r},\mathbf{s}]}\star\mathcal{N}$. 
\end{rem}

A classical result from \cite{HRS} shows how to parameterize a certain
family of $t$-structures by means of the torsion pairs of the heart
of a given $t$-structure. The following theorem can be understood
as a generalization of this result. It also extends similar results
for quasi-abelian subcategories of abelian categories \cite{T}. 
\begin{thm}
\cite[Thm. 3.9]{AET}\label{thm:proceso aet} Let $\mathcal{D}$ be
an extriangulated category with negative first extension, and let
$\mathbf{t}_{1}=(\mathcal{T}_{1},\mathcal{F}_{1})$ and $\mathbf{t}_{2}=(\mathcal{T}_{2},\mathcal{F}_{2})$
in $\stors\,\mathcal{\mathcal{D}}$ be such that $\mathbf{t}_{1}\leq\mathbf{t}_{2}$.
Then, there exist an isomorphism of posets 
\[
\Phi:[\mathbf{t}_{1},\mathbf{t}_{2}]\rightarrow\stors\,\mathcal{H}_{[\mathbf{t}_{1},\mathbf{t}_{2}]},\;(\mathcal{T},\mathcal{F})\mapsto(\mathcal{T}\cap\mathcal{F}_{1},\mathcal{T}_{2}\cap\mathcal{F})
\]
with inverse 
\[
\Psi:\stors\,\mathcal{H}_{[\mathbf{t}_{1},\mathbf{t}_{2}]}\rightarrow[\mathbf{t}_{1},\mathbf{t}_{2}],\;(\mathcal{X},\mathcal{Y})\mapsto(\mathcal{T}_{1}\star\mathcal{X},\mathcal{Y}\star\mathcal{F}_{2}).
\]
\end{thm}

\section{Extensions of hearts and their $s$-torsion pairs}
\begin{lem}
\label{lem:extensions of hearts} Let $\mathcal{D}$ be an extriangulated
category with negative first extension, and let $\mathbf{t}_{i}=(\mathcal{T}_{i},\mathcal{F}_{i})$
in $\stors\,\mathcal{D}$ for $i=1,2,3$ be such that $\mathbf{t}_{1}\leq\mathbf{t}_{2}\leq\mathbf{t}_{3}.$
Then 
\[
\mathcal{H}_{[\mathbf{t}_{1},\mathbf{t}_{3}]}=\mathcal{H}_{[\mathbf{t}_{1},\mathbf{t}_{2}]}\star\mathcal{H}_{[\mathbf{t}_{2},\mathbf{t}_{3}]}.
\]
Moreover, for $C\in\Hcal_{[\t_{1},\t_{3}]}$ and the $\mathfrak{s}$-conflation
$\suc[T_{2}][C][F_{2}]$ given by $\mathbf{t_{2}},$ we have that
$F_{2}\in\mathcal{H}_{[\mathbf{t}_{2},\mathbf{t}_{3}]}$ and $T_{2}\in\mathcal{H}_{[\mathbf{t}_{1},\mathbf{t}_{2}]}.$ 
\end{lem}

\begin{proof}
Since $\mathcal{H}_{[\mathbf{t}_{2},\mathbf{t}_{3}]}=\mathcal{T}_{3}\cap\mathcal{F}_{2}\subseteq\mathcal{T}_{3}\cap\mathcal{F}_{1}=\mathcal{H}_{[\mathbf{t}_{1},\mathbf{t}_{3}]},$
$\mathcal{H}_{[\mathbf{t}_{1},\mathbf{t}_{2}]}=\mathcal{T}_{2}\cap\mathcal{F}_{1}\subseteq\mathcal{T}_{3}\cap\mathcal{F}_{1}=\mathcal{H}_{[\mathbf{t}_{1},\mathbf{t}_{3}]}$
and $\mathcal{H}_{[\mathbf{t}_{1},\mathbf{t}_{3}]}$ is closed under
extensions, it follows that $\mathcal{H}_{[\mathbf{t}_{1},\mathbf{t}_{2}]}\star\mathcal{H}_{[\mathbf{t}_{2},\mathbf{t}_{3}]}\subseteq\mathcal{H}_{[\mathbf{t}_{1},\mathbf{t}_{3}]}.$
\
 Let $C\in\mathcal{H}_{[\mathbf{t}_{1},\mathbf{t}_{3}]}.$ Consider
the $\mathfrak{s}$-conflation $\suc[T_{2}][C][F_{2}]$ with $T_{2}\in\mathcal{T}_{2}$
and $F_{2}\in\mathcal{F}_{2}$ given by $\mathbf{t_{2}}$. By Proposition
\ref{lem:cerraduras stors}(a), we obtain that $F_{2}\in\mathcal{T}_{3}$
since $T_{2}\in\mathcal{T}_{2}\subseteq\mathcal{T}_{3}$ and $C\in\mathcal{H}_{[\mathbf{t}_{1},\mathbf{t}_{3}]}\subseteq\mathcal{T}_{3}$.
Similarly, it can be shown that $T_{2}\in\mathcal{F}_{1}.$ Therefore
$F_{2}\in\mathcal{H}_{[\mathbf{t}_{2},\mathbf{t}_{3}]}$ and $T_{2}\in\mathcal{H}_{[\mathbf{t}_{1},\mathbf{t}_{2}]};$
proving that $\mathcal{H}_{[\mathbf{t}_{1},\mathbf{t}_{3}]}\subseteq\mathcal{H}_{[\mathbf{t}_{1},\mathbf{t}_{2}]}\star\mathcal{H}_{[\mathbf{t}_{2},\mathbf{t}_{3}]}.$ 
\end{proof}
\begin{lem}
\label{lem:stors de extensiones de corazones} For an extriangulated
category $\D$ with negative first extension, and $\mathbf{t}_{i}=(\mathcal{T}_{i},\mathcal{F}_{i})$
in $\stors\,\mathcal{D}$ for $i=1,2,3$ such that $\mathbf{t}_{1}\leq\mathbf{t}_{2}\leq\mathbf{t}_{3}$,
the following statements hold true. 
\begin{enumerate}
\item If $(\mathcal{T},\mathcal{F})$ is an $s$-torsion pair in $\mathcal{H}_{[\mathbf{t}_{1},\mathbf{t}_{2}]}$,
then $(\mathcal{T},\mathcal{F}\star\mathcal{H}_{[\mathbf{t}_{2},\mathbf{t}_{3}]})$
is an $s$-torsion pair in $\mathcal{H}_{[\mathbf{t}_{1},\mathbf{t}_{3}]}$. 
\item If $(\mathcal{T},\mathcal{F})$ is an $s$-torsion pair in $\mathcal{H}_{[\mathbf{t}_{2},\mathbf{t}_{3}]}$,
then $(\mathcal{H}_{[\mathbf{t}_{1},\mathbf{t}_{2}]}\star\mathcal{T},\mathcal{F})$
is an $s$-torsion pair in $\mathcal{H}_{[\mathbf{t}_{1},\mathbf{t}_{3}]}$. 
\item $(\mathcal{H}_{[\t_{1},\t_{2}]},\mathcal{H}_{[\t_{2},\t_{3}]})$ is
an $s$-torsion pair in $\mathcal{H}_{[\t_{1},\t_{3}]}$.
\item If $(\mathcal{T},\mathcal{F})$ is an $s$-torsion pair in $\mathcal{H}_{[\mathbf{t}_{1},\mathbf{t}_{3}]}$
with $\mathcal{T}\subseteq\mathcal{H}_{[\t_{1},\t_{2}]}$, then $\mathcal{F}':=\mathcal{F}\cap\mathcal{H}_{[\t_{1},\t_{2}]}$
satisfies that $(\mathcal{T},\mathcal{F}')\in\stors\,\mathcal{H}_{[\t_{1},\t_{2}]}$
and $\mathcal{F}=\mathcal{F}'\star\mathcal{H}_{[\t_{2},\t_{3}]}$.
\item If $(\mathcal{T},\mathcal{F})$ is an $s$-torsion pair in $\mathcal{H}_{[\mathbf{t}_{1},\mathbf{t}_{3}]}$
with $\mathcal{F}\subseteq\mathcal{H}_{[\t_{2},\t_{3}]}$, then $\mathcal{T}':=\mathcal{T}\cap\mathcal{H}_{[\t_{2},\t_{3}]}$
satisfies that $(\mathcal{T}',\mathcal{F})\in\stors\,\mathcal{H}_{[\t_{2},\t_{3}]}$
and $\mathcal{T}=\mathcal{H}_{[\t_{1},\t_{2}]}\star\mathcal{T}'$.
\item For every $(\mathcal{T},\mathcal{F})\in\stors\,\mathcal{H}_{[\t_{1},\t_{2}]}$
there is $\t\in[\t_{1},\t_{2}]$ such that $\mathcal{T}=\mathcal{H}_{[\t_{1},\t]}$
and $\mathcal{F}=\mathcal{H}_{[\t,\t_{2}]}$. 
\end{enumerate}
\end{lem}

\begin{proof}
We only prove (a) and (d) since: (b) follows by dual arguments as
(a), (c) is a particular case of (a), (e) follows by dual arguments
as (d), and (f) is a consequence of Theorem \ref{thm:proceso aet}.

(a) Let $(\mathcal{T},\mathcal{F})$ be an $s$-torsion pair in $\mathcal{H}_{[\mathbf{t}_{1},\mathbf{t}_{2}]}.$
Then $\mathcal{H}_{[\t_{1},\t_{2}]}=\mathcal{T}\star\mathcal{F}$
and thus, by Lemma \ref{lem:extensions of hearts}, we have that 
\[
\mathcal{H}_{[\t_{1},\t_{3}]}=\mathcal{H}_{[\t_{1},\t_{2}]}\star\mathcal{H}_{[\t_{2},\t_{3}]}=\left(\mathcal{T}\star\mathcal{F}\right)\star\mathcal{H}_{[\t_{2},\t_{3}]}=\mathcal{T}\star\left(\mathcal{F}\star\mathcal{H}_{[\t_{2},\t_{3}]}\right).
\]
It remains to prove that $\mathcal{D}(\mathcal{T},\mathcal{F}\star\mathcal{H}_{[\t_{2},\t_{3}]})=0$
and that $\mathbb{E}^{-1}(\mathcal{T},\mathcal{F}\star\mathcal{H}_{[\t_{2},\t_{3}]})=0$.
Indeed, since $\mathcal{F}\subseteq\mathcal{T}^{\bot_{0}}\cap\mathcal{T}^{\bot_{-1}},$
$\mathcal{H}_{[\t_{2},\t_{3}]}\subseteq\mathcal{F}_{2}\subseteq\mathcal{T}_{2}^{\bot_{0}}\cap\mathcal{T}_{2}^{\bot_{-1}}\subseteq\mathcal{T}^{\bot_{0}}\cap\mathcal{T}^{\bot_{-1}}$
and $\mathcal{T}^{\bot_{0}}$ and $\mathcal{T}^{\bot_{-1}}$ are closed
under extensions, we have that $\mathcal{D}(\mathcal{T},\mathcal{F}\star\mathcal{H}_{[\t_{2},\t_{3}]})=0$
and $\mathbb{E}^{-1}(\mathcal{T},\mathcal{F}\star\mathcal{H}_{[\t_{2},\t_{3}]})=0$.

(d) Let $\t=(\mathcal{T},\mathcal{F})\in\stors\,\mathcal{H}_{[\mathbf{t}_{1},\mathbf{t}_{3}]}$
with $\mathcal{T}\subseteq\mathcal{H}_{[\t_{1},\t_{2}]}$. Consider
the $s$-torsion pair $\mathbf{c}=(\mathcal{H}_{[\t_{1},\t_{2}]},\mathcal{H}_{[\t_{2},\t_{3}]})\in\stors\,\mathcal{H}_{[\mathbf{t}_{1},\mathbf{t}_{3}]}$
showed in (c) and the following maps given by using Theorem \ref{thm:proceso aet}:
\[
\stors\,\mathcal{H}_{[\t_{1},\t_{3}]}\overset{\Psi_{13}}{\rightarrow}[\t_{1},\t_{3}]\hookleftarrow[\t_{1},\t_{2}]\overset{\Phi_{12}}{\rightarrow}\stors\,\mathcal{H}_{[\t_{1},\t_{2}]}.
\]
Since $\t\leq\mathbf{c}$ and $\mathbf{c}=\Phi_{13}(\t_{2})$, we
have that $\Psi_{13}(\t)\in[\t_{1},\t_{2}]$. Therefore, $\Phi_{12}\Psi_{13}(\t)\in\stors\,\mathcal{H}_{[\t_{1},\t_{2}]}$.
Observe that 
\[
\Phi_{12}\Psi_{13}(\t)=((\mathcal{T}_{1}\star\mathcal{T})\cap\mathcal{F}_{1},\mathcal{T}_{2}\cap(\mathcal{F}\star\mathcal{F}_{3}))=(\mathcal{T},\mathcal{F}\cap\mathcal{H}_{[\t_{1},\t_{2}]}).
\]
Indeed, by Proposition \ref{prop:clases ort en stors}, it is enough
to show that $(\mathcal{T}_{1}\star\mathcal{T})\cap\mathcal{F}_{1}=\mathcal{T}$.
For this, note that $(\mathcal{T}_{1}\star\mathcal{T})\cap\mathcal{F}_{1}\subseteq\mathcal{T}$
follows straightforward and the opposite contention follows from the
fact that $\mathcal{T}\subseteq\mathcal{H}_{[\t_{1},\t_{2}]}=\mathcal{T}_{2}\cap\mathcal{F}_{1}$. 
\end{proof}
As an application of Lemma \ref{lem:stors de extensiones de corazones},
we can get the following example in a direct way. This was already
known, from \cite[Sect. 1.2]{S}, \cite[Them. 2.15]{F}, \cite[Sect. 5]{T}. 
\begin{example}
\label{exa:categorias abelianas asociadas a quasiabeliana} Let $\mathcal{E}$
be a quasi-abelian category (see Example \ref{exa:quasi-abelian}).
We have mentioned that there is a triangulated category $\mathcal{D}$
equipped with a pair of $s$-torsion pairs $\t_{r},\t_{\ell}\in\stors\,\mathcal{D}$
such that $\Sigma\t_{r}\leq\t_{\ell}\leq\t_{r}$ and $\mathcal{H}_{[\Sigma\t_{r},\t_{\ell}]}\cong\mathcal{E}$.
It follows from Lemma \ref{lem:stors de extensiones de corazones}
(c) that $(\mathcal{H}_{[\Sigma\t_{r},\t_{\ell}]},\mathcal{H}_{[\t_{\ell},\t_{r}]})$
is a torsion pair in the abelian category $\mathcal{R}_{\mathcal{E}}:=\mathcal{H}_{[\Sigma\t_{r},\t_{r}]}$,
and that $(\mathcal{H}_{[\Sigma\t_{\ell},\Sigma\t_{r}]},\mathcal{H}_{[\Sigma\t_{r},\t_{\ell}]})$
is a torsion pair in the abelian category $\mathcal{L}_{\mathcal{E}}:=\mathcal{H}_{[\Sigma\t_{\ell},\t_{\ell}]}$.
Therefore, $\mathcal{E}$ can be realized as a torsion class in the
abelian category $\mathcal{R}_{\mathcal{E}}$, or as torsion-free
class in the abelian category $\mathcal{L}_{\mathcal{E}}.$ 
\end{example}

From the previous lemmas, we obtain the following result. 
\begin{cor}
\label{cor-extcorn} Let $\D$ be an extriangulated category with
negative first extension, and let $\mathbf{t}_{i}=(\mathcal{T}_{i},\mathcal{F}_{i})$
in $\stors\,\D,$ for $i=1,\cdots,n$, be such that $\mathbf{t}_{1}\leq\cdots\leq\mathbf{t}_{n}$.
Then, the following statements hold true. 
\begin{enumerate}
\item $\mathcal{H}_{[\mathbf{t}_{1},\mathbf{t}_{n}]}=\mathcal{H}_{[\mathbf{t}_{1},\mathbf{t}_{2}]}\star\mathcal{H}_{[\mathbf{t}_{2},\mathbf{t}_{3}]}\star\cdots\star\mathcal{H}_{[\mathbf{t}_{n-1},\mathbf{t}_{n}]}$. 
\item If $(\mathcal{T},\mathcal{F})$ is an $s$-torsion pair in $\mathcal{H}_{[\t_{k},\t_{k+1}]}$
for some $1\leq k<n$, then $(\mathcal{H}_{[\t_{1},\t_{k}]}\star\mathcal{T},\mathcal{F}\star\mathcal{H}_{[\t_{k+1},\t_{n}]})$
is an $s$-torsion pair in $\mathcal{H}_{[\t_{1},\t_{n}]}$. 
\end{enumerate}
\end{cor}

\begin{defn}
\label{nst-int} Let $\mathcal{D}$ be an extriangulated category
with negative first extension, and let $\mathbf{t}_{i}=(\mathcal{T}_{i},\mathcal{F}_{i})$
in $\stors\,\mathcal{D}$ for $i=1,2$ be such that $\mathbf{t}_{1}\leq\mathbf{t}_{2}.$
We say that $[\t_{1},\t_{2}]$ is a \textbf{normal interval} if $\mathbb{E}^{-1}(\mathcal{T}_{2},\mathcal{F}_{1})=0$. 
\end{defn}

\begin{example}
\label{exa:normal} Let $\mathcal{D}$ be an extriangulated category
with negative first extension. 
\begin{enumerate}
\item If $\mathbb{E}^{-1}=0$ (e.g. in the case $\mathcal{D}$ is exact),
then any interval in $\stors\,\mathcal{D}$ is normal. 
\item Let $\mathcal{D}$ be a triangulated category and $[\t_{1},\t_{2}]$
be an interval in $\stors\,\D.$ Then $\mathbb{E}^{-1}(\mathcal{T}_{2},\mathcal{F}_{1})=\mathcal{D}(\mathcal{T}_{2},\Sigma^{-1}\mathcal{F}_{1})$.
Therefore, if we fix $\t_{1},$ the interval $[\t_{1},\t_{2}]$ is
normal if, and only if, $\t_{2}$ satisfies that $\Sigma{}^{-1}\mathcal{F}_{1}\subseteq\mathcal{F}_{2}$,
or equivalently that $\Sigma\mathcal{T}_{2}\subseteq\mathcal{T}_{1}$.
Note that these are precisely the $s$-torsion pairs ($t$-structures)
parameterized by the HRS-tilting process \cite[Prop. 2.1]{W}. 
\end{enumerate}
\end{example}

\begin{lem}
\label{nimpex} Let $\mathcal{D}$ be an extriangulated category with
negative first extension, and let $[\mathbf{t}_{1},\mathbf{t}_{2}]$
be an interval in $\stors\,\mathcal{D}$ with $\mathbf{t}_{i}=(\mathcal{T}_{i},\mathcal{F}_{i}),$
for $i=1,2.$ Then, $[\t_{1},\t_{2}]$ is normal if and only if $\mathbb{E}^{-1}(\mathcal{H}_{[\t_{1},\t_{2}]},\mathcal{H}_{[\t_{1},\t_{2}]})=0$.
In particular, $\mathcal{H}_{[\t_{1},\t_{2}]}$ is an exact category
if $[\t_{1},\t_{2}]$ is normal. 
\end{lem}

\begin{proof}
$(\Rightarrow)$ It follows from the fact that $\mathbb{E}^{-1}(\T_{2},\F_{1})=0$
and $\mathcal{H}_{[\t_{1},\t_{2}]}=\F_{1}\cap\T_{2}.$

$(\Leftarrow)$ Observe that $\mathcal{T}_{2}=\mathcal{T}_{1}\star\mathcal{H}_{[\t_{1},\t_{2}]}$
and $\mathcal{F}_{1}=\mathcal{H}_{[\t_{1},\t_{2}]}\star\mathcal{F}_{2}$
(see Theorem \ref{thm:proceso aet}). Hence, since $\mathbb{E}^{-1}(\mathcal{H}_{[\t_{1},\t_{2}]},\mathcal{F}_{2})=0$,
$\mathbb{E}^{-1}(\mathcal{T}_{1},\mathcal{H}_{[\t_{1},\t_{2}]})=0$,
$\mathbb{E}^{-1}(\mathcal{H}_{[\t_{1},\t_{2}]},\mathcal{H}_{[\t_{1},\t_{2}]})=0$
and $\mathbb{E}^{-1}(\mathcal{T}_{1},\mathcal{F}_{1})=0$, we have
that 
\[
\mathbb{E}^{-1}(\mathcal{T}_{2},\mathcal{F}_{1})=\mathbb{E}^{-1}(\mathcal{T}_{1}\star\mathcal{H}_{[\t_{1},\t_{2}]},\mathcal{H}_{[\t_{1},\t_{2}]}\star\mathcal{F}_{2})=0.
\]
Lastly, by \cite[Prop. 2.6]{AET}, we conclude that $\mathcal{H}_{[\t_{1},\t_{2}]}$
is an exact category. 
\end{proof}
\begin{lem}
\label{n3int} For an extriangulated category $\D$ with negative
first extension and $\mathbf{t}_{i}=(\mathcal{T}_{i},\mathcal{F}_{i})$
in $\stors\,\mathcal{D}$ for $i=1,2,3$ such that $\mathbf{t}_{1}\leq\mathbf{t}_{2}\leq\mathbf{t}_{3},$
the following statements hold true. 
\begin{enumerate}
\item If $[\t_{1},\t_{3}]$ is normal, then $[\t_{1},\t_{2}]$ and $[\t_{2},\t_{3}]$
are normal. 
\item If $[\t_{1},\t_{2}]$ and $[\t_{2},\t_{3}]$ are normal and $\mathbb{E}^{-1}(\mathcal{H}_{[\t_{2},\t_{3}]},\mathcal{H}_{[\t_{1},\t_{2}]})=0$,
then $[\t_{1},\t_{3}]$ is normal. 
\end{enumerate}
\end{lem}

\begin{proof}
The proof of (a) is straightforward since $\mathcal{F}_{2}\subseteq\mathcal{F}_{1}$
and $\mathcal{T}_{2}\subseteq\mathcal{T}_{3}.$ \
 Assume that $[\t_{1},\t_{2}]$ and $[\t_{2},\t_{3}]$ are normal
and $\mathbb{E}^{-1}(\mathcal{H}_{[\t_{2},\t_{3}]},\mathcal{H}_{[\t_{1},\t_{2}]})=0$.
In order to prove (b), it is enough to show that $\mathbb{E}^{-1}(\mathcal{H}_{[\t_{1},\t_{3}]},\mathcal{H}_{[\t_{1},\t_{3}]})=0$
by Lemma \ref{nimpex}. For this, recall that $\mathcal{H}_{[\t_{1},\t_{3}]}=\mathcal{H}_{[\t_{1},\t_{2}]}\star\mathcal{H}_{[\t_{2},\t_{3}]}$
by Lemma \ref{lem:extensions of hearts}. And thus, since $\mathbb{E}^{-1}(\mathcal{H}_{[\t_{2},\t_{3}]},\mathcal{H}_{[\t_{1},\t_{2}]})=0$,
$\mathbb{E}^{-1}(\mathcal{H}_{[\t_{1},\t_{2}]},\mathcal{H}_{[\t_{1},\t_{2}]})=0$,
$\mathbb{E}^{-1}(\mathcal{H}_{[\t_{2},\t_{2}]},\mathcal{H}_{[\t_{1},\t_{2}]})=0$
by hypothesis and Lemma \ref{nimpex}, and $\mathbb{E}^{-1}(\mathcal{H}_{[\t_{1},\t_{2}]},\mathcal{H}_{[\t_{2},\t_{3}]})\subseteq\mathbb{E}^{-1}(\mathcal{T}_{2},\mathcal{F}_{2})=0$,
we have that $\mathbb{E}^{-1}(\mathcal{H}_{[\t_{1},\t_{3}]},\mathcal{H}_{[\t_{1},\t_{3}]})=0$. 
\end{proof}
\begin{prop}
\label{lemaA} Let $\mathcal{D}$ be an extriangulated category with
negative first extension and $\t_{1},\t_{2},\t_{3}\in\stors\,\mathcal{D}$
such that $\t_{1}\leq\t_{2}\leq\t_{3}$. Then, for an $\s$-conflation
$\suc[X][Y][Z]$ in $\mathcal{H}_{[\t_{1},\t_{3}]},$ the following
statements hold true, 
\begin{enumerate}
\item If $[\t_{2},\t_{3}]$ is normal and $Y\in\mathcal{H}_{[\t_{1},\t_{2}]}$,
then $Z\in\mathcal{H}_{[\t_{1},\t_{2}]}$. 
\item If $[\t_{1},\t_{2}]$ is normal and $Y\in\mathcal{H}_{[\t_{2},\t_{3}]}$,
then $X\in\mathcal{H}_{[\t_{2},\t_{3}]}$. 
\end{enumerate}
\end{prop}

\begin{proof}
We only prove (a) since the proof of (b) is similar. \
 Consider an $\s$-conflation $\suc[X][Y][Z]$ in $\C$ with $Y\in\mathcal{H}_{[\t_{1},\t_{2}]}$.
Note that $X\in{}^{\bot_{-1}}\mathcal{H}_{[\t_{2},\t_{3}]}$ since
$X\in\mathcal{H}_{[\t_{1},\t_{3}]}=\mathcal{T}_{3}\cap\mathcal{F}_{1}\subseteq\mathcal{T}_{3}$,
$\mathcal{H}_{[\t_{2},\t_{3}]}\subseteq\mathcal{F}_{2}$, and $[\t_{2},\t_{3}]$
is normal. Then, it follows from the previous $\s$-conflation that
$Z\in{}^{\bot_{0}}\mathcal{H}_{[\t_{2},\t_{3}]}$. Thus, by applying
Proposition \ref{prop:clases ort en stors} to the $s$-torsion pair
$(\mathcal{H}_{[\t_{1},\t_{2}]},\mathcal{H}_{[\t_{2},\t_{3}]})$ in
$\mathcal{H}_{[\t_{1},\t_{3}]}$ (see Lemma \ref{lem:stors de extensiones de corazones}
(c)), we have that $Z\in\mathcal{H}_{[\t_{1},\t_{2}]}.$ 
\end{proof}
\begin{defn}
Let $\mathcal{D}$ be an extriangulated category. We say that an $\mathfrak{s}$-conflation
$\suc[][][][a][b]$ in $\mathcal{D}$ is a \textbf{short exact sequence}
if: $a$ is the kernel of $b$ in $\mathcal{D}$, and $b$ is the
cokernel of $a$ in $\mathcal{D}$. 
\end{defn}

\begin{prop}
\label{prop:normal vs exactitud} Let $\mathcal{D}$ be an extriangulated
category with negative first extension, and let $\mathbf{t}_{i}=(\mathcal{T}_{i},\mathcal{F}_{i})$
in $\stors\,\mathcal{D}$ for $i=1,2,3$ be such that $\mathbf{t}_{1}\leq\mathbf{t}_{2}\leq\mathbf{t}_{3}.$
Then, for any $C\in\mathcal{H}_{[\t_{1},\t_{3}]},$ the following
statements hold true for the $\mathfrak{s}$-conflation $\eta:\:T_{2}\xrightarrow{f_{C}}C\xrightarrow{g_{C}}F_{2}$
in $\D$ given by $\t_{2}.$ 
\begin{enumerate}
\item $T_{2},F_{2}\in\mathcal{H}_{[\t_{1},\t_{3}]}$, and thus $\eta$ is
an $\mathfrak{s}$-conflation in $\mathcal{H}_{[\t_{1},\t_{3}]}$. 
\item If $[\t_{2},\t_{3}]$ is normal, then $f_{C}$ is a monomorphism in
$\mathcal{H}_{[\t_{1},\t_{3}]}$ and $f_{C}$ is the kernel of $g_{C}$
in $\mathcal{H}_{[\t_{1},\t_{3}]}$. 
\item If $[\t_{1},\t_{2}]$ is normal, then $g_{C}$ is an epimorphism in
$\mathcal{H}_{[\t_{1},\t_{3}]}$ and $g_{C}$ is the cokernel of $f_{C}$
in $\mathcal{H}_{[\t_{1},\t_{3}]}$. 
\item If $[\t_{1},\t_{2}]$ and $[\t_{2},\t_{3}]$ are normal, then $\eta$
is a short exact sequence in $\mathcal{H}_{[\t_{1},\t_{3}]}$. 
\end{enumerate}
\end{prop}

\begin{proof}
Notice that (a) follows from Lemma \ref{lem:extensions of hearts},
(c) is dual to (b) and (d) follows from (b) and (c). Thus, it is enough
to prove (b). \
 (b) Let us show that $f_{C}$ is monomorphism if $[\t_{2},\t_{3}]$
is normal. Consider a morphism $h:W\rightarrow T_{2}$ with $W\in\mathcal{H}_{[\t_{1},\t_{3}]}$
and such that $f_{C}\circ h=0.$ By applying the functor $\D(W,-)$
to $\eta,$ we get the exact sequence 
\[
\mathbb{E}^{-1}(W,F_{2})\xrightarrow{\eta^{-1}}\mathcal{D}(W,T_{2})\rightarrow\mathcal{D}(W,C).
\]
Note that $h\in\Ker[(\mathcal{D}(W,f_{C}))]=\im[\mathcal{\eta}^{-1}]$.
But $\mathbb{E}^{-1}(W,F_{2})=0$ since $W\in\mathcal{T}_{3}$, $F_{2}\in\mathcal{F}_{2}$
and $[\t_{2},\t_{3}]$ is normal. Therefore, $h=0$ and thus $f_{C}$
is a monomorphism. Finally, $f_{C}$ is the kernel for $g_{C}$ in
$\mathcal{H}_{[\t_{1},\t_{3}]}$ since it is a monomorphism and a
weak kernel for $g_{C}.$ 
\end{proof}
\begin{prop}
\label{ex-ker-coker} For an extriangulated category $\D$ with negative
first extension, $\mathbf{t}_{i}=(\mathcal{T}_{i},\mathcal{F}_{i})$
in $\stors\,\mathcal{D}$ for $i=1,2$ such that $[\t_{1},\t_{2}]$
is a normal interval in $\stors\,\D$ and $X,Y\in\mathcal{H}_{[\t_{1},\t_{2}]},$
the following statements hold true. 
\begin{enumerate}
\item Let $\chi:\:\suc[X][Y][Z][f][c]$ be an $\mathfrak{s}$-conflation
in $\mathcal{D}.$ Then $f$ has cokernel in $\mathcal{H}_{[\t_{1},\t_{2}]}$
and there are morphisms $X\xrightarrow{r}E\xrightarrow{h}Y$ in $\mathcal{H}_{[\t_{1},\t_{2}]}$
such that $f=h\circ r,$ $r$ is an epimorphism in $\mathcal{H}_{[\t_{1},\t_{2}]}$
and $h=\Ker[{(\Cok[f])}]=:\im[f]$ in $\mathcal{H}_{[\t_{1},\t_{2}]}.$ 
\item Let $\kappa:\:\suc[K][X][Y][k][f]$ be an $\mathfrak{s}$-conflation
in $\mathcal{D}.$ Then $f$ has kernel in $\mathcal{H}_{[\t_{1},\t_{2}]}$
and there are morphisms $X\xrightarrow{r}E\xrightarrow{h}Y$ in $\mathcal{H}_{[\t_{1},\t_{2}]}$
such that $f=h\circ r,$ $h$ is a monomorphism in $\mathcal{H}_{[\t_{1},\t_{2}]}$
and $r=\Cok[{\Ker[(f)]}]=:\coim[f]$ in $\mathcal{H}_{[\t_{1},\t_{2}]}.$ 
\end{enumerate}
\end{prop}

\begin{proof}
We only prove (a) since the proof of (b) is dual. \
 We have the $\mathfrak{s}$-conflation $\chi:\:\suc[X][Y][Z][f][c]$
in $\mathcal{D}.$ Since $\T_{2}$ is closed under cones (see Proposition
\ref{lem:cerraduras stors}) and $X,Y\in\mathcal{T}_{2},$ we have
that $Z\in\mathcal{T}_{2}.$ Consider the $\mathfrak{s}$-conflation
$\eta_{1}:\:\suc[T_{1}][Z][F_{1}][][c_{1}]$ induced by $\t_{1}.$
Then $F_{1}\in\mathcal{F}_{1}\cap\mathcal{T}_{2}=\mathcal{H}_{[\t_{1},\t_{2}]}$
since $T_{1}\in\T_{1}\subseteq\T_{2}.$ Moreover, by \cite[Rk. 2.22]{NP},
we have the following commutative diagram in $\D$ 
\[
\xymatrix{X\ar[r]^{r}\ar@{=}[d] & E\ar[r]\ar[d]^{h} & T_{1}\ar[d]\\
X\ar[r]^{f} & Y\ar[r]^{c}\ar[d]^{c_{2}} & Z\ar[d]^{c_{1}}\\
 & F_{1}\ar@{=}[r] & F_{1}
}
\]
where the rows and columns are $\mathfrak{s}$-conflations in $\D.$
By Proposition \ref{lem:cerraduras stors} (a,b), we get that the
$\mathfrak{s}$-conflation $\eta_{2}:\,E\xrightarrow{h}Y\xrightarrow{c_{2}}F_{1}$
lies in $\mathcal{H}_{[\t_{1},\t_{2}]}.$ Therefore (see Lemma \ref{nimpex})
$\eta_{2}$ is a short exact sequence in the exact category $\mathcal{H}_{[\t_{1},\t_{2}]},$
and thus $h=\Ker[(c_{2})]$ and $c_{2}=\Cok[h].$ \
 We assert that $r:X\to E$ is an epimorphism in $\mathcal{H}_{[\t_{1},\t_{2}]}$
and $c_{2}:Y\to F_{1}$ is the cokernel of $f:X\to Y$ in $\mathcal{H}_{[\t_{1},\t_{2}]}.$
Indeed, consider $W\in\mathcal{H}_{[\t_{1},\t_{2}]}.$ Then, by applying
the functor $\D(-,W)$ to the above diagram, we get the following
exact and commutative diagram 
\[
\xymatrix{\D(Z,W)\ar[r]^{\D(c,W)}\ar[d]_{\D(h_{1},W)} & \D(Y,W)\ar[r]^{\D(f,W)}\ar[d]^{\D(h,W)} & \D(X,W)\ar@{=}[d]\\
\D(T_{1},W)\ar[r] & \D(E,W)\ar[r]^{\D(r,W)} & \D(X,W)
}
.
\]
Since $W\in\F_{1},$ we get that $\D(T_{1},W)=0$ and thus $\D(r,W):\D(E,W)\to\D(X,W)$
in a monomorphism. In particular, we conclude that $r:X\to E$ is
an epimorphism in $\mathcal{H}_{[\t_{1},\t_{2}]}.$ Let us show now
that $c_{2}=\Cok[f]$ in $\mathcal{H}_{[\t_{1},\t_{2}]}.$ Indeed,
let $\lambda:Y\to W$ be such that $\lambda\circ f=0.$ Then, by the
above commutative diagram and since $\D(r,W):\D(E,W)\to\D(X,W)$ in
a monomorphism, we conclude that $\lambda\circ h=0.$ Now, using that
$c_{2}=\Cok[h],$ we get that there is a unique factorization of $\lambda:Y\to W$
through $c_{2}:Y\to F_{1},$ proving that $c_{2}=\Cok[f]$ in $\mathcal{H}_{[\t_{1},\t_{2}]}.$ 
\end{proof}
We recall that an additive category $\C$ is \textbf{semi-abelian}
if it has kernels, cokernels and the canonical morphism $\coim[f]\to\im[f]$
is monic and epic. 
\begin{thm}
\label{semi-ab-triang} Let $\D$ be an extriangulated category with
negative extension and $[\t_{1},\t_{2}]$ be a normal interval in
$\stors\,\D$ such that every morphism $f:X\rightarrow Y$ in $\mathcal{H}_{[\t_{1},\t_{2}]}$
admits the following $\s$-conflations in $\mathcal{D}$: 
\[
\chi:\:\suc[X][Y][Z][f][c]\text{ and }\kappa:\:\suc[K][X][Y][k][f].
\]
Then, $\mathcal{H}_{[\t_{1},\t_{2}]}$ is an exact and a semi-abelian
category. 
\end{thm}

\begin{proof}
By Lemma \ref{nimpex}, we know that the extriangulated category $\mathcal{H}_{[\t_{1},\t_{2}]}$
is an exact category. By Proposition \ref{ex-ker-coker}, we get that
$\mathcal{H}_{[\t_{1},\t_{2}]}$ has kernels and cokernels. Moreover,
for $f:X\to Y$ in $\mathcal{H}_{[\t_{1},\t_{2}]},$ we have the following
commutative diagram in $\mathcal{H}_{[\t_{1},\t_{2}]}$ 
\[
\xymatrix{ &  & E_{1}\ar[dr]^{h_{1}}\\
\Ker[(f)]\ar[r]^{\quad k} & X\ar[ru]^{r_{1}}\ar[rd]_{r_{2}}\ar[rr]^{f} &  & Y\ar[r]^{c\quad\quad} & \Cok[f]\\
 &  & E_{2}\ar[ru]_{h_{2}}
}
\]
where $E_{1}\xrightarrow{h_{1}}Y\xrightarrow{c}\Cok[f]$ and $\Ker[(f)]\xrightarrow{k}X\xrightarrow{r_{2}}E_{2}$
are short exact sequences in $\mathcal{H}_{[\t_{1},\t_{2}]},$ $r_{1}$
is epic, and $h_{2}$ is monic. Therefore, since $h_{1}$ is monic
and $h_{1}\circ(r_{1}\circ k)=f\circ k=0,$ we have that $r_{1}\circ k=0$
and thus there is $\lambda:E_{2}\to E_{1}$ such that $\lambda\circ r_{2}=r_{1}.$
In particular $\lambda$ is epic since $r_{1}$ is epic. On the other
hand, from the equality $h_{1}\circ\lambda\circ r_{2}=h_{1}\circ r_{1}=f=h_{2}\circ r_{2},$
we get that $h_{1}\circ\lambda=h_{2}$ since $r_{2}$ is epic. But
the equality $h_{1}\circ\lambda=h_{2}$ implies that $\lambda$ is
monic since $h_{2}$ is monic. 
\end{proof}
In general, an exact and semi-abelian category does not have to be
quasi-abelian and neither abelian. In \cite[Ex. 1]{R2} an example
of this is given. However, in case the extriangulated category $\D$
is triangulated, the theorem above can be strengthened as follows. 
\begin{thm}
\label{triang-n-smab} Let $\D$ be a triangulated category and $[\t_{1},\t_{2}]$
be a normal interval in $\stors\,\D.$ Then $\mathcal{H}_{[\t_{1},\t_{2}]}$
is a quasi-abelian category. 
\end{thm}

\begin{proof}
Since $[\t_{1},\t_{2}]$ is normal, we have by Example \ref{exa:normal}(b)
that $\Sigma\t_{2}\leq\t_{1}\leq\t_{2}.$ Hence, by Lemma \ref{lem:stors de extensiones de corazones}(c),
it follows that $(\mathcal{H}_{[\Sigma\t_{2},\t_{1}]},\mathcal{H}_{[\t_{1},\t_{2}]})$
is a torsion pair in the abelian category $\mathcal{H}_{[\Sigma\t_{2},\t_{2}]}.$
Now, by \cite[Corollary, p.193]{R1}, we have that a torsion-free
class in an abelian category is quasi-abelian. Therefore, $\mathcal{H}_{[\t_{1},\t_{2}]}$
is quasi-abelian. 
\end{proof}

\subsection{Parameterizing subintervals }

Let $[\t_{1},\t_{2}]$ be an interval in $\stors\,\D.$ In Theorem
\ref{thm:proceso aet}, it is shown how the elements of $[\t_{1},\t_{2}]$
are parameterized with the $s$-torsion pairs in $\mathcal{H}_{[\t_{1},\t_{2}]}$.
The objective of this section will be to parameterize the objects
of a subinterval of $[\t_{1},\t_{2}]$ by means of the elements of
$s$-torsion pairs in $\mathcal{H}_{[\t_{1},\t_{2}]}$. 
\begin{lem}
\label{lem:new stors} Let $\D$ be an extriangulated category with
negative first extension and let $\t_{i}=(\T_{i},\F_{i})$ in $\stors\,\D,$
for $i=1,2,3$ be such that $\t_{1}\leq\t_{2}\leq\t_{3}.$ Then, for
$(\mathcal{X},\mathcal{Y})$ in $\stors\,\mathcal{H}_{[\t_{1},\t_{3}]},$
the following statements hold true. 
\begin{enumerate}
\item $\t_{2}(\mathcal{Y})\subseteq\mathcal{H}_{[\t_{1},\t_{2}]}$ and $(1:\t_{2})(\mathcal{Y})\subseteq\mathcal{H}_{[\t_{2},\t_{3}]}$. 
\item $\t_{2}(\mathcal{\mathcal{X}})\subseteq\mathcal{H}_{[\t_{1},\t_{2}]}$
and $(1:\t_{2})(\mathcal{\mathcal{X}})\subseteq\mathcal{H}_{[\t_{2},\t_{3}]}$. 
\item If $[\t_{2},\t_{3}]$ is normal, then $\t_{2}(\mathcal{Y})\subseteq\mathcal{Y}$
and $(1:\t_{2})(\mathcal{Y})\subseteq\mathcal{X}^{\bot_{-1}}$. 
\item If $[\t_{1},\t_{2}]$ is normal, then $\left(1:\t_{2}\right)(\mathcal{X})\subseteq\mathcal{X}$
and $\t_{2}(\mathcal{X})\subseteq{}{}^{\bot_{-1}}\mathcal{Y}$. 
\item the pair $\t=(\mathcal{T}_{1}\star\mathcal{X},\mathcal{Y}\star\mathcal{F}_{3})$
is an $s$-torsion pair in $\mathcal{D}$ such that $\t_{1}\leq\t\leq\t_{3}$. 
\end{enumerate}
\end{lem}

\begin{proof}
Items (a) and (b) follow from Lemma \ref{lem:extensions of hearts}.
\
 (c) Let $C\in\mathcal{Y}$. On the one hand, for any $X\in\mathcal{X}$,
we have the exact sequence 
\[
\mathbb{E}^{-1}(X,F_{2})\rightarrow\mathcal{D}(X,T_{2})\rightarrow\mathcal{D}(X,C).
\]
Here, observe that: $\mathcal{D}(X,C)=0$ since $C\in\mathcal{Y}$;
and $\mathbb{E}^{-1}(X,F_{2})=0$ since $X\in\mathcal{T}\subseteq\mathcal{T}_{3}$,
$F_{2}\in\mathcal{F}_{2}$ and $[\t_{2},\t_{3}]$ is normal. And thus,
$\t_{2}(C)=T_{2}\in\mathcal{Y}$ by Proposition \ref{prop:clases ort en stors}.
\\
 On the other hand, note that $F_{2}\in\mathcal{X}^{\bot_{-1}}$ by
Proposition \ref{lem:cerraduras stors} since $F_{2}\in\mathcal{H}_{[\t_{2},\t_{3}]}\subseteq\mathcal{H}_{[\t_{1},\t_{3}]}$
(see (a) and Lemma \ref{lem:extensions of hearts}), $C\in\mathcal{Y}$
and $T_{2}\in\mathcal{Y}$. \
 (d) It follows with similar arguments as (c). \
 (e) It follows from Theorem \ref{thm:proceso aet}. 
\end{proof}
\begin{cor}
\label{cor:new stors vs interval} Let $\D$ be an extriangulated
category with negative first extension and let $\t_{i}=(\T_{i},\F_{i})$
in $\stors\,\D,$ for $i=1,2,3$ be such that $\t_{1}\leq\t_{2}\leq\t_{3}.$
Then, for $(\mathcal{X},\mathcal{Y})$ in $\stors\,\mathcal{H}_{[\t_{1},\t_{3}]}$
and $\t=(\mathcal{T},\mathcal{F}):=(\mathcal{T}_{1}\star\mathcal{X},\mathcal{Y}\star\mathcal{F}_{3})\in[\t_{1},\t_{3}],$
see Lemma \ref{lem:new stors} (e), the following statements hold
true. 
\begin{enumerate}
\item $\t\in[\t_{1},\t_{2}]$ $\Leftrightarrow$ $\mathcal{X}\subseteq\mathcal{H}_{[\t_{1},\t_{2}]}$
$\Leftrightarrow$ $\mathcal{H}_{[\t_{2},\t_{3}]}\subseteq\mathcal{Y}.$ 
\item $\t\in[\t_{2},\t_{3}]$ $\Leftrightarrow$ $\mathcal{H}_{[\t_{1},\t_{2}]}\subseteq\mathcal{X}$
$\Leftrightarrow$ $\mathcal{Y}\subseteq\mathcal{H}_{[\t_{2},\t_{3}]}.$ 
\end{enumerate}
\end{cor}

\begin{proof}
(a) By Lemma \ref{lem:stors de extensiones de corazones} (c), we
have that $(\mathcal{H}_{[\t_{1},\t_{2}]},\mathcal{H}_{[\t_{2},\t_{3}]})\in\stors\,\mathcal{H}_{[\t_{1},\t_{3}]}$
and since $(\mathcal{X},\mathcal{Y})$ in $\stors\,\mathcal{H}_{[\t_{1},\t_{3}]},$
we get by Proposition \ref{lem:cerraduras stors} (d) that $\mathcal{X}\subseteq\mathcal{H}_{[\t_{1},\t_{2}]}$
$\Leftrightarrow$ $\mathcal{H}_{[\t_{2},\t_{3}]}\subseteq\mathcal{Y}.$
Let now $\mathcal{T}_{1}\subseteq\mathcal{T}_{1}\star\mathcal{X}\subseteq\mathcal{T}_{2}$.
Hence $\mathcal{X}\subseteq\mathcal{T}_{2}$. Moreover $\mathcal{X}\subseteq\mathcal{H}_{[\t_{1},\t_{3}]}=\mathcal{F}_{1}\cap\mathcal{T}_{3}\subseteq\mathcal{F}_{1}$
and thus $\mathcal{X}\subseteq\mathcal{F}_{1}\cap\mathcal{T}_{2}=\mathcal{H}_{[\t_{1},\t_{2}]}$.\\
 Finally, assume that $\mathcal{X}\subseteq\mathcal{F}_{1}\cap\mathcal{T}_{2}$.
Hence $\mathcal{T}_{1}\star\mathcal{X}\subseteq\mathcal{T}_{2}$,
and so $\t\in[\t_{1},\t_{2}]$. \
 Lastly, (b) follows with dual arguments. 
\end{proof}
The following result is a generalization of Theorem \ref{thm:proceso aet}
and it is very useful to parameterize torsion pairs in quasi-abelian
categories, see Theorem \ref{param-stp-qac}. In order to do that,
we start with the following notion. 
\begin{defn}
Let $\D$ be an extriangulated category with negative first extension
and $\mathbf{x}\leq\mathbf{y}\leq\mathbf{z}$ in $\stors\,\D.$ We
consider the classes of $s$-torsion pairs: 
\begin{enumerate}
\item $\stors\,\mathcal{H}_{[\mathbf{x},\mathbf{y},\mathbf{z}]}^{+}:=\{(\T,\F)\in\stors\,\mathcal{H}_{[\mathbf{x},\mathbf{z}]}\;:\;\T\subseteq\mathcal{H}_{[\mathbf{x},\mathbf{y}]}\}.$ 
\item $\stors\,\mathcal{H}_{[\mathbf{x},\mathbf{y},\mathbf{z}]}^{-}:=\{(\T,\F)\in\stors\,\mathcal{H}_{[\mathbf{x},\mathbf{z}]}\;:\;\mathcal{H}_{[\mathbf{x},\mathbf{y}]}\subseteq\T\}.$ 
\end{enumerate}
\end{defn}

\begin{thm}
\label{thm:parametrizando subintervalos} For an extriangulated category
$\D$ with negative first extension and $\t_{i}=(\mathcal{T}_{i},\mathcal{F}_{i})$
in $\stors\,\mathcal{D}$ for $i=1,2,3,4$ such that $\t_{1}\leq\t_{2}\leq\t_{3}\leq\t_{4},$
the following statements hold. 
\begin{enumerate}
\item The map $\Phi_{r}:[\t_{2},\t_{3}]\rightarrow\stors\,\mathcal{H}_{[\t_{2},\t_{3},\t_{4}]}^{+},\;(\mathcal{T},\mathcal{F})\mapsto(\mathcal{T}\cap\mathcal{F}_{2},\mathcal{T}_{4}\cap\mathcal{F}),$
is an isomorphism of posets with inverse $(\mathcal{X},\mathcal{Y})\mapsto(\mathcal{T}_{2}\star\mathcal{X},\mathcal{Y}\star\mathcal{F}_{4})$. 
\item The map $\Phi_{\ell}:[\t_{2},\t_{3}]\rightarrow\stors\,\mathcal{H}_{[\t_{1},\t_{2},\t_{3}]}^{-},\;(\mathcal{T},\mathcal{F})\mapsto(\mathcal{T}\cap\mathcal{F}_{1},\mathcal{T}_{3}\cap\mathcal{F}),$
is an isomorphism of posets with inverse $(\mathcal{X},\mathcal{Y})\mapsto(\mathcal{T}_{1}\star\mathcal{X},\mathcal{Y}\star\mathcal{F}_{3})$. 
\item The map $\varepsilon:\stors\,\mathcal{H}_{[\t_{2},\t_{3},\t_{4}]}^{+}\to\stors\,\mathcal{H}_{[\t_{1},\t_{2},\t_{3}]}^{-}$
given by 
\begin{center}
$(\X,\Y)\mapsto((\T_{2}*\X)\cap\F_{1},(\Y*\F_{4})\cap\T_{3}),$ 
\par\end{center}

is an isomorphism of posets with inverse 
\begin{center}
$(\T,\F)\mapsto((\T_{1}*\T)\cap\F_{2},(\F*\F_{3})\cap\T_{4}).$ 
\par\end{center}

\end{enumerate}
\end{thm}

\begin{proof}
We prove only (a) since (b) is similar and (c) follows from (a) and
(b). \
 From Theorem \ref{thm:proceso aet}, we have the isomorphism of posets
\[
\Phi:[\t_{2},\t_{4}]\rightarrow\stors\,\mathcal{H}_{[\t_{2},\t_{4}]},\quad(\mathcal{T},\mathcal{F})\mapsto(\mathcal{T}\cap\mathcal{F}_{2},\mathcal{T}_{4}\cap\mathcal{F}),
\]
with inverse 
\[
\Psi:\stors\,\mathcal{H}_{[\t_{2},\t_{4}]}\to[\t_{2},\t_{4}],\quad(\mathcal{X},\mathcal{Y})\mapsto(\mathcal{T}_{2}\star\mathcal{X},\mathcal{Y}\star\mathcal{F}_{4}).
\]
Since $[\t_{2},\t_{3}]\subseteq[\t_{2},\t_{4}],$ by Corollary \ref{cor:new stors vs interval},
we have that 
\[
\Psi^{-1}[\t_{2},\t_{3}]=\left\{ (\mathcal{X},\mathcal{Y})\in\stors\,\mathcal{H}_{[\t_{2},\t_{4}]}\,|\:\mathcal{X}\subseteq\mathcal{H}_{[\t_{2},\t_{3}]}\right\} =\stors\,\mathcal{H}_{[\t_{2},\t_{3},\t_{4}]}^{+}.
\]
Therefore, by restricting $\Phi$ and $\Psi$, we get the desired
isomorphism. 
\end{proof}
As an application, we have the following result parameterizing the
torsion pairs of a quasi-abelian category. We recall that in an exact
category $\D,$ see \cite[Prop. 2.6]{AET}, the class $\stors\,\D$
is just the class of all the torsion pairs in $\D.$ 
\begin{thm}
\label{param-stp-qac} Let $\mathcal{E}$ be a quasi-abelian category,
and let $\mathcal{L}_{\mathcal{E}}$ and $\mathcal{R}_{\mathcal{E}}$
be the associated abelian categories of $\mathcal{E}$ (see Example
\ref{exa:categorias abelianas asociadas a quasiabeliana}). Then,
the following posets are isomorphic. 
\begin{enumerate}
\item $\stors\,\mathcal{E}.$ 
\item $\left\{ (\mathcal{X},\mathcal{Y})\in\stors\,\mathcal{L}_{\mathcal{E}}\,|\:\mathcal{Y}\subseteq\mathcal{E}\right\} .$ 
\item $\left\{ (\mathcal{X},\mathcal{Y})\in\stors\,\mathcal{R}_{\mathcal{E}}\,|\:\mathcal{X}\subseteq\mathcal{E}\right\} .$ 
\end{enumerate}
\end{thm}

\begin{proof}
Let $\mathcal{D}$ be the triangulated category associated to $\mathcal{E}$.
Recall that there are $\t_{r},\t_{\ell}\in\stors\,\mathcal{D}$ such
that $\Sigma\t_{r}\leq\t_{\ell}\leq\t_{r}$ and $\mathcal{E}\cong\mathcal{H}_{[\Sigma\t_{r},\t_{\ell}]}$
(see Example \ref{exa:quasi-abelian}). Then, the sought isomorphisms
are given by Theorem \ref{thm:parametrizando subintervalos} (under
the setting: $\t_{1}:=\Sigma\t_{\ell}$, $\t_{2}:=\Sigma\t_{r}$,
$\t_{3}:=\t_{\ell}$ and $\t_{4}:=\t_{r}$), together with the isomorphism
$[\t_{2},\t_{3}]\cong\stors\,\mathcal{H}_{[\t_{2},\t_{3}]}$ given
by Theorem \ref{thm:proceso aet}. 
\end{proof}

\section{The extended heart of a $t$-structure}


The objective of this section is to study a special kind of extensions
of hearts in triangulated categories. The advantage that triangulated
categories have over the rest of the extriangulated categories with
negative first extensions is that the functor $\mathbb{E}^{-1}$ coincides
with $\mathcal{D}(-,\Sigma^{-1}-)$. So we have that $\mathcal{X}^{\bot_{-1}}=(\Sigma\mathcal{X})^{\bot_{0}}$
and $^{\bot_{0}}(\Sigma^{-1}\mathcal{X})={}^{\bot_{-1}}\mathcal{X}$.
This fact is very useful when studying $t$-structures. We remind
the reader that a pair of classes $(\mathcal{U},\mathcal{W})$ is
a $t$-structure if, and only if, $(\mathcal{U},\Sigma^{-1}\mathcal{W})$
is an $s$-torsion pair, or equivalently $(\Sigma\mathcal{U},\mathcal{W})$
is an $s$-torsion pair. 
\begin{defn}
\label{cor-extendido} Let $\mathcal{D}$ be a triangulated category
and $\u=(\mathcal{U},\mathcal{W})$ be a $t$-structure in $\D.$
For $\mathbf{u}_{1}:=(\Sigma\mathcal{U},\mathcal{W}),$ $\mathbf{u}_{2}:=(\mathcal{U},\Sigma^{-1}\mathcal{W})$
and $\u_{3}:=(\Sigma^{-1}\mathcal{U},\Sigma^{-2}\mathcal{W}),$ we
have $\u_{1}\leq\u_{2}\leq\u_{3}$ in $\stors\,\D.$ Notice that $\Hcal:=\mathcal{H}_{[\u_{1},\u_{2}]}=\W\cap\U$
is the heart $\Hcal_{\u}$ of $\u.$ The class $\C=\C_{\u}:=\Hcal_{[\u_{1},\u_{3}]}=\W\cap\Sigma^{-1}\U$
is \textbf{the extended heart of $\u$}. Recall that, for any $C\in\mathcal{D}$,
we have the $\mathfrak{s}$-conflation $\tau_{\U}^{\leq}(C)\stackrel{f_{C}}{\to}C\stackrel{g_{C}}{\to}\tau_{\U}^{>}(C)$
given by $\u_{2}$. Thus, the torsion pair $\u_{2}$ induces the torsion
functor $\tau_{\U}^{\leq}:\mathcal{D}\rightarrow\mathcal{U}$ and
the torsion-free functor $\tau_{\U}^{>}:\mathcal{D}\rightarrow\Sigma^{-1}\mathcal{W}.$ 
\end{defn}

We recall that the heart $\Hcal=\Hcal_{\u}$ is an abelian category
(see \cite{BBD}). On the other hand, since the extended heart $\C=\C_{\u}$
is closed under extensions in the triangulated category $\D,$ we
have that $\C$ is an extriangulated category with negative first
extension. In what follows, we show some basic properties of the extended
heart $\C.$ 
\begin{prop}
\label{basic-p-C} For a triangulated category $\D,$ a $t$-structure
$\u=(\mathcal{U},\mathcal{W})$ in $\D,$ the heart $\Hcal=\Hcal_{\u}$
and the extended heart $\C=\C_{\u},$ the following statements hold
true. 
\begin{enumerate}
\item $\mathcal{C}=\Hcal\star\Sigma^{-1}\Hcal,$ $\tau_{\U}^{\leq}(\C)\subseteq\mathcal{H}$
and $\tau_{\U}^{>}(\C)\subseteq\Sigma^{-1}\mathcal{H}.$ 
\item If $\mathbf{t}=(\T,\F)$ is a torsion pair in $\Hcal,$ then $(\T,\F\star\Sigma^{-1}\Hcal)$,
$(\mathcal{H}\star\Sigma^{-1}\mathcal{T},\Sigma^{-1}\mathcal{F})$
and $(\Hcal,\Sigma^{-1}\Hcal)$ are $s$-torsion pairs in $\C$. 
\item For any $C\in\C,$ the $\mathfrak{s}$-conflation $\tau_{\U}^{\leq}(C)\stackrel{f_{C}}{\to}C\stackrel{g_{C}}{\to}\tau_{\U}^{>}(C)$
given by $\u_{2}$ (see Definition \ref{cor-extendido}) is a short
exact sequence in $\C.$ 
\item If $\Hcal\neq0,$ then the extriangulated category $\C$ is not abelian. 
\end{enumerate}
\end{prop}

\begin{proof}
Observe that $\Hcal_{[\u_{2},\u_{3}]}=\Sigma^{-1}\Hcal,$ $\Hcal_{[\u_{1},\u_{2}]}=\Hcal$
and $\C=\Hcal_{[\u_{1},\u_{3}]}.$ Then: (a) follows from Lemma \ref{lem:extensions of hearts},
(b) from Lemma \ref{lem:stors de extensiones de corazones} and (c)
from Proposition \ref{prop:normal vs exactitud}. \
 (d) Let $0\neq H\in\Hcal.$ Consider the canonical triangle $\Sigma^{-1}H\to0\to H\to H$
in the triangulated category $\D.$ Then, by (a) we get the $\mathfrak{s}$-conflation
$\eta:\;\Sigma^{-1}H\to0\to H$ in $\C.$ Suppose that $\C$ is abelian.
Then $\eta$ is a short exact sequence in $\C$ and thus $H=0,$ which
is a contradiction proving that $\C$ is not abelian. 
\end{proof}
\begin{defn}
\label{DefTu} For a triangulated category $\D,$ a $t$-structure
$\u=(\mathcal{U},\mathcal{W})$ in $\D,$ and the extended heart $\C=\C_{\u}$,
we introduce the following classes of $t$-structures in $\D.$ 
\begin{enumerate}
\item $\mathbb{T}_{\u}$ is the class of all the $t$-structures $(\U',\W')$
in $\D$ such that $\Sigma\,\U\subseteq\U'\subseteq\Sigma^{-1}\U.$ 
\item $\mathbb{T}_{\u}^{+}:=\{(\U',\W')\in\mathbb{T}_{u}\;:\:\U'\subseteq\U\}.$ 
\item $\mathbb{T}_{\u}^{-}:=\{(\U',\W')\in\mathbb{T}_{u}\;:\:\U\subseteq\U'\}.$ 
\end{enumerate}
\end{defn}

\begin{lem}
\label{new-t-structure} For a triangulated category $\D,$ a $t$-structure
$\u=(\mathcal{U},\mathcal{W})$ in $\D,$ the extended heart $\C=\C_{\u}$
and an $s$-torsion pair $(\X,\Y)$ in $\C,$ the following statements
hold true. 
\begin{enumerate}
\item $\tau_{\U}^{\leq}(\Y)\subseteq\Y,$ $\tau_{\U}^{>}(\Y)\subseteq\Sigma\,\Y,$
$\tau_{\U}^{>}(\X)\subseteq\X$ and $\tau_{\U}^{\leq}(\X)\subseteq\Sigma^{-1}\X.$ 
\item $(\Sigma\U\star\X,\Sigma\Y\star\Sigma^{-1}\W)\in\mathbb{T}_{\u}.$ 
\end{enumerate}
\end{lem}

\begin{proof}
It follows from Lemma \ref{lem:new stors} since the intervals $[\u_{1},\u_{2}]$
and $[\u_{2},\u_{3}],$ considered in Definition \ref{cor-extendido},
are normal. 
\end{proof}
\begin{defn}
\label{setting-new-t-structures} Let $\u=(\U,\W)$ be a $t$-structure
in the triangulated category $\D,$ and let $\C=\C_{\u}$ be the extended
heart of $\u$. For an $s$-torsion pair $\mathbf{x}=(\X,\Y)$ in
$\C,$ we define the \textbf{extended tilt} $\u_{\mathbf{x}}=\Et_{\u}({\mathbf{x}})$
of $\u$ with respect to $\mathbf{x}$ as follows $\u_{\mathbf{x}}:=(\mathcal{U}_{\mathbf{x}},\mathcal{W}_{\mathbf{x}}):=(\Sigma\U\star\X,\Sigma\Y\star\Sigma^{-1}\W).$
Thus, by Lemma \ref{new-t-structure}(b), we get the map $\Et_{\u}:\stors\,\C\to\mathbb{T}_{\u},\;(\X,\Y)\mapsto\Et_{u}(\X,\Y)$
which is a morphism of posets. 
\end{defn}

\begin{cor}
\label{cor. X subseteq H} Let $\u=(\U,\W)$ be a $t$-structure in
the triangulated category $\D,$ and let $\C=\C_{\u}$ and $\Hcal=\Hcal_{\u}$
be, respectively, the extended heart and the heart of $\u.$ Then,
for an $s$-torsion pair $\mathbf{x}=(\X,\Y)$ in $\C,$ the extended
tilt $\u_{\mathbf{x}}:=(\mathcal{U}_{\mathbf{x}},\mathcal{W}_{\mathbf{x}}),$
satisfies the following conditions. 
\begin{enumerate}
\item $\Sigma\mathcal{U}\subseteq\mathcal{U}_{\mathbf{x}}\subseteq\mathcal{U}$
$\Leftrightarrow$ $\mathcal{X}\subseteq\mathcal{H}$ $\Leftrightarrow$
$\Sigma^{-1}\Hcal\subseteq\Y.$ 
\item $\mathcal{U}\subseteq\mathcal{U}_{\mathbf{x}}\subseteq\Sigma^{-1}\mathcal{U}$
$\Leftrightarrow$ $\mathcal{H}\subseteq\mathcal{X}$ $\Leftrightarrow$
$\Y\subseteq\Sigma^{-1}\Hcal.$ 
\end{enumerate}
\end{cor}

\begin{proof}
It follows from Corollary \ref{cor:new stors vs interval}. 
\end{proof}
In the following result, we show that the extended tilt $\Et_{\u}$
gives an isomorphism of posets between the $s$-torsion pairs in the
extended heart $\C=\C_{\u}$ and some special class of $t$-structures.
Moreover, such a bijection is well restricted on certain subposets
described in the items (a) and (b) below. 
\begin{cor}
\label{corresp-general} Let $\u=(\mathcal{U},\mathcal{W})$ be a
t-structure in the triangulated category $\mathcal{D}$ with heart
$\mathcal{H}=\Hcal_{\u}$ and extended heart $\mathcal{C}=\C_{\u},$
and let $\mathbb{T}_{\u}^{\pm}$ be the classes described in Definition
\ref{DefTu}. Consider the following classes 
\begin{enumerate}
\item $\stors\,\C^{+}:=\{(\X,\Y)\in\stors\,\C\;:\;\X\subseteq\Hcal\};$ 
\item $\stors\,\C^{-}:=\{(\X,\Y)\in\stors\,\C\;:\;\Hcal\subseteq\X\},$ 
\end{enumerate}
Then, the extended tilt $\Et_{\u}:\stors\,\C\to\mathbb{T}_{\u}$ is
an isomorphism of posets whose inverse is given by $(\U',\W')\mapsto(\U'\cap\W,\Sigma^{-1}(\W'\cap\U)).$
Moreover, $\Et_{\u}(\stors\,\C^{+})=\mathbb{T}_{\u}^{+}$ and $\Et_{\u}(\stors\,\C^{-})=\mathbb{T}_{\u}^{-}.$ 
\end{cor}

\begin{proof}
It follows from Theorems \ref{thm:proceso aet} and \ref{thm:parametrizando subintervalos}. 
\end{proof}
Now, we want to relate the extended tilt $\Et_{\u}$ with the right
and left Happel-Reiten-Smal{\o} tilt in the heart of a $t$-structure
$\u.$ 
\begin{defn}
\cite{HRS,W} Let $\u=(\mathcal{U},\mathcal{W})$ be a $t$-structure
in the triangulated category $\mathcal{D},$ with heart $\mathcal{H}=\Hcal_{\u}=\U\cap\W.$ 
\begin{enumerate}
\item The \textbf{right HRS-tilt} of $\mathbf{t}\in\stors\,\Hcal,$ with
respect to $\u,$ is $\R_{\u}(\mathbf{t}),$ where 
\[
\R_{\u}:\stors\,\Hcal\to\mathbb{T}_{\u}^{+},\;\mathbf{t}=(\mathcal{T},\mathcal{F})\mapsto(\Sigma\mathcal{U}\star\mathcal{T},\Sigma\mathcal{F}\star\mathcal{W}).
\]
\item The \textbf{left HRS-tilt} of $\mathbf{t}\in\stors\,\Hcal,$ with
respect to $\u,$ is $\Le_{\u}(\mathbf{t}),$ where 
\[
\Le_{\u}:\stors\,\Hcal\to\mathbb{T}_{\u}^{-},\;\mathbf{t}=(\mathcal{T},\mathcal{F})\mapsto(\mathcal{U}\star\Sigma^{-1}\mathcal{T},\mathcal{F}\star\Sigma^{-1}\mathcal{W}).
\]
Notice that $\Sigma^{-1}:\mathbb{T}_{\u}^{+}\to\mathbb{T}_{\u}^{-}$
is a bijection and $\Sigma^{-1}\R_{\u}=\Le_{\u}.$ 
\end{enumerate}
\end{defn}

Now, for a given $t$-structure $\u,$ we describe the relationship
between the extended tilt $\Et_{u}$ and the right HRS-tilt $\R_{u}.$ 
\begin{thm}
\label{ext-tilt-HRS-p} Let $\u=(\mathcal{U},\mathcal{W})$ be a $t$-structure
in $\mathcal{D}$ with heart $\Hcal$ and extended heart $\mathcal{C},$
and let $\Et_{\u}^{+}:\stors\,\C^{+}\to\mathbb{T}_{\u}^{+}$ be the
isomorphism of posets given by the restriction of the extended tilt
$\Et_{\u}$ on $\stors\,\C^{+}$ (see Corollary \ref{corresp-general}).
Then, the following statements hold true. 
\begin{enumerate}
\item The map $\mu^{+}:\stors\,\Hcal\to\stors\,\C^{+},\,(\T,\F)\mapsto(\T,\F*\Sigma^{-1}\Hcal),$
is an isomorphism of posets whose inverse is given by $\lambda^{+}:\stors\,\C^{+}\to\stors\,\Hcal,\,(\X,\Y)\mapsto(\X,\Y\cap\Hcal).$
Furthermore $\Y=(\Y\cap\Hcal)*\Sigma^{-1}\Hcal$ and $\F=(\F*\Sigma^{-1}\Hcal)\cap\Hcal.$ 
\item $\Et_{\u}^{+}\circ\mu^{+}=\R_{u}.$ 
\end{enumerate}
\end{thm}

\begin{proof}
(a) Notice that the map $\mu^{+}$ is well defined (see Proposition
\ref{basic-p-C} (b)). Let us prove that $\lambda^{+}$ is well defined.
Indeed, for $(\X,\Y)\in\stors\,\C^{+},$ we need to show that $\Hcal=\X*(\Y\cap\Hcal).$
However, this equality can be obtained from Proposition \ref{lemaA}(a).
Finally, the equalities in item (a) can be obtained from Proposition
\ref{prop:clases ort en stors}. \
 (b) Let $(\T,\F)\in\stors\,\Hcal.$ Then $\Et_{\u}^{+}\,\mu^{+}(\T,\F)=(\Sigma\mathcal{U}\star\mathcal{T},(\Sigma\F*\Hcal)*\Sigma^{-1}\W);$
and thus $\Et_{\u}^{+}\,\mu^{+}(\T,\F)=\R_{u}(\T,\F)$ since $(\Sigma\F*\Hcal)*\Sigma^{-1}\W=\Sigma\F*\W.$ 
\end{proof}
Now, for a given $t$-structure $\u,$ we describe the relationship
between the extended tilt $\Et_{u}$ and the left HRS-tilt $\Le_{u}.$ 
\begin{thm}
\label{ext-tilt-HRS-n} Let $\u=(\mathcal{U},\mathcal{W})$ be a $t$-structure
in $\mathcal{D}$ with heart $\Hcal$ and extended heart $\mathcal{C},$
and let $\Et_{\u}^{-}:\stors\,\C^{-}\to\mathbb{T}_{\u}^{-}$ be the
isomorphism of posets given by the restriction of the extended tilt
$\Et_{\u}$ on $\stors\,\C^{-}$ (see Corollary \ref{corresp-general}).
Then, the following statements hold true. 
\begin{enumerate}
\item The map $\mu^{-}:\stors\,\Sigma^{-1}\Hcal\to\stors\,\C^{-},\,(\T,\F)\mapsto(\Hcal*\T,\F),$
is an isomorphism of posets whose inverse is given by 
\[
\lambda^{-}:\stors\,\C^{-}\to\stors\,\Sigma^{-1}\Hcal,\,(\X,\Y)\mapsto(\X\cap\Sigma^{-1}\Hcal,\Y).
\]
Furthermore, $\X=\Hcal*(\X\cap\Sigma^{-1}\Hcal),$ $\T=(\Hcal*\T)\cap\Sigma^{-1}\Hcal$
and $\stors\,\C^{-}=\{(\X,\Y)\in\stors\,\C\;:\;\Y\subseteq\Sigma^{-1}\Hcal\}.$ 
\item We have the following commutative diagram of posets 
\[
\xymatrix{\stors\,\C^{-}\ar[r]^{\Et_{\u}^{-}} & \mathbb{T}_{\u}^{-}\\
\stors\,\Sigma^{-1}\Hcal\ar[u]^{\mu^{-}} & \stors\,\Hcal\ar[l]^{\quad\Sigma^{-1}}\ar[u]_{\Le_{\u}},
}
\]
where each arrow is a poset isomorphism. That is, $\Et_{\u}^{-}\circ\mu^{-}\circ\Sigma^{-1}=\,\Le_{u}.$ 
\end{enumerate}
\end{thm}

\begin{proof}
(a) From Corollary \ref{cor. X subseteq H} (b), we get that 
\[
\stors\,\C^{-}=\{(\X,\Y)\in\stors\,\C\;:\;\Y\subseteq\Sigma^{-1}\Hcal\}.
\]
The above equality will be a strong point to prove the item (a). Indeed,
from Lemma \ref{lem:stors de extensiones de corazones} (b), we have
that the map $\mu^{-}$ is well defined. Let us prove that $\lambda^{-}$
is also well defined. Indeed, for $(\X,\Y)\in\stors\,\C^{-},$ we
need to show that $\Sigma^{-1}\Hcal=(\X\cap\Sigma^{-1}\Hcal)*\Y.$
However, this equality can be obtained from Proposition \ref{lemaA}(b).
Finally, the equalities in item (a) can be obtained from Proposition
\ref{prop:clases ort en stors}. \
 (b) Let $(\T,\F)\in\stors\,\Hcal.$ Then 
\[
\Et_{\u}^{-}\circ\mu^{-}\circ\Sigma^{-1}(\T,\F)=(\Sigma\mathcal{U}\star(\Hcal*\Sigma^{-1}\mathcal{T}),\F*\Sigma^{-1}\W);
\]
and thus $\Et_{\u}^{-}\circ\mu^{-}\circ\Sigma^{-1}(\T,\F)=\Le_{u}(\T,\F)$
since $\Sigma\mathcal{U}\star(\Hcal*\Sigma^{-1}\mathcal{T})=\mathcal{U}\star\Sigma^{-1}\mathcal{T}.$ 
\end{proof}
\bibliographystyle{plain}

\begin{thebibliography}{10}
\bibitem{AET} T. Adachi, H. Enomoto, M. Tsukamoto. Intervals of $s$-torsion
pairs in extriangulated categories with negative first extensions.\emph{
Math. Proc. Camb. Philos. Soc., }174(3):451--469, 2023.

\bibitem{AC} L. Angeleri H{\"u}gel, F. U. Coelho. Infinitely generated
tilting modules of finite projective dimension. \emph{Forum Math.,}
13(2):239--250, 2001.

\bibitem{AMV} L. Angeleri H{\"u}gel, F. Marks, J. Vit{\'o}ria.
Silting modules. \emph{ Int. Math. Res. Not.} 2016(4):1251--1284,
2016.

\bibitem{ABP} A. Argud{\'i}n-Monroy, D. Bravo, C.E. Parra. TTF classes
generated by silting modules. \emph{Preprint}, arXiv:2411.17581, 2024.

\bibitem{AMP}A. Argud{\'i}n-Monroy, O. Mendoza, C.E. Parra. Recollements,
coproducts and products in extriangulated categories. \emph{Preprint},
arXiv:2511.09595v1, 2025.

\bibitem{AP} A. Argud{\'i}n-Monroy, C.E. Parra. Universal co-extensions
of torsion abelian groups. \emph{J. Algebra,} 647:1--27, 2024.

\bibitem{AB} M. Auslander, R.-O. Buchweitz. The homological theory
of maximal Cohen-Macaulay approximations. \emph{M{\'e}m. Soc. Math.
Fr., Nouv. S{\'e}r. }38:5--37, 1987.

\bibitem{B} S. Bazzoni, A characterization of $n$-cotilting and
$n$-tilting modules. \emph{J. Algebra,} 273(1):359--372, 2004.

\bibitem{BBD} A. Beilinson, J. Bernstein, P. Deligne. Faisceaux pervers.
\emph{Asterisque,} 100, 172 p. 1982.

\bibitem{BGLS} R. Bennett-Tennenhaus, I. Goodbody, J. Letz, A. Shah.
Tensor extriangulated categories. \emph{J. Algebra}, 685(1): 361--405,
2026

\bibitem{Buhler} T. B{\"u}hler. Exact categories. \emph{Expo. Math.}
28(1): 1--69, 2010. 

\bibitem{CHZ} X.-W. Chen, Z. Han, Y. Zhou. Derived equivalences via
HRS-tilting. \emph{Adv. Math.} 354, Article ID 106749, 26 p., 2019.

\bibitem{CW} A. Cipriani, J. Woolf. When are there enough projective
perverse sheaves? \emph{Glasg. Math. J.} 64(1):185--196, 2022.

\bibitem{CGM} R. Colpi, E, Gregorio, F. Mantese. On the heart of
a faithful torsion theory. \emph{J. Algebra}, 307(2):841--863, 2007. 

\bibitem{F} L. Fiorot. $n$-quasi-abelian categories vs $n$-tilting
torsion pairs. \emph{Doc. Math}. 26:149--197, 2021.

\bibitem{FS} L. Frerick, D. Sieg. Exact categories in functional
analysis. \emph{Preprint}. 2010

\bibitem{G} M. Gorsky, H. Nakaoka, Y. Palu. Positive and negative
extensions in extriangulated categories. \emph{Preprint}, arXiv:2103.12482
{[}math.CT{]}, 2021.

\bibitem{HRS} D. Happel, I. Reiten, S.O. Smal{\o}. \emph{Tilting
in abelian categories and quasitilted algebras}. Providence, RI: American
Mathematical Society (AMS), 1996.

\bibitem{HHZ} J. He, Y. Hu, P. Zhou. Torsion pairs and recollements
of extriangulated categories. \emph{Comm. Algebra}, 50(5):2018--2036,
2022. DOI: 10.1080/00927872.2021.1996585

\bibitem{KSZ} F. Kong, K. Song, P. Zhang. Decomposition of torsion
pairs on module categories.\emph{ J. Algebra, }388:248--267, 2013.

\bibitem{K} H. Krause, \emph{Homological theory of representations}.
Cambridge: Cambridge University Press, 2022

\bibitem{LN} Y. Liu, H. Nakaoka. Hearts of twin cotorsion pairs on
extriangulated categories. \emph{J. Algebra}, 528:96--149, 2019. 

\bibitem{M} B. Mitchell, \emph{Theory of Categories}. Academic Press,
1965.

\bibitem{NP} H. Nakaoka, Y. Palu. Extriangulated categories, Hovey
twin cotorsion pairs and model structures, \textit{Cah. Topol. Geom.
Differ. Categ}\emph{.,} 60(2):117--193, 2019.

\bibitem{N} A. Neeman. \emph{Triangulated categories}. Princeton,
NJ: Princeton University Press, 2001.

\bibitem{PS1} C.E. Parra, M. Saor{\'i}n. Direct limits in the heart
of t-structure: the case of a torsion pair, \emph{J. Pure Appl. Algebra,}
219(9):4117--4143, 2015.

\bibitem{PS2}C.E. Parra, M. Saor{\'i}n, On hearts which are module
categories. \emph{J. Math. Soc. Japan,} 68(4):1421--1460, 2016.

\bibitem{PS3}C.E. Parra, M. Saor{\'i}n, The {HRS} tilting process
and {Grothendieck} hearts of t-structures. \emph{Representations
of algebras, geometry and physics, Maurice Auslander distinguished
lectures and international conference, Woods Hole Oceanographic Institute,
Woods Hole, MA, USA, April 25--30, 2018.} Providence, RI: American
Mathematical Society (AMS):209--241, 2021.

\bibitem{PSV2} C.E. Parra, M. Saor{\'i}n, S. Virili. Tilting preenvelopes
and cotilting precovers in general abelian categories. \emph{Algebr.
Represent. Theory.} 26(4):1087--1140, 2023.

\bibitem{PSV3}C.E. Parra, M. Saor{\'i}n, S. Virili. Locally finitely
presented and coherent hearts. \emph{Rev. Mat. Iberoam.} 39(1):201--268,
2023.

\bibitem{R} W. Rump. $*$-modules, tilting, and almost abelian categories.
\emph{Comm. Algebra,} 29(8): 3293--3325, 2001.

\bibitem{R1} W. Rump. Almost abelian categories. \emph{Cah. Topologie
G{\'e}om. Diff{\'e}r. Cat{\'e}goriques, }42(3):163--225, 2001.

\bibitem{R2}W. Rump. A counterexample to Raikov's conjecture. \emph{Bull.
Lond. Math. Soc}., 40(6):985--994, 2008.

\bibitem{S} J.P. Schneiders. \emph{Quasi-abelian categories and sheaves}.
SMF, Paris: M�m. Soc. Math. Fr., Nouv. S�r, 1998.

\bibitem{T} A. Tattar. Torsion pairs and quasi-abelian categories.
\emph{Algebr. Represent. Theory,} 24(6):1557--1581, 2021.

\bibitem{W} J. Woolf. Stability conditions, torsion theories and
tilting, \emph{ J. Lond. Math. Soc.} 82, 3 (2010), pp. 663--682

\bibitem{ZZ} B. Zhu, X. Zhuang. Tilting subcategories in extriangulated
categories. \emph{Front. Math. China.} 15(1):225--253, 2020. 
\end{thebibliography}

\section{Extended tilting objects}

Let us begin by introducing the following notation inspired in the
abelian case, see \cite{PS1,AB}. 
\begin{defn}
Let $\D$ be an extriangulated category. We say that $\mathcal{S}\subseteq\D$
is a \textbf{cogenerating class} in $\D$ if, for each $X\in\D$,
there is an $\E$-triangle $\suc[X][S][D]$ with $S\in\mathcal{S}$. 
\end{defn}

\begin{defn}
Let $\D$ be an extriangulated category and $V\in\D$. 
\begin{enumerate}
\item We say that $X\in\D$ is \textbf{$V$-generated} if either $X=0$
or there exists an $\E$-triangle $\suc[A][\coprod_{\alpha}V][X][][f]$
for some set $\alpha$ and $f\neq0.$ The class of all the $V$-generated
objects in $\D$ is denoted by $\Gen_{\D}(V)$ or $\Gen(V).$ If $\Gen(V)=\D$
we say that $V$ is a \textbf{generator} in $\D.$ 
\item We say that $X\in\D$ is \textbf{$V$-presented} if either $X=0$
or there exists an $\E$-triangle $\suc[A][\coprod_{\alpha}V][X][][f]$
for some set $\alpha$ such that $f\neq0$ and $A$ is $V$-generated.
The class of all the $V$-presented objects in $\D$ is denoted by
$\Pres_{\D}(V)$ or $\Pres(V).$ 
\item $\Add(V)$ is the class of all the objects in $\D$ which are direct
summands of coproducts of the form $\coprod_{\alpha}V$ where $\alpha$
is a set. 
\end{enumerate}
\end{defn}

In the previous setting, we could have that $\Pres(V)=\Gen(V)=\D$
even if $\D$ does not have arbitrary coproducts. For example, consider
$\D=\modd[\Lambda]$ the category of finitely generated left $\Lambda$-modules
for some Artin algebra $\Lambda.$ In this case $\text{Pres}(\Lambda)=\text{Gen}(\Lambda)=\modd[\Lambda]$
but $\modd[\Lambda]$ does not have arbitrary coproducts.

The following definition is inspired from \cite[Def. 6.1]{PSV2}. 
\begin{defn}
\label{ext-tilting} Let $\D$ be an extriangulated category with
negative first extension. We say that $V\in\D$ is an \textbf{extended
tilting object} if the following statements hold true. 
\begin{itemize}
\item[(T1)] $\mathbf{t}(V):=(\text{Gen}(V),V^{\bot_{0}})\in\stors\,\D.$ 
\item[(T2)] $\Gen(V)=\text{Pres}(V)$ and $\text{Gen}(V)\subseteq V^{\perp_{1}}$. 
\item[(T3)] $\text{Gen}(V)$ is a cogenerating class in $\D$. 
\end{itemize}
In such case, we say that $\mathbf{t}(V)$ is the \textbf{extended
tilting $s$-torsion pair} associated to $V.$ 
\end{defn}

\begin{lem}
\label{lem:tiltexact}Let $\mathcal{D}$ be an exact category. If
$V\in\mathcal{D}$ is an extended tilting object, then $\Gen(V)=V^{\bot_{1}}$. 
\end{lem}

\begin{proof}
It is enough to show that $V^{\bot_{1}}\subseteq\Gen(V)$. Let $X\in V^{\bot_{1}}$.
By (T3), there is an $\mathbb{E}$-triangle $\suc[X][T][X']$ in $\mathcal{D}$
with $T\in\Gen(V)$. Moreover, by (T1), there is an $\mathbb{E}$-triangle
$\suc[T'][X'][F]$ with $T'\in\Gen(V)$ and $F\in V^{\bot_{0}}$.
Then, by \cite[De. 2.12 (ET4*)]{NP}, we have the following commutative
diagram in $\mathcal{D}$, where rows and columns are $\mathbb{E}$-triangles.
\[
\xymatrix{X\ar[r]\ar@{=}[d] & E\ar[r]\ar[d]^{y} & T'\ar[d]\\
X\ar[r] & T\ar[r]\ar[d]^{x} & X'\ar[d]\\
\, & F\ar@{=}[r] & F,
}
\]
Here, observe that $x=0$ since $T\in\Gen(V)$, $F\in V^{\bot_{0}}$
and $(\Gen(V),V^{\bot_{0}})$ is a torsion pair. Then, by the dual
of \cite[Ex. 2.6]{Buhler}, $y$ is an isomorphism. And thus, $E\in\Gen(V)$.
Now, since $T'\in\Gen(V)=\Pres(V)$ (see (T2)), it follows from \cite[Prop. 3.15]{NP}
that we have the following commutative diagram, where the rows and
columns are $\mathbb{E}$-triangles and $T''\in\Gen(V)$.
\[
\xymatrix{ & T''\ar@{=}[r]\ar[d] & T''\ar[d]\\
X\ar[r]\ar@{=}[d] & E'\ar[d]\ar[r] & \coprod_{\alpha}V\ar[d]\\
X\ar[r] & E\ar[r] & T'
}
\]
Note that $E'\in\Gen(V)\star\Gen(V)=\Gen(V)$ and that the $\mathbb{E}$-triangle
$\suc[X][E'][\coprod_{\alpha}V]$ splits since $X\in V^{\bot_{1}}$.
Therefore, $X\in\Gen(V)$. 
\end{proof}
Let us show, in the following examples, that the definition of extended
tilting object agrees with the usual notion of tilting object in those
settings. 
\begin{example}
Let $\mathcal{D}$ be an extriangulated category with negative first
extension. 
\begin{enumerate}
\item Let $\mathcal{D}$ be abelian. Then, it follows from Lemma \ref{lem:tiltexact}
that $V\in\D$ is a tilting object as in \cite[Def. 6.1]{PSV2} if,
and only if, $V$ is extended tilting. 
\item Let $\mathcal{D}=\mathrm{Mod(R)}$ be the category of left $R$-modules,
for a ring $R.$ Then, by (a), \cite[Cor. 2.6 and 6.9]{PSV2} and
\cite[Thm. 3.11]{B}, we have that: $V\in\D$ is extended tilting
if, and only if, $V$ is $1$-tilting in the sense of Angeleri H{\"u}gel-Coelho
(see \cite{AC}). 
\item Let $\mathcal{D}$ be an extriangulated category such that every $D\in\mathcal{D}$
admits an $\s$-conflation $\suc[D][0][D']$ (e.g. a triangulated
category). Then, $0$ is an extended tilting object. Indeed, note
that $\Gen(0)=\{0\}$ and that $0^{\bot_{0}}=\mathcal{D}$. Therefore,
$(\Gen(0),0^{\bot_{0}})\in\stors\mathcal{D}$, and thus (T1) holds.
Moreover, (T2) also holds since $\Pres(0)=\{0\}$ and $\Gen(0)\subseteq0^{\bot_{1}}=\mathcal{D}$.
Lastly, (T3) is satisfied by assumption.
\item Let $\D$ be a triangulated category. Then, $V\in\D$ is extended
tilting if, and only if, $V$ is a zero object in $\D.$ Indeed, let
$V\in\D$ be extended tilting in $\D$. On the one hand, since $V\in\Gen(V)\subseteq V^{\bot_{1}}$,
we have that $\Sigma V\in V^{\bot_{0}}$. On the other hand, since
$(\Gen(V),V^{\bot_{0}})\in\stors(\mathcal{D})$, we have that $(\Sigma^{-1}\Gen(V),V^{\bot_{0}})$
is a $t$-structure. In particular, this means that $V^{\bot_{0}}$
is closed under negative shifts. And thus, $V=\Sigma^{-1}\Sigma V\in V^{\bot_{0}}$.
Therefore, $V=0$. 
\item Let $\mathcal{D}$ be with enough projectives and injectives (i.e.
every $D\in\mathcal{D}$ admits $\s$-conflations $\suc[X][I][X']$
and $\suc[X''][P][X]$ such that $P\in{}^{\bot_{1}}\mathcal{D}$ and
$I\in\mathcal{D}^{\bot_{1}}$). In \cite{ZZ}, Zhu and Zhuang defined
the notion of tilting class of projective dimension $\leq n$. In
\cite[Thm. 1]{ZZ} we can find conditions under which it is fulfilled
that: if $V$ is extended tilting, then $\Add(V)$ is a tilting class
of projective dimension $\leq1$. Namely, if $\Gen(V)=V^{\bot_{1}}$,
then $\Add(V)$ is tilting class of projective dimension $\leq1$
in the sense of \cite{ZZ} (see Corollary \ref{cor:apptilt} below). 
\end{enumerate}
\end{example}

\subsection{Projective dimension of extended tilting objects}

We will show in this section that the projective dimension of an extended
tilting object is $\leq1$. For this, we will need to introduce the
higher $\E$-extension groups of an extriangulated category $\D.$

As a precedent, let us recall that one can define higher $\E$-extension
groups in exact categories by splicing short exact sequences \cite[Chap. 6]{FS}.
A similar thing has been done for extriangulated categories in \cite[Sect. 3.2]{G}.
For the reader's convenience, and in order to have a language in the
forthcoming proofs of the paper, we recall some ideas from \cite[Sect. 3.2]{G}.

Let $\mathcal{D}$ be a small extriangulated category. For $A,B\in\mathcal{D}$,
we define recursively $\E^{n}(A,B)$ as follows. Set $\E^{1}(A,B):=\E(A,B)$.
Also recall that, for $\alpha:A'\rightarrow A$, $\beta:B\rightarrow B'$
and $\delta\in\E(A,B)$, $\delta\cdot\alpha:=\E(\alpha,B)(\delta)$
and $\beta\cdot\delta:=\E(A,\beta)(\delta)$. \\
 For $C\in\mathcal{D}$, denote by $(\delta\:{}_{C}\:\eta)$ the ordered
pair $(\delta,\eta)\in\E(C,B)\times\E(A,C)$. Let $S^{2}(A,B):=\left(\coprod_{C\in\mathcal{D}}\E(C,B)\times\E(A,C)\right)/\sim$,
where $(\delta\:{}_{X}\:\eta)\sim(\delta'\:{}_{Y}\:\eta')$ if: 
\begin{itemize}
\item there is a morphism $f:X\rightarrow Y$ such that $\delta'\cdot f=\delta$
and $f\cdot\eta=\eta'$, or 
\item there is a morphism $f:Y\rightarrow X$ such that $\delta\cdot f=\delta'$
and $f\cdot\eta'=\eta$. 
\end{itemize}
Note that $\sim$ is reflexive and symmetric but not necessarily transitive.
Hence, we consider the equivalence relation $\approx$ generated by
$\sim$. Lastly, define $\E^{2}(A,B):=\left(\coprod_{C\in\mathcal{D}}\E(C,B)\times\E(A,C)\right)/\approx$.
It can be proved that $\E^{2}(A,B)$ is an abelian group with the
Baer sum: 
\[
\overline{(\delta\:{}_{X}\:\eta)}\:+\:\overline{(\delta'\:{}_{Y}\:\eta')}:=\overline{(\delta\cdot\pi_{1}+\delta'\cdot\pi_{2})\:_{X\amalg Y}\:(\mu_{1}\cdot\eta+\mu_{2}\cdot\eta')},
\]
where $\pi_{1}:X\amalg Y\rightarrow X$ and $\pi_{2}:X\amalg Y\rightarrow Y$
are the canonical projections and $\mu_{1}:X\rightarrow X\amalg Y$
and $\mu_{2}:Y\rightarrow X\amalg Y$ are the canonical inclusions.
Observe that the following equalities hold: 
\[
\overline{((\eta+\eta')\:_{X}\:\delta)}\:=\:\overline{((\eta\cdot\pi_{1}+\eta'\cdot\pi_{2})\:_{X\amalg X}\:(\mu_{1}\cdot\delta+\mu_{2}\cdot\delta))}\:=\:\overline{(\eta\:_{X}\:\delta)}\:+\:\overline{(\eta'\:_{X}\:\delta)}
\]
where the first equality is induced by the morphism $\mu_{1}+\mu_{2}:X\rightarrow X\amalg X$;
\[
\overline{(\eta\:_{X}\:(\delta+\delta'))}\:=\:\overline{((\eta\cdot\pi_{1}+\eta\cdot\pi_{2})\:_{X\amalg X}\:(\mu_{1}\cdot\delta+\mu_{2}\cdot\delta'))}\:=\:\overline{(\eta\:_{X}\:\delta)}\:+\:\overline{(\eta\:_{X}\:\delta')}
\]
where the first equality is given by $\pi_{1}+\pi_{2}:X\amalg X\rightarrow X$;
\[
\overline{(0\:_{X}\:\delta)}\:=\:\overline{((0\cdot\pi_{1}+0\cdot\pi_{2})\:_{X\amalg X}\:(\mu_{1}\cdot\delta+\mu_{2}\cdot\delta))}\:=\:\overline{(0\:_{X}\:\delta)}\:+\:\overline{(0\:_{X}\:\delta)}
\]
where the first equality follows from the morphism $\mu_{1}+\mu_{2}:X\rightarrow X\amalg X$;
and similarly that 
\[
\overline{(\eta\:_{X}\:0)}\:=\:\overline{((\eta\cdot\pi_{1}+\eta\cdot\pi_{2})\:_{X\amalg X}\:(\mu_{1}\cdot0+\mu_{2}\cdot0))}\:=\:\overline{(\eta\:_{X}\:0)}\:+\:\overline{(\eta\:_{X}\:0)}.
\]
By using this, one can prove that: 
\[
\overline{(\eta\:_{X}\:0)}=0=\overline{(0\:_{X}\:\delta)}\text{ and }\overline{(-\eta\:_{X}\:\delta)}=-\overline{(\eta\:_{X}\:\delta)}=\overline{(\eta\:_{X}\:-\delta)}.
\]
Lastly, for $\alpha:A'\rightarrow A$, $\beta:B\rightarrow B'$ and
$\epsilon=\overline{(\delta\:{}_{X}\:\eta)}\in\E^{2}(A,B)$, define
$\epsilon\cdot\alpha=\E^{2}(\alpha,B)(\epsilon):=\overline{(\delta\:{}_{X}\:(\eta\cdot\alpha))}$
and $\beta\cdot\epsilon:=\E^{2}(A,\beta)(\epsilon):=\overline{((\beta\cdot\delta)\:{}_{X}\:\eta)}$.

Let $n>1$. For $C\in\mathcal{D}$, denote by $(\delta\:{}_{C}\:\eta)$
an ordered pair $(\delta,\eta)\in\E(C,B)\times\E^{n}(A,C)$. Let $S^{n+1}(A,B):=\left(\coprod_{C\in\mathcal{D}}\E(C,B)\times\E^{n}(A,C)\right)/\sim$,
where $(\delta\:{}_{X}\:\eta)\sim(\delta'\:{}_{Y}\:\eta')$ if 
\begin{itemize}
\item there is a morphism $f:X\rightarrow Y$ such that $\delta'\cdot f=\delta$
and $f\cdot\eta=\eta'$, or 
\item there is a morphism $f:Y\rightarrow X$ such that $\delta\cdot f=\delta'$
and $f\cdot\eta'=\eta$. 
\end{itemize}
Consider the equivalence relation $\approx$ generated by $\sim$.
Define $\E^{n+1}(A,B):=\left(\coprod_{C\in\mathcal{D}}\E(C,B)\times\E^{n}(A,C)\right)/\approx$.
It can be proved that $\E^{n+1}(A,B)$ is an abelian group with the
Baer sum: 
\[
\overline{(\delta\:{}_{X}\:\eta)}\:+\:\overline{(\delta'\:{}_{Y}\:\eta')}:=\overline{(\delta\cdot\pi_{1}+\delta'\cdot\pi_{2})\:_{X\amalg Y}\:(\mu_{1}\cdot\eta+\mu_{2}\cdot\eta')},
\]
where $\pi_{1}:X\amalg Y\rightarrow X$ and $\pi_{2}:X\amalg Y\rightarrow Y$
are the canonical projections and $\mu_{1}:X\rightarrow X\amalg Y$
and $\mu_{2}:Y\rightarrow X\amalg Y$ are the canonical inclusions.
As before, we have that: 
\begin{equation}
\overline{(\eta\:_{X}\:0)}=0=\overline{(0\:_{X}\:\delta)}\text{ and }\overline{(-\eta\:_{X}\:\delta)}=-\overline{(\eta\:_{X}\:\delta)}=\overline{(\eta\:_{X}\:-\delta)}.\label{eq:high}
\end{equation}
Lastly, for $\alpha:A'\rightarrow A$, $\beta:B\rightarrow B'$ and
$\epsilon=\overline{(\delta\:{}_{X}\:\eta)}\in\E^{n}(A,B)$, define
$\epsilon\cdot\alpha=\E^{n+1}(\alpha,B)(\epsilon):=\overline{(\delta\:{}_{X}\:(\eta\cdot\alpha))}$
and $\beta\cdot\epsilon:=\E^{n+1}(A,\beta)(\epsilon):=\overline{((\beta\cdot\delta)\:{}_{X}\:\eta)}$. 
\begin{defn}
Let $\mathcal{D}$ be a small extriangulated category and $V\in\mathcal{D}$.
Define the \textbf{projective dimension} of $V$ as $\pd(V)=\min\{k\in\mathbb{N}\,|\:\mathbb{E}^{k+1}(V,-)=0\}$,
if there is $n>0$ such that $\mathbb{E}^{n}(V,-)=0$; or as $\pd(V)=\infty$,
if $\mathbb{E}^{n}(V,-)\neq0$ for all $n>0$. For a class of objects
$\mathcal{V}\subseteq\mathcal{D}$, we say that $\pd(\mathcal{V})\leq n$
if $\pd(X)\leq n$ for all $X\in\mathcal{V}$. 
\end{defn}

\begin{rem}
\label{rem:conesorto}Let $\mathcal{D}$ be a small extriangulated
category and $V\in\mathcal{D}$. 
\begin{enumerate}
\item One can prove with routine arguments that $V^{\bot_{>0}}:=\{X\in\mathcal{D}\,|\:\E^{k}(V,X)=0\:\forall k>0\}$
is closed under cones and extensions (see \cite[Thm. 3.5]{G}).
\item If $\mathbb{E}^{n}(V,-)=0$, then $\mathbb{E}^{k}(V,-)=0$ for all
$k\geq n$ (see (\ref{eq:high})). 
\item $\pd(V)=\min\{m\in\mathbb{N}\,|\:\mathbb{E}^{k}(V,-)=0\,\forall k>m\}$.
\end{enumerate}
\end{rem}

\begin{lem}
\label{lemitapdT}Let $\mathcal{D}$ be a small extriangulated category
with negative first extension and let $V$ be an object of $\mathcal{D}$
such that $V^{\bot_{1}}=V^{\bot_{>0}}$. If $V^{\bot_{1}}$ is a cogenerating
class in $\mathcal{D}$ and there is $(\mathcal{T},\mathcal{F})\in\stors(\mathcal{D})$
with $\mathcal{T}\subseteq V^{\bot_{1}}$ and $\mathcal{F}\subseteq V^{\bot_{0}}$,
then $\pd(V)\leq1$. 
\end{lem}

\begin{proof}
Let us prove that $\mathbb{E}^{2}(V,D)=0$ for any $D\in\mathcal{D}$.
By assumption, there is an $\s$-conflation $\suc[D][E][D']$ with
$E\in V^{\bot_{>0}}$. Moreover, since $(\mathcal{T},\mathcal{F})\in\stors(\mathcal{D})$,
there is an $\s$-conflation $\suc[T][D'][F]$ with $T\in V^{\bot_{1}}$
and $F\in V^{\bot_{0}}$. Then, by \cite[De. 2.12 (ET4*)]{NP}, we
have the following commutative diagram in $\mathcal{D}$, where rows
and columns are $\mathbb{E}$-triangles. 
\[
\xymatrix{D\ar[r]\ar@{=}[d] & M\ar[r]\ar[d]^{y} & T\ar[d]\\
D\ar[r] & E\ar[r]\ar[d]^{x} & D'\ar[d]\\
\, & F\ar@{=}[r] & F
}
\]
Observe that the central column in the diagram gives us the exact
sequence 
\[
0=\Hom[\mathcal{D}][V][F]\rightarrow\mathbb{E}(V,M)\rightarrow\mathbb{E}(V,E)=0.
\]
Hence, $M\in V^{\bot_{1}}=V^{\bot_{>0}}$. Now, by \cite[Thm. 3.5]{G},
the top row gives us an exact sequence 
\[
0=\mathbb{E}(V,T)\rightarrow\mathbb{E}^{2}(V,D)\rightarrow\mathbb{E}^{2}(V,M)=0.
\]
Therefore, $\mathbb{E}^{2}(V,D)=0$. 
\end{proof}
The following result is inspired from \cite[Cor. 2.6]{PSV2}. 

\begin{lem}
\label{lem:lemitapd}Let $\mathcal{D}$ be a small extriangulated
category with negative first extension and $V\in\mathcal{D}$. If
$V$ is extended tilting, then $\Gen(V)\subseteq V^{\bot_{>0}}$. 
\end{lem}

\begin{proof}
Let us show that $\mathcal{T}:=\Gen(V)\subseteq V^{\bot_{>0}}.$ We
proceed by induction on $n$ to prove that $\E^{n}(V,\mathcal{T})=0$.
The case $n=1$ is trivial. Let $n>1$ and assume that $\E^{n-1}(V,\mathcal{T})=0$.
For $X\in\mathcal{T}$ and $\epsilon\in\E^{n}(V,X)$, we have that
$\epsilon=\overline{(\eta\:_{C}\:\delta)}$ with $\eta\in\E(C,X)$
and $\delta\in\E^{n-1}(V,C)$. Let $\eta$ be realized by $\suc[X][E][C]$.
Since $\mathcal{T}=\Gen(V)$ is cogenerating by (T3), there is an
$\s$-conflation $\suc[E][T][S]$ with $T\in\mathcal{T}$. Then, by
\cite[Def. 2.12 (ET4)]{NP}, there is a morphism $f:C\rightarrow C'$
and an $\s$-conflation $\eta':\:\suc[X][T][C']$ such that $\eta'\cdot f=\eta$
as can be seen in the following diagram 
\[
\xymatrix{X\ar[r]\ar@{=}[d] & E\ar[r]\ar[d] & C\ar[d]^{f}\\
X\ar[r] & T\ar[r]\ar[d] & C'\ar[d]\\
 & S\ar@{=}[r] & S.
}
\]
Observe that $C'\in V^{\bot_{1}}$ since $X,T\in\mathcal{T}$, $(\mathcal{T},\mathcal{F}):=(\Gen(V),V^{\bot_{0}})$
is an $s$-torsion pair and $\mathcal{T}\subseteq V^{\bot_{1}}$ by
(T1) and (T2) (see Proposition \ref{lem:cerraduras stors} (a)). Hence
$f\cdot\delta\in\E^{n-1}(V,C')=0$ by induction hypothesis. Therefore
$\epsilon=\overline{(\eta\:_{C}\:\delta)}=\overline{((\eta'\cdot f)\:_{C}\:\delta)}=\overline{(\eta'\:_{C}\:(f\cdot\delta))}=\overline{(\eta'\:_{C}\:0)}=0.$
\end{proof}
\begin{lem}
\label{lem:lemita2}Let $\mathcal{D}$ be a extriangulated category
with negative first extension and $V\in\mathcal{D}$ be an extended
tilting object in $\mathcal{D}$. If $\eta:\:\suc[F][T][C][][a]$
is an $\s$-conflation in $\mathcal{D}$ with $F\in V^{\bot_{0}}\cap V^{\bot_{1}}$
and $T\in\Gen(V)-\{0\}$, then $C\in\Gen(V)$. 
\end{lem}

\begin{proof}
Since $T\neq0$, there is an $\s$-conflation $\eta':\:\suc[T'][\coprod_{\alpha}V][T][][b]$
with $b\neq0$ and $T'\in\mathcal{T}:=\Gen(V)$ by (T2). Thus, by
\cite[Def. 2.12 (ET4*)]{NP}, there is a an $\s$-conflation $\suc[\eta_{0}:\:M][\coprod_{\alpha}V][C]$
and a commutative diagram as the one below satisfying the standard
compatibilities (see \cite[Rem. 2.22]{NP}) 
\[
\xymatrix{T'\ar[r]\ar@{=}[d] & M\ar[r]\ar[d] & F\ar[d]\\
T'\ar[r] & \coprod_{\alpha}V\ar[r]^{b}\ar[d] & T\ar[d]^{a}\\
 & C\ar@{=}[r] & C.
}
\]
 Here, observe that $\Hom[\mathcal{D}][V][a]$ and $\Hom[\mathcal{D}][V][b]$
are epimorphisms because of the following exact sequences induced
by $\eta$ and $\eta'$: 
\begin{alignat*}{1}
\Hom[\mathcal{D}][V][T] & \xrightarrow{\Hom[\mathcal{D}][V][a]}\Hom[\mathcal{D}][V][C]\rightarrow\mathbb{E}(V,F)=0\\
\Hom[\mathcal{D}][V][\coprod_{\alpha}V] & \xrightarrow{\Hom[\mathcal{D}][V][b]}\Hom[\mathcal{D}][V][T]\rightarrow\mathbb{E}(V,T')=0,
\end{alignat*}
where $\mathbb{E}(V,T')=0$ since $T'\in\mathcal{T}\subseteq V^{\bot_{1}}$
by (T2). Lastly, note that, in case of having $a\circ b=0$, we get
that $\Hom[\mathcal{D}][V][C]=0$ because 
\[
0=\Hom[\mathcal{D}][V][a\circ b]:\Hom[\mathcal{D}][V][\coprod_{\alpha}V]\rightarrow\Hom[\mathcal{D}][V][C]
\]
 is an epimorphism. In other words, $C\in\mathcal{F}:=V^{\bot_{0}}$.
And thus, $T\in\mathcal{T}\cap\mathcal{F}=0$ (because of the $\s$-conflation
$\eta$), which is a contradiction since $b\neq0$. Therefore, $a\circ b\neq0$,
and hence the $\s$-conflation $\suc[\eta_{0}:\:M][\coprod_{\alpha}V][C]$
gives us that $C\in\mathcal{T}$.
\end{proof}
\begin{prop}
\label{pdT1} Let $\mathcal{D}$ be a small extriangulated category
with negative first extension, $V\in\mathcal{D}$ be an extended tilting
object in $\mathcal{D}$, and $(\mathcal{T},\mathcal{F}):=(\Gen(V),V^{\bot_{0}})$.
Then, $\pd(V)\leq1$ if any of the following conditions hold.
\begin{enumerate}
\item Every $F\in\mathcal{F}\cap V^{\bot_{1}}$ admits an $\s$-conflation
$\suc[F][T][C']$ with $T\in\mathcal{T}-\{0\}$. 
\item Every $\s$-conflation $\suc[F][0][C']$ with $F\in\mathcal{F}\cap V^{\bot_{1}}$
satisfies that $C'\in\mathcal{F}\cap V^{\bot_{1}}$. 
\item $\mathcal{T}=V^{\bot_{1}}$ (e.g. when $\mathcal{D}$ is exact). 
\end{enumerate}
\end{prop}

\begin{proof}
Observe that it is enough to prove that $V^{\bot_{1}}\subseteq V^{\bot_{>0}}$
by Lemma \ref{lemitapdT} together with (T1), (T2) and (T3). For this,
we will use the fact that $\mathcal{T}\subseteq V^{\bot_{>0}}$ (see
Lemma \ref{lem:lemitapd}). For $X\in V^{\bot_{1}}$, consider the
$\s$-conflation $\suc[T][X][F]$ given by the $s$-torsion pair $(\mathcal{T},\mathcal{F})$.
Note that $X\in V^{\bot_{>0}}$ if and only if $F\in V^{\bot_{>0}}$
(see Remark \ref{rem:conesorto}). Hence, it is enough to show that
$\E^{n}(V,F)=0$ for all $n>0$. The case $n=1$ is trivial since
$X\in V^{\bot_{1}}$ and $T\in V^{\bot_{>0}}$. For $n>1$, we will
need that condition (a), (b) or (c) holds true. 

Assume that condition (a) holds. That is, there is an $\s$-conflation
$\eta':\:\suc[F][T_{0}][C_{1}][][a]$ with $T_{0}\in\mathcal{T}-\{0\}$.
Hence, $C_{1}\in\mathcal{T}$ by Lemma \ref{lem:lemita2}. And thus,
we have that $F\in V^{\bot_{>0}}$ since $\mathcal{T}\subseteq V^{\bot_{>0}}$
and $V^{\bot_{>0}}$ is closed under cones. 

Assume that condition (b) holds. If there is an $\s$-conflation $\eta':\:\suc[F][T_{0}][C_{1}][][a]$
with $T_{0}\in\mathcal{T}-\{0\}$, we can proceed as before to prove
that $F\in V^{\bot_{>0}}$. If that is not the case, then there is
an $\s$-conflation $\eta_{0}:\:\suc[F][0][C_{1}]$ by (T3). Moreover,
we can consider a sequence of $\s$-conflations $\{\eta_{i}:\:\suc[C_{i}][T_{i}][C_{i+1}]\}_{i=0}^{\infty}$
with $T_{i}\in\mathcal{T}$ for all $i\geq0$, $T_{0}=0$, and $C_{0}=F$.
We consider the following cases. 

\textbf{Case 1}: $T_{i}=0$ for all $i\geq0$. Note that, by assumption,
we have that $C_{i}\in\mathcal{F}\cap V^{\bot_{1}}$ for all $i>0$.
One can prove for the $\s$-conflation $\eta_{i}:\:\suc[C_{i}][0][C_{i+1}]$
that $\mathbb{E}^{n+1}(V,C_{i})=0$ if $\mathbb{E}^{n}(V,C_{i})=0$
and $\mathbb{E}^{n}(V,C_{i+1})=0$. Hence, using this fact recursively,
we can conclude that $F\in V^{\bot_{>0}}$. 

\textbf{Case 2}: there is $N\geq1$ such that $T_{N}\neq0$ and $T_{i}=0$
for all $0\leq i<N$. Note that the $\s$-conflation $\eta_{N}:\:\suc[C_{N}][T_{N}][C_{N+1}]$
implies that $C_{N+1}\in\mathcal{T}\subseteq V^{\bot_{>0}}$ by Lemma
\ref{lem:lemita2}. Hence, proceeding recursively as in Case 1, we
can conclude that $F\in V^{\bot_{>0}}$. 

Lastly, if (c) holds true, then $V^{\bot_{1}}=\mathcal{T}\subseteq V^{\bot_{>0}}$.
\end{proof}
\begin{rem}
Let $\mathcal{D}$ be a extriangulated category with negative first
extension and $(\mathcal{T},\mathcal{F})\in\stors(\mathcal{D})$.
One can check that, if $\mathcal{T}$ satisfies that every $\s$-conflation
$\suc[D][T][D']$ with $T\in\mathcal{T}$ satisfies that $D'\in\mathcal{T}$,
then $(\mathcal{T},\mathcal{F})$ satisfies condition (a) from Proposition
\ref{pdT1}. Note that Proposition \ref{lemaA} provides a family
of non-abelian extriangulated categories where this property holds. 
\end{rem}

\begin{cor}
\label{cor:apptilt}Let $\mathcal{D}$ be a small extriangulated category
with negative first extension and $V\in\mathcal{D}$. If $\mathcal{D}$
has enough projectives and enough injectives, $V$ is extended tilting
and $\Gen(V)=V^{\bot_{1}}$, then $\Add(V)$ is a tilting class of
projective dimension $\leq1$ in the sense of \cite{ZZ}.
\end{cor}

\begin{proof}
Let $\mathcal{V}:=\Add(V)$. Before proceeding with the proof, we
observe the following facts. In \cite{ZZ}, the higher extension groups
are defined using syzygies and cosyzygies as
\begin{equation}
\mathbb{E}^{n}(X,Y):=\mathbb{E}(\Omega^{n-1}X,Y)\cong\mathbb{E}(X,\text{\ensuremath{\mho}}^{n-1}Y),\label{eq:syz}
\end{equation}
where $\Omega^{n-1}X$ is the $(n-1)$-th syzygy of $X$ and $\text{\ensuremath{\mho}}^{n-1}Y$
is the $(n-1)$-th cosyzygy of $Y$ (see \cite[Prop. 5.2]{LN} for
more details and unexplained notation). One can prove that such approach
is equivalent to ours (see \cite[Sec. 2.10]{BGLS}). Note that, by
using the isomorphism in (\ref{eq:syz}), we have that $\mathcal{V}\subseteq\Gen(V)\subseteq V^{\bot_{1}}$,
$V^{\bot_{1}}=\mathcal{V}^{\bot_{1}}$ , $V^{\bot_{>0}}=\mathcal{V}^{\bot_{>0}}$,
and $\pd(\mathcal{V})\leq1$ by \cite[Prop 3.1.(b)]{AMP}. Lastly,
for a class of objects $\mathcal{X}\subseteq\mathcal{D}$, denote
as $\Gen_{0}(\mathcal{\mathcal{X}})$ the class of objects $D\in\mathcal{D}$
admitting an $\s$-deflation 
\[
\suc[D_{0}][T_{0}][D]
\]
with $T_{0}\in\mathcal{V}$; and observe that $\Gen(V)\subseteq\Gen_{0}(\mathcal{V})$. 

It is clear that $\mathcal{V}$ is closed under direct summands. Hence,
to prove that $\mathcal{V}$ is a tilting class in the sense of \cite{ZZ},
it is enough to show that every object in $\mathcal{V}^{\bot_{>0}}$
admits a $\mathcal{V}$-precover and that $\Gen_{0}(\mathcal{V})=\mathcal{V}^{\bot_{>0}}$
by \cite[Thm. 1]{ZZ}. For this, observe that 
\[
\mathcal{V}^{\bot_{_{1}}}=V^{\bot_{1}}=V^{\bot_{>0}}=\mathcal{\mathcal{V}}^{\bot_{>0}}
\]
by Proposition \ref{pdT1} and Remark \ref{rem:conesorto}(b). Now,
since $\mathcal{V}^{\bot_{>0}}=V^{\bot_{1}}=\Gen(V)=\Pres(V)$, for
every $X\in\mathcal{V}^{\bot_{>0}}-\{0\}$ there is a deflation $\suc[X'][V_{0}][X][][f]$
with $V_{0}\cong\coprod_{\alpha}V$ and $X'\in\Gen(V)\subseteq V^{\bot_{1}}$.
This implies that $f$ is an $\Add(V)$-precover. Moreover, since
$\pd(\mathcal{V})\leq1$, it follows from \cite[Lem. 12]{ZZ} that
$\Gen_{0}(\mathcal{V}^{\bot_{>0}})=\mathcal{V}^{\bot_{>0}}$. And
thus, since $\mathcal{V}^{\bot_{>0}}=V^{\bot_{1}}=\Gen(V)$, we have
that 
\[
\mathcal{V}^{\bot_{>0}}=\Gen(V)\subseteq\Gen_{0}(\mathcal{V})\subseteq\Gen_{0}(\mathcal{V}^{\bot_{>0}})=\mathcal{V}^{\bot_{>0}}.
\]
Therefore, it follows from \cite[Thm. 1]{ZZ} that $\mathcal{V}$
is a tilting class of projective dimension $\leq1$. 
\end{proof}

\subsection{Extended tilting objects are $\mathbb{E}$-universal}

In this section we will show that extended tilting objects in extriangulated
categories are $\mathbb{E}$-universal. Note that the notion of an
$\mathbb{E}$-universal object given below is inspired from the notion
of an $\mathbf{Ext}$-universal object for abelian categories (see
\cite[Def. 5.6]{PSV2} or \cite{AP} for more details). 
\begin{defn}
Let $\D$ be an extriangulated category. An object $V$ in $\D$ is
\textbf{$\mathbf{\E}$-universal} if, for any $D\in\D$, there is
an $\E$-triangle in $\D$ of the form: 
\[
\varepsilon:\:\suc[D][A][\coprod_{\alpha}V]
\]
for some set $\alpha\neq\emptyset$, such that the canonical map $\text{Hom}_{\D}(V,\coprod_{\alpha}V)\to\E^{1}(V,D)$
is surjective. In such case, $\varepsilon$ is called a \textbf{universal
$\mathbf{\E}$-triangle} of $V$ by $D$. 
\end{defn}

We can adapt the proof of \cite[Lem. 6.6]{PSV2} to our context as
follows. 
\begin{prop}
\label{et-univ} Let $\D$ be an extriangulated category with negative
first extension. Then, any extended tilting object $V\in\D$ is $\E$-universal. 
\end{prop}

\begin{proof}
Let $V\in\D$ be an extended tilting object and $D\in\D$. Since $\mathcal{T}:=\text{Gen}(V)$
is a cogenerating class in $\D$, there is an $\s$-conflation $\suc[D][T][D']$
in $\D$ with $T\in\T$. Now, from the $\s$-conflation (see Remark
\ref{adj-stor}) $\delta_{D'}:T'\to D'\to F$ associated to the $s$-torsion
pair $\mathbf{t}(V):=(\T,\F)=(\text{Gen}(V),V^{\bot_{0}})$ together
with \cite[Rk. 2.22]{NP}, we get the following commutative diagram
in $\D$ 
\[
\xymatrix{D\ar[r]\ar@{=}[d] & Z\ar[r]\ar[d] & T'\ar[d]\\
D\ar[r] & T\ar[r]\ar[d] & D'\ar[d]\\
\, & F\ar@{=}[r] & F,
}
\]
where $T'\in\T,$ $F\in\F$ and all the rows and columns are $\E$-triangles.
Applying the functor $\text{Hom}_{\D}(V,-)$ to the left column of
the above diagram, we obtain the exact sequence: 
\[
0=\text{Hom}_{\C}(V,F)\to\E^{1}(V,Z)\to\E^{1}(V,T)=0
\]
where the last equality hold, since $T\in\T\subseteq V^{\perp_{1}}$.
Thus, $Z\in V^{\perp_{1}}$. Now, let $T'\neq0.$ Then, from the equality
$\T=\text{Pres}(V),$ \cite[Prop. 3.15]{NP} together with the first
row in the above diagram, we deduce the following commutative diagram
in $\D$: 
\[
\xymatrix{\, & T''\ar@{=}[r]\ar[d] & T''\ar[d]\\
D\ar[r]\ar@{=}[d] & D''\ar[r]\ar[d] & \coprod_{\alpha}V\ar[d]\\
D\ar[r] & Z\ar[r] & T'
}
\]
for some set $\alpha\neq\emptyset$ and $T''\in\T.$ In case $T'=0,$
we can use the same diagram by replacing the $s$-conflation $T''\to\coprod_{\alpha}V\to T'$
with $V\to V\to0.$ In particular, from this diagram, we get that
$D''\in V^{\bot_{1}}$ since $V^{\bot_{1}}$ is closed under extensions
in $\D$. Now, applying the functor $\text{Hom}_{\D}(V,-)$ to the
upper row in the last diagram, we obtain the exact sequence: 
\[
\text{Hom}_{\D}(V,\coprod_{\alpha}V)\to\E^{1}(V,D)\to\E^{1}(V,D'')=0
\]
where the last equality is due to the fact that $D''\in V^{\bot_{1}}$.
Therefore, such $\E$-triangle is a universal $\E$-triangle of $V$
by $D$ as desired. 
\end{proof}

\subsection{Extended tilting objects in restricted hearts}

In this section we will show that an extended tilting object in a
restricted heart is nothing more than a quasi-tilting object in a
standard heart. Let us recall the notion of quasi-tilting object from
\cite[Def. 6.1]{PSV2}. 
\begin{defn}
\label{quasi-tilting} Let $\D$ be an extriangulated category with
negative first extension. We say that $V\in\D$ is a\textbf{ quasi-tilting
object} if the following statements hold true. 
\begin{itemize}
\item[(T1)] $\mathbf{t}(V):=(\text{Gen}(V),V^{\bot_{0}})\in\stors\,\D.$ 
\item[(QT)] $\Gen(V)=\text{Pres}(V)=\overline{\Gen}_{\mathcal{D}}(V)\cap V^{\bot_{1}}$,
where $\overline{\Gen}_{\mathcal{D}}(V)$ is the class of objects
$X\in\mathcal{D}$ that admit an $\s$-conflation $\suc[X][V_{0}][X']$
in $\mathcal{D}$ with $V_{0}\in\Gen(V)$. 
\end{itemize}
\end{defn}

\begin{lem}
\label{lem:torsion quasitilt}Let $\mathcal{D}$ be an extriangulated
category with negative first extension and $V\in\mathcal{D}$. If
$(\mathcal{T},\mathcal{F})$ is an $s$-torsion pair with $\mathcal{T}=\Gen(V)$,
then $\mathcal{F}=V^{\bot_{0}}$. 
\end{lem}

\begin{proof}
Since $V\in\mathcal{T}$, we have that $\mathcal{F}\subseteq V^{\bot_{0}}$.
For the opposite contention, consider $X\in V^{\bot_{0}}$ and the
canonical $\s$-conflation $\eta:\:\suc[T][X][F][u]$ with $T\in\mathcal{T}$
and $F\in\mathcal{F}$. Observe that, since $\mathcal{T}=\Gen(V)$,
$T=0$ or there is an $\s$-conflation $\suc[A][V_{0}][T][][f]$ with
$f\neq0$ and $V_{0}=\coprod_{\alpha}V$. But, the exact sequence
\[
\mathbb{E}^{-1}(V_{0},F)\overset{\eta^{-1}}{\rightarrow}\Hom[\mathcal{D}][V_{0}][T]\rightarrow\Hom[\mathcal{D}][V_{0}][X]
\]
implies that $f\in\Ker[{(\Hom[\mathcal{D}][V_{0}][u])=\im[\eta^{-1}]=0}]$
because $V_{0}\in\mathcal{T}$ and $F\in\mathcal{F}$. Therefore,
$f=0$ and thus $X\in\mathcal{F}$. 
\end{proof}
\begin{lem}
\label{lem:key-1-1} Let $\D$ be a triangulated category, $\t_{1}\leq\t_{2}\leq\t_{3}$
in $\stors\,\mathcal{D},$ $V\in\mathcal{H}_{[\t_{1},\t_{2}]}$ and
$I$ be a set. Then, the coproduct $\coprod_{I}^{\mathcal{H}_{[\t_{1},\t_{2}]}}V$
of $I$-copies of $V$ in $\mathcal{H}_{[\t_{1},\t_{2}]}$ exists
if, and only if, the coproduct $\coprod_{I}^{\mathcal{H}_{[\t_{1},\t_{3}]}}V$
of $I$-copies of $V$ in $\mathcal{H}_{[\t_{1},\t_{3}]}$ exists.
In such case, we have that both coproducts are isomorphic. 
\end{lem}

\begin{proof}
By Lemma \ref{lem:stors de extensiones de corazones} (c), we have
that $(\mathcal{H}_{[\t_{1},\t_{2}]},\mathcal{H}_{[\t_{2},\t_{3}]})\in\stors\,\mathcal{H}_{[\t_{1},\t_{3}]}$.
Moreover, by Lemma \ref{lem:extensions of hearts} and Remark \ref{adj-stor},
we get that the inclusion functor $\iota:\mathcal{H}_{[\t_{1},\t_{2}]}\hookrightarrow\mathcal{H}_{[\t_{1},\t_{3}]}$
is a left adjoint of the functor $\t_{2}{|_{\mathcal{\mathcal{H}}_{[\t_{1},\t_{3}]}}}:\mathcal{H}_{[\t_{1},\t_{3}]}\to\mathcal{H}_{[\t_{1},\t_{2}]}$.
Using such adjoint, we get that $\coprod_{I}^{\mathcal{H}_{[\t_{1},\t_{2}]}}V$
is a coproduct of $I$-copies of $V$ in $\mathcal{H}_{[\t_{1},\t_{3}]}$.
On the other hand, suppose that there exists $\coprod_{I}^{\mathcal{H}_{[\t_{1},\t_{3}]}}V$,
the coproduct of $I$-copies of $V$ in $\mathcal{H}_{[\t_{1},\t_{3}]}$.
Using once again the Lemma \ref{lem:extensions of hearts}, we have
an $\s$-conflation $\eta:\:\suc[T][\coprod_{I}^{\mathcal{H}_{[\t_{1},\t_{3}]}}V][F][][f]$,
where $T\in\mathcal{H}_{[\t_{1},\t_{2}]}$ and $F\in\mathcal{H}_{[\t_{2},\t_{3}]}$.
Note that $f\circ\mu_{i}=0$ for all $i\in I$, where $\mu_{i}:V\hookrightarrow\coprod_{I}^{\mathcal{H}_{[\t_{1},\t_{3}]}}V$
is the canonical inclusion. Hence, by the universal property of coproducts,
we have that $f=0$. Therefore, we have the distinguished triangle
$\Sigma^{-1}F\to T\to\coprod_{I}^{\mathcal{H}_{[\t_{1},\t_{3}]}}V\xrightarrow{f}F,$
with $f=0,$ in the triangulated category $\D,$ and thus $T\simeq\Sigma^{-1}F\amalg\coprod_{I}^{\mathcal{H}_{[\t_{1},\t_{3}]}}V.$
Hence $\coprod_{I}^{\mathcal{H}_{[\t_{1},\t_{3}]}}V\in\mathcal{H}_{[\t_{1},\t_{2}]},$
and thus, the existence of the coproduct $\coprod_{I}^{\mathcal{H}_{[\t_{1},\t_{2}]}}V$
together with the isomorphism $\coprod_{I}^{\mathcal{H}_{[\t_{1},\t_{2}]}}V\simeq\coprod_{I}^{\mathcal{H}_{[\t_{1},\t_{3}]}}V$
is clear. 
\end{proof}
\begin{thm}
\label{thm:restrict tilt} Let $\D$ be a triangulated category and
$\t_{1}\leq\t_{2}$ in $\stors\,\mathcal{D}$ be such that $\t_{2}\leq\Sigma^{-1}\t_{1}$.
If $V\in\mathcal{H}_{[\t_{1},\t_{2}]}$ satisfies that $\overline{\Gen}_{\mathcal{H}_{[\t_{1},\Sigma^{-1}\t_{1}]}}(V)\subseteq\mathcal{H}_{[\t_{1},\t_{2}]}$,
then the following statements are equivalent. 
\begin{enumerate}
\item $V$ is an extended tilting object in $\Hcal_{[\t_{1},\t_{2}]}.$ 
\item $V$ is a quasi-tilting object in $\mathcal{H}_{[\t_{1},\Sigma^{-1}\t_{1}]}$
and $\overline{\Gen}_{\mathcal{H}_{[\t_{1},\Sigma^{-1}\t_{1}]}}(V)=\mathcal{H}_{[\t_{1},\t_{2}]}$. 
\end{enumerate}
Moreover, if one of the above statements holds true, then 
\[
V_{\mathcal{H}_{[\t_{1},\Sigma^{-1}\t_{1}]}}^{\bot_{0}}=V_{\mathcal{H}_{[\t_{1},\t_{2}]}}^{\bot_{0}}\star\mathcal{H}_{[\t_{2},\Sigma^{-1}\t_{1}]}.
\]
\end{thm}

\begin{proof}
Let $\t_{1}=(\mathcal{T}_{1},\mathcal{F}_{1})$ and $\t_{2}=(\mathcal{T}_{2},\mathcal{F}_{2})$.
Observe that $\mathcal{T}:=\Gen_{\mathcal{H}_{[\t_{1},\t_{2}]}}(V)=\Gen_{\mathcal{H}_{[\t_{1},\Sigma^{-1}\t_{1}]}}(V)$
and $\Pres_{\mathcal{H}_{[\t_{1},\t_{2}]}}(V)=\Pres_{\mathcal{H}_{[\t_{1},\Sigma^{-1}\t_{1}]}}(V)$
by Lemma \ref{lem:key-1-1} and Proposition \ref{lemaA}. For simplicity,
we denote $\overline{\Gen}(V):=\overline{\Gen}_{\mathcal{H}_{[\t_{1},\Sigma^{-1}\t_{1}]}}(V)$.

$(b)\Rightarrow(a)$ Since $V$ is quasi-tilting in $\mathcal{H}_{[\t_{1},\Sigma^{-1}\t_{1}]}$,
we have the $s$-torsion pair $(\mathcal{T},V_{\mathcal{H}_{[\t_{1},\Sigma^{-1}\t_{1}]}}^{\bot_{0}})\in\stors\,\mathcal{H}_{[\t_{1},\Sigma^{-1}\t_{1}]}$.
Hence, by Lemma \ref{lem:stors de extensiones de corazones}(d) and
Lemma \ref{lem:torsion quasitilt}, it follows that $(\mathcal{T},V_{\mathcal{H}_{[\t_{1},\t_{2}]}}^{\bot_{0}})\in\stors\,\Hcal_{[\t_{1},\t_{2}]}$
and $V_{\mathcal{H}_{[\t_{1},\Sigma^{-1}\t_{1}]}}^{\bot_{0}}=V_{\mathcal{H}_{[\t_{1},\t_{2}]}}^{\bot_{0}}\star\mathcal{H}_{[\t_{2},\Sigma^{-1}\t_{1}]}$.
We have proved that $V$ satisfies (T1) in $\mathcal{H}_{[\t_{1},\t_{2}]}$.
To prove (T2), it is enough to show that $\mathcal{T}=V_{\mathcal{H}_{[\t_{1},\t_{2}]}}^{\bot_{1}}$.
For this, observe that 
\[
\mathcal{T}=V_{\mathcal{H}_{[\t_{1},\Sigma^{-1}\t_{1}]}}^{\bot_{1}}\cap\overline{\Gen}(V)=V_{\mathcal{H}_{[\t_{1},\Sigma^{-1}\t_{1}]}}^{\bot_{1}}\cap\mathcal{H}_{[\t_{1},\t_{2}]}=V_{\mathcal{H}_{[\t_{1},\t_{2}]}}^{\bot_{1}}.
\]
Lastly, to prove (T3), we observe that $\Gen(V)$ is a cogenerating
class in $\mathcal{H}_{[\t_{1},\t_{2}]}$ because $\overline{\Gen}(V)=\mathcal{H}_{[\t_{1},\t_{2}]}$
in $\mathcal{H}_{[\t_{1},\Sigma^{-1}\t_{1}]}$ and $\mathcal{T}_{2}$
is closed under cones.

$(a)\Rightarrow(b)$ We have $(\mathcal{T},V_{\mathcal{H}_{[\t_{1},\t_{2}]}}^{\bot_{0}})\in\stors\,\mathcal{H}_{[\t_{1},\t_{2}]}$.
Thus, by Lemma \ref{lem:stors de extensiones de corazones}(a) and
Lemma \ref{lem:torsion quasitilt}, it follows that $(\mathcal{T},V_{\mathcal{H}_{[\t_{1},\Sigma^{-1}\t_{1}]}}^{\bot_{0}})\in\stors\,\mathcal{H}_{[\t_{1},\Sigma^{-1}\t_{1}]}$
and that $V_{\mathcal{H}_{[\t_{1},\Sigma^{-1}\t_{1}]}}^{\bot_{0}}=V_{\mathcal{H}_{[\t_{1},\t_{2}]}}^{\bot_{0}}\star\mathcal{H}_{[\t_{2},\Sigma^{-1}\t_{1}]}$.
Therefore, $V$ satisfies (T1) in $\mathcal{H}_{[\t_{1},\Sigma^{-1}\t_{1}]}$.
\\
 Let us prove that $\mathcal{H}_{[\t_{1},\t_{2}]}=\overline{\Gen}(V)$.
Since $V$ is extended tilting in $\mathcal{H}_{[\t_{1},\t_{2}]}$,
we have that $\Gen(V)$ is a cogenerating class in $\mathcal{H}_{[\t_{1},\t_{2}]}$.
Hence, $\mathcal{H}_{[\t_{1},\t_{2}]}\subseteq\overline{\Gen}(V)$.
Therefore, we have proved the desired equality since the opposite
contention is an hypothesis. \\
 Lastly, we proceed to prove that $\overline{\Gen}(V)\cap V_{\mathcal{H}_{[\t_{1},\Sigma^{-1}\t_{1}]}}^{\bot_{1}}=V_{\mathcal{H}_{[\t_{1},\t_{2}]}}^{\bot_{1}}=\mathcal{T}$.
For this, we firstly note that $\mathcal{T}=V_{\mathcal{H}_{[\t_{1},\t_{2}]}}^{\bot_{1}}$
by Lemmas \ref{lem:tiltexact} and \ref{nimpex} since $[\t_{1},\t_{2}]$
is normal. And thus, since $\mathcal{T}\subseteq\mathcal{H}_{[\t_{1},\t_{2}]}=\overline{\Gen}(V)$,
we have that 
\[
V_{\mathcal{H}_{[\t_{1},\t_{2}]}}^{\bot_{1}}=\mathcal{T}\cap V_{\mathcal{H}_{[\t_{1},\Sigma^{-1}\t_{1}]}}^{\bot_{1}}\subseteq\overline{\Gen}(V)\cap V_{\mathcal{H}_{[\t_{1},\Sigma^{-1}\t_{1}]}}^{\bot_{1}}=\mathcal{H}_{[\t_{1},\t_{2}]}\cap V_{\mathcal{H}_{[\t_{1},\Sigma^{-1}\t_{1}]}}^{\bot_{1}}=V_{\mathcal{H}_{[\t_{1},\t_{2}]}}^{\bot_{1}}.
\]
\end{proof}

\subsection{Extended tilting objects in extended hearts }

Let $\D$ be a triangulated category and $\t_{1}\leq\t_{2}\leq\t_{3}$
in $\stors\,\mathcal{D}.$ The aim of this section is to study extended
tilting objects in $\mathcal{H}_{[\t_{1},\t_{3}]}$. We will see that,
under certain conditions, an object $V$ in $\mathcal{H}_{[\t_{1},\t_{3}]}$
is extended tilting if, and only if, it is a projective generator
in $\mathcal{H}_{[\t_{1},\t_{2}]}$. In particular, one could use
this result to determine when the heart of a $t$-structure has a
projective generator. Recall that this is of great interest for several
reasons (see for example \cite{CW,PS2}). In particular, hearts with
a projective generator are related to the theory of silting modules,
see for example \cite{AMV,ABP}. 
\begin{prop}
\label{prop:key-1-1} For a triangulated category $\D$ and $\t_{1}\leq\t_{2}\leq\t_{3}$
in $\stors\,\mathcal{D},$ the following statements hold true for
an extended tilting object $V\in\mathcal{H}_{[\t_{1},\t_{3}]}$. 
\begin{enumerate}
\item Let $\t_{1}\leq\Sigma\t_{2}$. Then, $V\in\mathcal{H}_{[\t_{1},\t_{2}]}$. 
\item Let $\t_{1}\leq\Sigma\t_{2}$. If $[\t_{2},\t_{3}]$ is normal, then
$\mathcal{H}_{[\t_{1},\t_{2}]}$ is a cogenerating class in $\mathcal{H}_{[\t_{1},\t_{3}]}$
and $\Gen_{\mathcal{H}_{[\t_{1},\t_{3}]}}(V)\subseteq\mathcal{H}_{[\t_{1},\t_{2}]}.$ 
\item There is $\t'_{2}\in\stors\,\mathcal{D}$ such that $\mathcal{H}_{[\t_{1},\t'_{2}]}=\Gen_{\mathcal{H}_{[\t_{1},\t_{3}]}}(V)$.
In this case, 
\begin{enumerate}
\item $V$ is a projective generator in $\mathcal{H}_{[\t_{1},\t'_{2}]};$ 
\item if $W$ is an extended tilting in $\mathcal{H}_{[\t_{1},\t_{3}]}$
such that $\Gen_{\mathcal{H}_{[\t_{1},\t_{3}]}}(V)=\Gen_{\mathcal{H}_{[\t_{1},\t_{3}]}}(W)$,
then $\Add(V)=\Add(W)$. 
\end{enumerate}
\item If $\t_{1}\leq\Sigma\t_{2}$ and $\t_{3}=\Sigma^{-1}\t_{2}$, then:
\begin{enumerate}
\item $\mathcal{H}_{[\t_{1},\t_{2}]}=\Gen_{\mathcal{H}_{[\t_{1},\t_{3}]}}(V)=V_{\mathcal{H}_{[\t_{1},\t_{3}]}}^{\bot_{1}}$
and $V$ is a projective generator in $\mathcal{H}_{[\t_{1},\t_{2}]};$ 
\item $\mathcal{H}_{[\t_{1},\Sigma\t_{2}]}=0$, $\mathcal{H}_{[\t_{1},\t_{3}]}=\mathcal{H}_{[\Sigma\t_{2},\Sigma^{-1}\t_{2}]}$,
and $\t_{1}=\Sigma\t_{2}$; 
\item if $W$ is an extended tilting in $\mathcal{H}_{[\t_{1},\t_{3}]}$,
then $\Add(V)=\Add(W)$. 
\end{enumerate}
\end{enumerate}
\end{prop}

\begin{proof}
Set $\mathsf{C}:=\mathcal{H}_{[\t_{1},\t_{3}]}$ and $\mathsf{H}:=\mathcal{H}_{[\t_{1},\t_{2}]}$,
and let $\mathbf{t}=(\mathcal{T},\mathcal{F}):=(\text{Gen}_{\Csf}(V),V_{\Csf}^{\bot_{0}})$
be the associated extended tilting $s$-torsion pair in $\Csf$. Note
that $\Sigma\,\t\leq\t$ for any $\t\in\stors\,\mathcal{D}$, and
that $[\t_{2},\t_{3}]$ is normal if $\t_{3}=\Sigma^{-1}\t_{2}$.

(a) From Lemma \ref{lem:extensions of hearts}, we get an $\s$-conflation
$\suc[A][V][B][][f]$ in $\Csf$ with $A\in\Hsf$ and $B\in\mathcal{H}_{[\t_{2},\t_{3}]}$.
Suppose that $f\neq0$. Then $B\in\text{Gen}_{\Csf}(V)=\mathcal{T}$.
Now, since $B\in\mathcal{H}_{[\t_{2},\t_{3}]}$, we have that $\Sigma B\in\Sigma\mathcal{H}_{[\t_{2},\t_{3}]}=\mathcal{H}_{[\Sigma\t_{2},\Sigma\t_{3}]}\subseteq\Csf$.
Thus, we can consider the canonical $\s$-conflation $\eta:\;\suc[T][\Sigma B][F]$
with respect to the $s$-torsion pair $\mathbf{t},$ where $T\in\mathcal{T}$
and $F\in\mathcal{F}$. By applying the functor $\text{Hom}_{\Csf}(V,-)$
to $\eta,$ we obtain the following exact sequence 
\[
\E^{-1}(V,F)\to\text{Hom}_{\Csf}(V,T)\to\text{Hom}_{\Csf}(V,\Sigma B).
\]
Observe that $\text{Hom}_{\Csf}(V,T)=0$ and $\Sigma B\in\mathcal{F}$
since $\E^{-1}(V,F)=0$ and 
\[
\text{Hom}_{\Csf}(V,\Sigma B)=\text{Hom}_{\mathcal{D}}(V,\Sigma B)=\E_{\Csf}^{1}(V,B)=0,
\]
where the last equality is due to the fact that $B\in\text{Gen}_{\Csf}(V)=V^{\perp_{1}}$
$B\in\text{Gen}_{\Csf}(V)\subseteq V^{\perp_{1}}$. Now, from the
fact that $\Sigma B\in\mathcal{F}$, we obtain 
\[
0=\E^{-1}(V,\Sigma B)=\text{Hom}_{\D}(V,B)=\text{Hom}_{\Csf}(V,B).
\]
But this is a contradiction since $0\neq f\in\text{Hom}_{\Csf}(V,B)$.
Hence we conclude that $f=0.$ Thus, we get a distinguished triangle
$\Sigma^{-1}B\to A\to V\xrightarrow{f}B$ in $\D,$ with $f=0,$ and
then $A=\Sigma^{-1}B\amalg V.$ Therefore $V\in\Hsf$ since $A\in\Hsf.$ 

(b) It is enough to show that $\mathcal{T}=\text{Gen}_{\Csf}(V)\subseteq\Hsf$
since $\text{Gen}_{\Csf}(V)$ is a cogenerating class in $\Csf.$
For this, observe that $V\in\Hsf$ by (a). And thus, by using Lemma
\ref{lem:key-1-1} together with Proposition \ref{lemaA}(a), we can
conclude that $\mathcal{T}=\text{Gen}_{\Csf}(V)\subseteq\Hsf$.

(c) The existence of $\t'_{2}$ follows from Theorem \ref{thm:proceso aet}. 

(c1) $V$ is a projective object of $\Hsf$ since $\T\subseteq V^{\perp_{1}}$.
Let us show that $V$ is a generator of $\mathcal{H}_{[\t_{1},\t'_{2}]}$.
For this, we take a non-zero object $X\in\mathcal{H}_{[\t_{1},\t'_{2}]}$,
and recall that $\mathcal{H}_{[\t_{1},\t'_{2}]}=\mathcal{T}=\text{Gen}_{\Csf}(V)=\text{Pres}_{\Csf}(V)$.
Thus, there is an $\s$-conflation 
\[
\theta:\;\suc[Y][\coprod_{\alpha}{}^{\Csf}\,V][X],
\]
in $\Csf,$ for some set $\alpha\neq\emptyset$ and $Y\in\text{Gen}_{\Csf}(V)=\mathcal{H}_{[\t_{1},\t'_{2}]}$.
From Lemma \ref{lem:key-1-1}, we know that $\coprod{}_{\alpha}^{\Csf}\,V\simeq\coprod{}_{\alpha}^{\Hsf}\,V.$
Hence by the $\s$-conflation $\theta,$ it follows that $X\in\Gen_{\Hsf}(V).$
Therefore $\text{Gen}_{\Hsf}(V)=\mathcal{H}_{[\t_{1},\t'_{2}]}$ as
desired.

(c2) Let us prove that $W\in\Add(V)$. By Theorem \ref{thm:proceso aet},
we have that $\text{Pres}_{\mathcal{H}_{[\t_{1},\t_{3}]}}(V)=\mathcal{H}_{[\t_{1},\t'_{2}]}=\text{Pres}_{\mathcal{H}_{[\t_{1},\t_{3}]}}(W)$.
Hence, there is an $\s$-conflation $\varepsilon:\;\suc[H][\coprod_{\beta}V][W]$
in $\mathcal{H}_{[\t_{1},\t'_{2}]}$. Note that $\varepsilon$ splits
since $\mathcal{H}_{[\t_{1},\t_{2}]}\subseteq W^{\bot_{1}}$, and
thus $W\in\Add(V).$ Similarly, we have that $V\in\Add(W)$ and therefore
$\Add(V)=\Add(W).$

(d1) Firstly, note that we have that $\Hsf=\mathcal{H}_{[\t_{1},\t_{2}]}$
is a cogenerating class in $\Csf=\mathcal{H}_{[\t_{1},\t_{3}]}$ and
$\Gen_{\Csf}(V)\subseteq\Hsf$ by (b). 

Let us prove that $\Hsf\subseteq\Gen_{\Csf}(V)$. For this, consider
$H\in\Hsf$. On the one hand, if $H$ admits a non-zero morphism $f:V\rightarrow H$,
then there is an $\s$-conflation $\suc[X][V][H][][f]$ in $\mathcal{D}$.
Note that $X\in\Sigma^{-1}\Hsf\star\Hsf\subseteq\Csf$. And thus,
$H\in\Gen_{\Csf}(V)$. On the other hand, if $H\in V^{\bot_{0}}$,
then $H=0$. Indeed, by the previous arguments, we have that $V\amalg H\in\Gen_{\Csf}(V)$.
And thus, $H\in\Gen_{\Csf}(V)$ by Proposition \ref{lem:cerraduras stors}(a).
Therefore, $H=0$ since $H\in V_{\Csf}^{\bot_{0}}\cap\Gen_{\Csf}(V)$. 

Let us prove that $\Hsf=V_{\Csf}^{\bot_{1}}$. For this, consider
$Z\in V_{\Csf}^{\bot_{1}}$. Since $\Hsf$ is a cogenerating class
in $\Csf$, there is an $\s$-conflation $\suc[Z][H_{1}][H_{2}]$
in $\Csf$ with $H_{1}\in\Hsf$ and $H_{2}\in\Csf$. Note that $H_{2}\in\Hsf$
by Proposition \ref{lemaA}(a). Then, since $\Hsf=\Gen_{\Csf}(V)=\Pres_{\Csf}(V)$,
there is an $\s$-conflation $\suc[H_{3}][\coprod_{\alpha}V][H_{2}]$
with $H_{3}\in\Hsf$ (for $H_{2}=0$, consider the $\s$-conflation
$\suc[V][V][0][1]$). Then, from \cite[Prop. 3.15]{NP}, we deduce
the following commutative diagram in $\Csf$: 
\[
\xymatrix{\, & H_{3}\ar@{=}[r]\ar[d] & H_{3}\ar[d]\\
Z\ar[r]\ar@{=}[d] & W\ar[r]\ar[d] & \coprod_{\alpha}V\ar[d]\\
Z\ar[r] & H_{1}\ar[r] & H_{2}
}
\]
From the middle column, we can deduce that $W\in\Hsf$. Moreover,
the middle row splits since $Z\in V^{\bot_{1}}$ and 
\[
\mathbb{E}_{\Csf}(\coprod_{\alpha}V,Z)=\Hom[\mathcal{D}][\coprod_{\alpha}V][\Sigma Z]\cong\Hom[\mathcal{D}][V][\Sigma Z]^{\alpha}\cong\mathbb{E}_{\Csf}(V,Z)^{\alpha}=0.
\]
Therefore, $Z\in\Hsf$. 

Lastly, since $\Hsf=\Gen_{\Csf}(V)$, it follows from the previous
item that $V$ is a projective generator of $\Hsf$.

(d2) Consider $X\in\mathcal{H}_{[\Sigma^{-1}\t_{1},\t_{2}]}$. Observe
that $X,\Sigma X\in\mathcal{H}_{[\t_{1},\t_{2}]}=\Hsf$. Then, since
$V$ is projective in $\Hsf$, we have 
\[
0=\mathbb{E}_{\Hsf}(V,X)=\Hom[\mathcal{D}][V][\Sigma X]=\Hom[\Hsf][V][\Sigma X]=\Hom[\Csf][V][\Sigma X].
\]
Therefore, $\Sigma X=0$ since $\Sigma X\in\Hsf=\Gen_{\Csf}(V)$.
And thus, $X=0$. We have proved that $\mathcal{H}_{[\Sigma^{-1}\t_{1},\t_{2}]}=0$.
And thus, it follows from Theorem \ref{thm:proceso aet} that $\mathcal{H}_{[\t_{1},\t_{3}]}=\mathcal{H}_{[\Sigma\t_{2},\Sigma^{-1}\t_{2}]}$
and $\t_{1}=\Sigma\t_{2}$.

(d3) It follows from (d1) that $\mathcal{H}_{[\t_{1},\t_{2}]}=\Gen_{\mathcal{H}_{[\t_{1},\t_{3}]}}(W)$.
And thus, by the previous item, we have that $\Add(V)=\Add(W)$. 
\end{proof}
\begin{thm}
\label{teo:aplication-main-1} Let $\D$ be a triangulated category
and $\t_{1}\leq\t_{2}\leq\t_{3}$ in $\stors\,\mathcal{D},$ be such
that: $\t_{1}\leq\Sigma\t_{2}$ and $\t_{3}=\Sigma^{-1}\t_{2}$. Then,
for $V\in\mathcal{H}_{[\t_{1},\t_{3}]},$ the following statements
are equivalent. 
\begin{enumerate}
\item $V$ is an extended tilting object in $\Hcal_{[\t_{1},\t_{3}]}$.
\item $V$ is an extended tilting object in $\Hcal_{[\t_{1},\t_{3}]}$ such
that $\Gen_{\Hcal_{[\t_{1},\t_{3}]}}(V)=V_{\Hcal_{[\t_{1},\t_{3}]}}^{\bot_{1}}$. 
\item $\mathcal{H}_{[\t_{1},\t_{2}]}$ is a cogenerating class in $\mathcal{H}_{[\t_{1},\t_{3}]}$
and $V$ is a projective generator in $\mathcal{H}_{[\t_{1},\t_{2}]}.$ 
\end{enumerate}
Moreover, in case any of these statements is satisfied, we have that
$\t_{1}=\Sigma\t_{2}$. 

\end{thm}

\begin{proof}
Set the notation $\Csf:=\mathcal{H}_{[\t_{1},\t_{3}]}$ and $\Hsf:=\mathcal{H}_{[\t_{1},\t_{2}]}$.
We also recall that $\t:=(\Hsf,\mathcal{H}_{[\t_{2},\t_{3}]})$ is
an $s$-torsion pair in $\Csf,$ see Lemma \ref{lem:stors de extensiones de corazones}.
Note that $[\t_{2},\t_{3}]$ is normal.

The implication (b) $\Rightarrow$ (c) and the fact that $\t_{1}=\Sigma\t_{2}$
can be obtained from Proposition \ref{prop:key-1-1}. 

Let us show that (c) implies (a). Indeed, let $V$ be a projective
generator in $\Hsf$ which is a cogenerating class in $\Csf.$ Thus
$\Hsf=\Gen_{\Hsf}(V)$ and by Lemma \ref{lem:key-1-1} and Proposition
\ref{lemaA}(a), it follows that $\Hsf=\Gen{}_{\Hsf}(V)=\Gen{}_{\Csf}(V).$
Now, using that $\text{Pres}_{\Hsf}(V)=\text{Gen}_{\Hsf}(V)=\Hsf$
and $\text{Pres}_{\Hsf}(V)\subseteq\text{Pres}_{\Csf}(V)$, we get
that $\text{Pres}_{\Csf}(V)=\text{Gen}_{\Csf}(V)=\Hsf.$\\
 We assert now that $\mathcal{H}_{[\t_{2},\t_{3}]}=V^{\bot_{0}}.$
Indeed, the inclusion $\mathcal{H}_{[\t_{2},\t_{3}]}\subseteq V^{\bot_{0}}$
follows since $V\in\Hsf$ and $\t\in\stors\,\Csf$. Now, let $Z\in V^{\bot_{0}}.$
Consider the $\s$-conflation $\delta_{Z}:\;\suc[H][Z][H']$ in $\Csf$
with $H\in\Hsf$ and $H'\in\mathcal{H}_{[\t_{2},\t_{3}]}$ given by
$\t$. Applying the functor $(V,-):=\text{Hom}_{\D}(V,-)$ in $\delta_{Z},$
we get the exact sequence 
\[
(V,\Sigma^{-1}H')\rightarrow(V,H)\rightarrow(V,Z)
\]
where: $(V,Z)=0$ since $Z\in V^{\bot_{0}}$; and $(V,\Sigma^{-1}H')=0$
since\textbf{ }$V\in\Hsf$, $\Sigma^{-1}H'\in\mathcal{H}_{[\t_{3},\Sigma^{-1}\t_{3}]}\subseteq\mathcal{H}_{[\t_{2},\Sigma^{-1}\t_{3}]}$
and $(\Hsf,\mathcal{H}_{[\t_{2},\Sigma^{-1}\t_{3}]})\in\stors\,\mathcal{H}_{[\t_{1},\Sigma^{-1}\t_{3}]}$
(see Lemma \ref{lem:stors de extensiones de corazones}(c)). Hence
$0=\text{Hom}_{\D}(V,H)=\text{Hom}_{\Hsf}(V,H)$ and thus $H=0$ since
$V$ is a generator in $\Hsf$. This shows that $V^{\bot_{0}}=\mathcal{H}_{[\t_{2},\t_{3}]}$.
Therefore, $(\text{Gen}_{\Csf}(V),V^{\bot_{0}})=(\Hsf,\mathcal{H}_{[\t_{2},\t_{3}]})$
which is an $s$-torsion pair in $\Csf.$ It remains to prove that
$\Gen_{\Csf}(V)\subseteq V^{\bot_{1}}$, but this follows from the
fact that $V$ is a projective object in $\Hsf$ and $\Gen_{\Csf}(V)=\Hsf$. 

Lastly, implication (a) $\Rightarrow$ (b) follows from Proposition
\ref{prop:key-1-1}(d1). 
\end{proof}
Consider Examples \ref{exa:quasi-abelian} and \ref{exa:categorias abelianas asociadas a quasiabeliana}.
Let $\mathcal{E}$ be a quasi-abelian category. We know that there
is a triangulated category $\mathcal{D}$ equipped with $\t_{r},\t_{\ell}\in\stors(\mathcal{D})$
such that $\Sigma\t_{r}\leq\t_{\ell}\leq\t_{r}$ and $\mathcal{E}\cong\mathcal{H}_{[\Sigma\t_{r},\t_{\ell}]}$.
As an application of our results, we can give conditions for $\mathcal{E}$
to be abelian. 
\begin{cor}
Let $\mathcal{E}$ be a quasi-abelian category. Then, the following
statements are equivalent for an object $V\in\mathcal{H}_{[\Sigma\t_{\ell},\Sigma^{-1}\t_{r}]}$. 
\begin{enumerate}
\item $V$ is an extended tilting object in $\mathcal{H}_{[\Sigma\t_{\ell},\Sigma^{-1}\t_{r}]}$. 
\item $\mathcal{H}_{[\Sigma\t_{\ell},\t_{r}]}$ is a cogenerating class
in $\mathcal{H}_{[\Sigma\t_{\ell},\Sigma^{-1}\t_{r}]}$ and $V$ is
a projective generator in $\mathcal{H}_{[\Sigma\t_{\ell},\t_{r}]}$. 
\end{enumerate}
Moreover, in case any of these statements is satisfied, we have that
$\mathcal{E}$ is an abelian category.
\end{cor}

\begin{proof}
Consider $\t_{1}:=\Sigma\t_{\ell}$, $\t_{2}:=\t_{r}$ and $\t_{3}:=\Sigma^{-1}\t_{r}$.
Observe that $\t_{1}\leq\Sigma\t_{2}$ and that $\t_{3}=\Sigma^{-1}\t_{2}$.
Hence, by Theorem \ref{teo:aplication-main-1}, the assertion follows. 
\end{proof}
\begin{cor}
\label{apl1-mainT} For a triangulated category $\D,$ and a $t$-structure
$\u=(\U,\W)$ in $\D$ with heart $\Hcal$ and extended heart $\C,$
the following statements hold true. 
\begin{enumerate}
\item $(\Hcal,\Sigma^{-1}\Hcal)\in\stors\,\C$ and the extriangulated category
$\C$ is not abelian if $\Hcal\neq0.$ 
\item For any $V\in\C,$ the following statements are equivalent.
\begin{enumerate}
\item $V$ is an extended tilting object in $\C$. 
\item $V$ is an extended tilting object in $\C$ such that $\Gen_{\mathcal{C}}(V)=V_{\mathcal{C}}^{\bot_{1}}$. 
\item $\Hcal$ is a cogenerating class in $\C$ and $V$ is a projective
generator in $\Hcal.$ 
\end{enumerate}
\item $\Hcal$ is a cogenerating class in $\C$ if, and only if, $\C=\Sigma^{-1}\Hcal*\Hcal.$ 
\end{enumerate}
\end{cor}

\begin{proof}
Consider $\u_{1}:=(\Sigma\U,\W),$ $\u_{2}:=(\U,\Sigma^{-1}\W)$ and
$\u_{3}:=(\Sigma^{-1}\U,\Sigma^{-2}\W).$ Then we have that $\u_{1}\leq\u_{2}\leq\u_{3}$
in $\stors\,\D.$ Thus, the item (a) follows from Lemma \ref{lem:stors de extensiones de corazones}
(c) and Proposition \ref{basic-p-C} (d); and the item (b) can be
obtained from Theorem \ref{teo:aplication-main-1}. \
 Let us show (c). Indeed, it is clear that $\Hcal$ is a cogenerating
class in $\C$, whenever $\C=\Sigma^{-1}\Hcal*\Hcal.$ Assume now
that $\Hcal$ is a cogenerating class in $\C.$ Then, by Proposition
\ref{lemaA} (a), we get that $\C=\Sigma^{-1}\Hcal*\Hcal.$ 
\end{proof}
In what follows we will see some situations where Theorem \ref{teo:aplication-main-1}
can be applied. 

Consider the triangulated category $\D(\mathcal{A})$ which is the
unbounded derived category of an abelian category $\mathcal{A}$.
For a torsion pair $\t=(\T,\F)$ in $\mathcal{A},$ we have (see \cite[Prop. 2.1]{HRS})
the Happel-Reiten-Smal{\o} $t$-structure $\u(\t):=(\U_{\t},\W_{t})$
in $\D(\mathcal{A}),$ where 
\begin{alignat*}{1}
\U_{\t} & :=\{X\in\D(\mathcal{A})\;:\;H^{i}(X)=0\;\text{ for }i>0,\;H^{0}(X)\in\T\}\text{ and }\\
\W_{\t} & :=\{X\in\D(\mathcal{A})\;:\;H^{i}(X)=0\;\text{ for }i<-1,\;H^{-1}(X)\in\F\}.
\end{alignat*}
 Let $\mathcal{H}_{\t}$ be the heart of the Happel-Reiten-Smal{\o} 
$t$-structure $\u(\t)=(\mathcal{U}_{\mathbf{t}},\mathcal{W}_{\mathbf{t}})$.
It is well-known that there is a triangulated functor $G:\mathcal{D}^{b}(\mathcal{H}_{\t})\rightarrow\mathcal{D}^{b}(\mathcal{A})$
whose restriction on $\mathcal{H}_{\t}$ coincides with the natural
inclusion. In \cite[Thm. A]{CHZ}, it is proved that $G$ is an equivalence
if and only if every $A\in\mathcal{A}$ admits an exact sequence 
\[
0\rightarrow F_{0}\rightarrow F_{1}\rightarrow A\rightarrow T_{0}\rightarrow T_{1}\rightarrow0
\]
with $F_{0},F_{1}\in\mathcal{F}$ and $T_{0},T_{1}\in\mathcal{T}$
such that the corresponding class in $\mathbb{E}_{\mathcal{A}}^{3}(T_{1},F_{0})$
vanishes. Observe that the existence of such exact sequences implies
that 
\begin{equation}
\mathcal{A}\subseteq\mathcal{F}\star\Sigma\mathcal{F}\star\Sigma^{-1}\mathcal{T}\star\mathcal{T}\label{eq:chen}
\end{equation}
 in $\mathcal{D}(\mathcal{A})$. In the following corollary we will
show a situation where this condition implies that $\C_{\t}=\Sigma^{-1}\Hcal_{\t}*\Hcal_{\t}$
(see Corollary \ref{apl1-mainT}(c)). 
\begin{rem}
\label{remap}Note that the condition in (\ref{eq:chen}) also appears
in many contexts as the following ones (see \cite{CHZ} for further
details). 
\begin{enumerate}
\item $\mathcal{A}=\mathbf{Fac}(\mathcal{F})\star\mathcal{T}$, where $\mathbf{Fac}(\mathcal{F})$
is the class of objects $A$ in $\mathcal{A}$ admitting a short exact
sequence $0\rightarrow\suc[A'][F][A]\rightarrow0$ with $F\in\mathcal{F}$.
This context includes the case when $\mathcal{F}$ is a generating
class of $\mathcal{A}$.
\item $\mathcal{A}=\mathcal{F}\star\mathbf{Sub}(\mathcal{T})$, where $\mathbf{Sub}(\mathcal{T})$
is the class of objects $A$ in $\mathcal{A}$ admitting a short exact
sequence $0\rightarrow\suc[A][T][A']\rightarrow0$ with $T\in\mathcal{T}$.
This context includes the case when $\mathcal{T}$ is a cogenerating
class of $\mathcal{A}$.
\item $\mathcal{A}=\mathcal{F}\star\mathcal{T}$, which includes the case
when $(\mathcal{T},\mathcal{F})$ is a \textbf{splitting} torsion
pair (i.e. $\mathbb{E}(\mathcal{F},\mathcal{T})=0$). 
\end{enumerate}
\end{rem}

\begin{cor}
\label{cor:remap}\label{cor:main}\label{prop:key-2-1}Let $\mathcal{A}$
be a abelian category and $\mathbf{t}=(\mathcal{T},\mathcal{F})$
be a torsion pair in $\mathcal{A}$. Suppose that one of the following
conditions holds:
\begin{enumerate}
\item [(i)] $\Hom[\mathcal{D}(\mathcal{A})][\mathcal{F}][\Sigma^{2}\mathcal{F}]=0\text{, }\Hom[\mathcal{D}(\mathcal{A})][\mathcal{T}][\Sigma^{2}\mathcal{T}]=0\text{, }\Hom[\mathcal{D}(\mathcal{A})][\mathcal{T}][\Sigma^{3}\mathcal{F}]=0$,
and $\mathcal{A}\subseteq\mathcal{F}\star\Sigma\mathcal{F}\star\Sigma^{-1}\mathcal{T}\star\mathcal{T}$
in $\mathcal{D}(\mathcal{A})$;
\item [(ii)] $\mathcal{A}$ is hereditary and $(\mathcal{T},\mathcal{F})$
is splitting; 
\item [(iii)] $\mathcal{A}$ has enough injectives, $\Hom[\mathcal{D}(\mathcal{A})][\mathcal{T}][\Sigma^{3}\mathcal{F}]=0$,
and $\mathcal{T}$ is a cogenerating class in $\mathcal{A}$; 
\item [(iii)'] $\mathcal{A}$ has enough projectives, $\Hom[\mathcal{D}(\mathcal{A})][\mathcal{T}][\Sigma^{3}\mathcal{F}]=0$,
and $\mathcal{F}$ is a generating class in $\mathcal{A}$.
\end{enumerate}
Then, the heart $\Hcal_{\mathbf{t}}$ and the extended heart $\C_{\mathbf{t}}$
of the Happel-Reiten-Smal{\o} $t$-structure $\u(\t)=(\mathcal{U}_{\mathbf{t}},\mathcal{W}_{\mathbf{t}})$
satisfy that $\C_{\t}=\Sigma^{-1}\Hcal_{\t}*\Hcal_{\t}$ and, for
any $V\in\C_{\t},$ the following statements are equivalent:
\begin{enumerate}
\item $V$ is an extended tilting object in $\C_{\t}$; 
\item $V$ is a projective generator in $\Hcal_{\t}.$ 
\end{enumerate}
\end{cor}

\begin{proof}
By Corollary \ref{apl1-mainT}, it is enough to show that $\C_{\t}\subseteq\Sigma^{-1}\Hcal_{\t}*\Hcal_{\t}.$ 

Assume that (i) holds. The condition $\Hom[\mathcal{D}(\mathcal{A})][\mathcal{F}][\Sigma^{2}\mathcal{F}]=0$
implies that every $\s$-conflation $\suc[\Sigma F][X][F']$ with
$F,F'\in\mathcal{F}$ splits. And thus, $\Sigma\mathcal{F}\star\mathcal{F}\subseteq\mathcal{F}\star\Sigma\mathcal{F}$.
Similarly, conditions $\Hom[\mathcal{D}(\mathcal{A})][\mathcal{T}][\Sigma^{2}\mathcal{T}]=0$
and $\Hom[\mathcal{D}(\mathcal{A})][\mathcal{T}][\Sigma^{3}\mathcal{F}]=0$
imply that $\Sigma\mathcal{T}\star\mathcal{T}\subseteq\mathcal{T}\star\Sigma\mathcal{T}$
and that $\Sigma\mathcal{F}\star\Sigma^{-1}\mathcal{T}\subseteq\Sigma^{-1}\mathcal{T}\star\Sigma\mathcal{F}$,
respectively. Therefore, since $\mathcal{H}_{\t}=\Sigma\mathcal{F}\star\mathcal{T}$,
we have that 
\begin{alignat*}{1}
\mathcal{C}_{\t} & =\mathcal{H}_{\t}\star\Sigma^{-1}\mathcal{H}_{\t}\\
 & =\Sigma\mathcal{F}\star\mathcal{T}\star\mathcal{F}\star\Sigma^{-1}\mathcal{T}\\
 & =\Sigma\mathcal{F}\star\mathcal{A}\star\Sigma^{-1}\mathcal{T}\\
 & \subseteq\Sigma\mathcal{F}\star(\mathcal{F}\star\Sigma\mathcal{F}\star\Sigma^{-1}\mathcal{T}\star\mathcal{T})\star\Sigma^{-1}\mathcal{T}\\
 & =(\Sigma\mathcal{F}\star\mathcal{F})\star\Sigma\mathcal{F}\star\Sigma^{-1}\mathcal{T}\star(\mathcal{T}\star\Sigma^{-1}\mathcal{T})\\
 & \subseteq\mathcal{F}\star(\Sigma\mathcal{F}\star\Sigma\mathcal{F})\star(\Sigma^{-1}\mathcal{T}\star\Sigma^{-1}\mathcal{T})\star\mathcal{T}\\
 & \subseteq\mathcal{F}\star(\Sigma\mathcal{F}\star\Sigma^{-1}\mathcal{T})\star\mathcal{T}\\
 & \subseteq\mathcal{F}\star\Sigma^{-1}\mathcal{T}\star\Sigma\mathcal{F}\star\mathcal{T}\\
 & =\Sigma^{-1}\mathcal{H}_{\t}\star\mathcal{H}_{\t}.
\end{alignat*}

Assume that (ii) holds. Observe that $\mathcal{A}\subseteq\mathcal{F}\star\Sigma\mathcal{F}\star\Sigma^{-1}\mathcal{T}\star\mathcal{T}$
by Remark \ref{remap}(c). Moreover, since $\mathcal{A}$ is hereditary,
it follows that $\mathbf{Ext}_{\mathcal{A}}^{n}(-,-)=0$ for all $n>1$.
Hence, condition (i) holds. And thus, $\C_{\t}\subseteq\Sigma^{-1}\Hcal_{\t}*\Hcal_{\t}.$

Assumme that (iii) holds. On the one hand, $\mathcal{T}$ cogenerating
implies that $\Inj[\mathcal{A}]\subseteq\mathcal{T}$. And thus, since
$\mathcal{A}$ has enough injectives and $\mathcal{T}$ is closed
under quotients, we have that $\mathcal{A}\subseteq\Sigma^{-1}\mathcal{T}\star\Inj[\mathcal{A}]\subseteq\Sigma^{-1}\mathcal{T}\star\mathcal{T}$
in $\mathcal{D}(\mathcal{A})$. Here, is important to note that $\Inj[\mathcal{A}]\star\Sigma^{-1}\mathcal{T}\subseteq\Sigma^{-1}\mathcal{T}\star\Inj[\mathcal{A}]$
because 
\[
\Hom[\mathcal{D}(\mathcal{A})][\Sigma^{-1}\mathcal{T}][{\Sigma\Inj[\mathcal{A}]}]\cong\Hom[\mathcal{D}(\mathcal{A})][\mathcal{T}][\Sigma^{2}\Inj[\mathcal{A}]]\cong\mathbf{Ext}_{\mathcal{A}}^{2}(\mathcal{T},\Inj[\mathcal{A}])=0.
\]
On the other hand, the condition $\Hom[\mathcal{D}(\mathcal{A})][\mathcal{T}][\Sigma^{3}\mathcal{F}]=0$
implies that $\Sigma\mathcal{F}\star\Sigma^{-1}\mathcal{T}\subseteq\Sigma^{-1}\mathcal{T}\star\Sigma\mathcal{F}$.
Therefore, since $\mathcal{H}_{\t}=\Sigma\mathcal{F}\star\mathcal{T}$,
we have that 
\begin{alignat*}{1}
\mathcal{C}_{\t} & =\mathcal{H}_{\t}\star\Sigma^{-1}\mathcal{H}_{\t}\\
 & =\Sigma\mathcal{F}\star\mathcal{T}\star\mathcal{F}\star\Sigma^{-1}\mathcal{T}\\
 & =\Sigma\mathcal{F}\star\mathcal{A}\star\Sigma^{-1}\mathcal{T}\\
 & \subseteq\Sigma\mathcal{F}\star(\Sigma^{-1}\mathcal{T}\star\Inj[\mathcal{A}])\star\Sigma^{-1}\mathcal{T}\\
 & =(\Sigma\mathcal{F}\star\Sigma^{-1}\mathcal{T})\star(\Inj[\mathcal{A}]\star\Sigma^{-1}\mathcal{T})\\
 & \subseteq(\Sigma\mathcal{F}\star\Sigma^{-1}\mathcal{T})\star(\Sigma^{-1}\mathcal{T}\star\Inj[\mathcal{A}])\\
 & =\Sigma\mathcal{F}\star\Sigma^{-1}\mathcal{T}\star\Inj[\mathcal{A}]\\
 & \subseteq(\Sigma\mathcal{F}\star\Sigma^{-1}\mathcal{T})\star\mathcal{T}\\
 & \subseteq(\Sigma^{-1}\mathcal{T}\star\Sigma\mathcal{F})\star\mathcal{T}\\
 & \subseteq\Sigma^{-1}\mathcal{H}_{\t}\star\mathcal{H}_{\t}.
\end{alignat*}

Laslty, if (iii)' holds, then it follows from dual arguments as above
that $\C_{\t}\subseteq\Sigma^{-1}\Hcal_{\t}*\Hcal_{\t}$. 
\end{proof}
\begin{example}
\label{ejemplo} Let $K$ be a field and $R:=K\times K$. Observe
that every torsion pair in $\text{Mod-}R$ splits since $R$ is semi-simple.
Furthermore, $R$ is also a non trivial hereditary ring. Note that
$I:=K\times0$ is a finitely generated idempotent two-sided ideal
of $R$. Consider the torsion pair $\mathbf{t}:=(\text{Mod-}R/I,R/I^{\bot_{0}})$
in $\text{Mod-}R,$ and let $\Hcal_{\mathbf{t}}$ be the heart and
$\C_{\t}$ be the extended heart of the Happel-Reiten-Smal{\o} $t$-structure
$\u(\t).$ Then, by \cite[Ex. 5.8]{PS2}, we know that $\Hcal_{\mathbf{t}}$
has a projective generator $V.$ Thus, from Corollary \ref{prop:key-2-1}(ii),
we conclude that $V$ is an extended tilting object in the extriangulated
category $\C_{\mathbf{t}}$ which is not abelian. 
\end{example}

\begin{example}
\label{exa:ult}Let $\mathcal{A}$ be the abelian category obtained
as the HRS-tilt of $\Mod[\mathbb{Z}]$ associated to the torsion pair
$(\mathcal{T},\mathcal{F})$, where $\mathcal{T}$ is the class of
torsion abelian groups (i.e. $\mathcal{A}:=\Sigma\mathcal{F}\star\mathcal{T}$
in $\mathcal{D}(\mathbb{Z})$). It was proved in \cite[Cor. 4.3(b)]{CGM}
that $\mathcal{A}$ is a non-hereditary Grothendieck category and
that $\t:=(\Sigma\mathcal{F},\mathcal{T})$ is a  torsion pair with
$\Sigma\mathcal{F}$ cogenerating in $\mathcal{A}$. Moreover, if
$S$ denotes the shift functor of the derived category $\mathcal{D}(\mathcal{A})$,
we have that $\Hom[\mathcal{D}(\mathcal{A})][\Sigma\mathcal{\mathcal{F}}][S^{3}\mathcal{T}]=0$.
Indeed, every $M\in\Mod[\mathbb{Z}]$ admits a short exact sequence
$0\rightarrow\suc[F'][F][M]\rightarrow0$ with $F,F'\in\mathcal{F}$.
Then, by \cite[Thm. 3.4]{CHZ}, it follows that $\Ext[X][2][\mathcal{A}][Y]\cong\Hom[\mathcal{D}^{b}(\mathbb{Z})][X][\Sigma^{2}Y]$
for any $X,Y\in\mathcal{A}$. Hence, by \cite[Rem. 3.1.17]{BBD},
there is a monomorphism $\Hom[\mathcal{D}(\mathcal{A})][\Sigma\mathcal{\mathcal{F}}][S^{3}\mathcal{T}]\rightarrow\Hom[\mathcal{D}(\mathbb{Z})][\Sigma\mathcal{\mathcal{F}}][\Sigma^{3}\mathcal{T}].$
And thus, since $\Hom[\mathcal{D}(\mathbb{Z})][\Sigma\mathcal{\mathcal{F}}][\Sigma^{3}\mathcal{T}]\cong\Hom[\mathcal{D}(\mathbb{Z})][\mathcal{\mathcal{F}}][\Sigma^{2}\mathcal{T}]=0$,
we have that $\Hom[\mathcal{D}(\mathcal{A})][\Sigma\mathcal{\mathcal{F}}][S^{3}\mathcal{T}]=0$.
Therefore, it follows from Corollary \ref{cor:main}(iii), that $\Sigma\mathbb{Z}$
is an extended tilting object of the extended heart $\mathcal{C}_{\t}$
since it is a projective generator of $\mathcal{H}_{\t}$. 
\end{example}

\section*{Acknowledgments }

The first named author was supported by a postdoctoral fellowship
EPM(1) 2024 from SECIHTI. The first and second named authors were
supported by the Project PAPIIT-IN100124 Universidad Nacional Aut{\'o}noma
de M{\'e}xico. The third named author was supported by ANID+FONDECYT/REGULAR+1240253. 

{\footnotesize{}{}\vskip3mm Alejandro Argud{\'i}n-Monroy}\\
 {\footnotesize{}{} Instituto de Matem{\'a}ticas}\\
 {\footnotesize{}{} Universidad Nacional Aut{\'o}noma de M{\'e}xico,}\\
 {\footnotesize{}{} Circuito Exterior, Ciudad Universitaria,}\\
 {\footnotesize{}{} CDMX 04510, M{\'E}XICO.}\\
 {\footnotesize{}{} }\texttt{\footnotesize{}{}argudin@ciencias.unam.mx}{\footnotesize\par}

\noindent {\footnotesize{}{}\vskip3mm \noindent  Octavio Mendoza
Hern{\'a}ndez}\\
 {\footnotesize{}{} Instituto de Matem{\'a}ticas}\\
 {\footnotesize{}{} Universidad Nacional Aut{\'o}noma de M{\'e}xico,}\\
 {\footnotesize{}{} Circuito Exterior, Ciudad Universitaria,}\\
 {\footnotesize{}{} CDMX 04510, M{\'E}XICO.}\\
 {\footnotesize{}{} }\texttt{\footnotesize{}{}omendoza@matem.unam.mx}{\footnotesize\par}

\noindent {\footnotesize{}{}\vskip3mm \noindent  Carlos E. Parra}\\
 {\footnotesize{}{} Instituto de Ciencias F{\'i}sicas y Matem{\'a}ticas}\\
 {\footnotesize{}{} Edificio Emilio Pugin, Campus Isla Teja}\\
 {\footnotesize{}{} Universidad Austral de Chile}\\
 {\footnotesize{}{} 5090000 Valdivia, CHILE}\\
 {\footnotesize{}{} }\texttt{\footnotesize{}{}carlos.parra@uach.cl}{\footnotesize\par}
\end{document}